\documentclass[12pt]{article}

\usepackage{amsmath,amssymb,amsthm}

\usepackage{pictex}

\newcommand{\ssize}{\scriptstyle} 

\DeclareMathOperator{\rAut}{r-Aut}
\DeclareMathOperator{\Aut}{Aut}
\DeclareMathOperator{\Ext}{Ext}
\DeclareMathOperator{\End}{End}
\DeclareMathOperator{\rad}{rad}
\DeclareMathOperator{\add}{add}
\DeclareMathOperator{\Tr}{Tr}
\DeclareMathOperator{\mo}{mod}
\DeclareMathOperator{\Hom}{Hom}
\DeclareMathOperator{\bdim}{\mathbf{dim}}
\DeclareMathOperator{\btype}{\mathbf{type}}

\DeclareMathOperator{\Ker}{Ker}
\DeclareMathOperator{\Cok}{Cok}
\DeclareMathOperator{\soc}{soc}
\DeclareMathOperator{\tp}{top}
\DeclareMathOperator{\Img}{Im}

\newtheorem{Thm}{Theorem}[section]
\newtheorem{Lem}[Thm]{Lemma}
\newtheorem{Cor}[Thm]{Corollary}
\newtheorem{Prop}[Thm]{Proposition}

\newcommand{\Rahmen}[1]%
   {$$\vbox{\hrule\hbox%
                  {\vrule%
                       \hskip0.5cm%
                            \vbox{\vskip0.3cm\relax%
                               \hbox{$\displaystyle{#1}$}%
                                  \vskip0.3cm}%
                       \hskip0.5cm%
                  \vrule}%
           \hrule}$$}
\newcommand{\arr}[2]
  {\arrow <1.5mm> [0.25,0.75] from #1 to #2}

\begin{document}
\title{
The Auslander bijections:\\
How morphisms are determined by modules.}
\author{Claus Michael Ringel \\
\small Department of Mathematics, Shanghai Jiao Tong University \\
\small Shanghai 200240, P. R. China, and \\
\small King Abdulaziz University, P O Box 80200\\ 
\small Jeddah, Saudi Arabia\\
\texttt{\small ringel@math.uni-bielefeld.de}}\date{}
\maketitle

\begin{abstract} 
Let $\Lambda$ be an artin algebra. In his seminal Philadelphia Notes
published in 1978, 
M.~Auslander introduced the concept of morphisms being determined by modules.
Auslander was very passionate about these investigations (they also form part of the final chapter of the Auslander-Reiten-Smal\o{} book and could and should be seen as its culmination). The theory presented by Auslander has to be considered as an exciting frame for working with the category of $\Lambda$-modules, incorporating all what is known about irreducible maps (the usual Auslander-Reiten theory), but the frame is much wider and allows for example to take into account families of modules --- an important feature of module categories. What Auslander has achieved 
is a clear description of the poset structure of the category of $\Lambda$-modules 
as well as a blueprint for interrelating individual modules and families of modules. Auslander has subsumed his considerations under the heading of ``morphisms being determined by modules''. Unfortunately, the wording in itself seems to be somewhat misleading, and the basic definition may look quite technical and unattractive, at least at first sight. This could be the reason that for over 30 years, Auslander's powerful results did not gain the attention they deserve. The aim of this survey is to outline the general setting for Auslander's ideas and to show the wealth of these ideas by exhibiting 
many examples.
\end{abstract}

\section{Introduction}
\label{sec:1}
There are 
two basic mathematical structures: groups and lattices, or, more generally,
semigroups and posets. 
A first glance at any category should focus the attention on these two structures: 
to symmetry groups (for example the automorphism groups of the individual objects), 
as well as to the posets 
given by suitable sets of morphisms, for example by looking at inclusion maps (thus
dealing with the poset of all subobjects of an object), or at the possible
factorizations of morphisms. In this way, one distinguishes
between local symmetries and global directedness.

The present survey deals with the category $\mo\Lambda$ of finite length modules
over an artin algebra $\Lambda$. Its aim is to report on the work of
M. Auslander in his seminal Philadelphia Notes published in 1978. 
Auslander was very passionate about these investigations and they also form part of
the final chapter of the Auslander-Reiten-Smal\o{} book: there, they could (and should) be seen as a kind of culmination. 
It seems to be surprising that the feedback until now is quite meager.
After all, the theory presented by Auslander has to be considered as 
an exciting frame for working with the category $\mo\Lambda$,
incorporating what is called the Auslander-Reiten theory 
(to deal with the irreducible maps),
but this frame is much wider and allows for example to take into account 
families of modules --- an important feature of a module category.
Indeed, many of the concepts which are relevant when considering the categories 
$\mo\Lambda$ fit into the frame!
What Auslander has achieved (but he himself may not have realized it) was
a clear description of the poset structure of $\mo \Lambda$ and of the
interplay between families of modules. 

Auslander's considerations are subsumed under the heading of morphisms being
determined by modules, but the wording in itself seems to be somewhat misleading,
and the basic definition looks quite technical and unattractive, 
at least at first sight. This could be the reason that for over 30 years, Auslander's powerful 
results did not gain the attention they deserve. 
	\medskip

Here is a short summary: Let $\Lambda$ be an artin algebra. The modules which we
consider will be left $\Lambda$-modules of finite length, and maps (or morphisms)
will be $\Lambda$-module homomorphisms, unless otherwise specified.
Auslander asks for a description
of the class of maps ending in a fixed module $Y$. 
Two maps
$f: X \to Y$ and $f': X' \to Y$ are said to be {\it right equivalent} provided
there are maps $h: X \to X'$ and $h': X'\to X$ such that $f = f'h$ and $f'= fh'$.
The right equivalence class of $f$ will be denoted by $[f\rangle.$
The object studied
by Auslander is 
 the set of right equivalence classes of maps ending in $Y$, we denote this set 
by 
$$
 [\to Y\rangle.
$$
It is a poset via the relation $\le$ which is defined as follows:
$$
 [f: X \to Y\rangle \ \le \ [f': X' \to Y\rangle
$$ 
provided there is a homomorphism $h: X \to X'$ with $f = f'h,$ thus provided
the following diagram commutes: 
$$
{\beginpicture
\setcoordinatesystem units <1cm,1cm>
\put{$X$} at 0 0 
\put{$X'$} at 1 -.9
\put{$Y$} at 3 0
\put{$f$} at 1.5 0.2
\put{$f'$} at 2.1 -0.75
\put{$h$} at 0.25 -0.6
\arrow <1.5mm> [0.25,0.75] from 0.3 0 to 2.7 0
\arrow <1.5mm> [0.25,0.75] from 0.25 -.25 to 0.75 -.8
\arrow <1.5mm> [0.25,0.75] from 1.4 -.8 to 2.7 -.2
\endpicture}
$$
It is easy to see that the poset $[\to Y\rangle$ is a lattice, thus we call it 
the {\it right factorization lattice} for $Y$.

Looking at maps $f: X \to Y$,
we may (and often will) assume that $f$ is {\it right minimal}, thus that there is no 
non-zero direct summand $X'$ of $X$ with $f(X') = 0.$ Note that any right equivalence
class contains a right minimal map, and if $f: X \to Y$ and
$f': X' \to Y$ are right minimal maps, then any $h: X \to X'$ with $f = f'h$
has to be an isomorphism. 

Of course,
to analyze the poset $[\to Y\rangle$ is strongly related to a study of the
contravariant $\Hom$-functor $\Hom(-,Y)$, however the different nature of these
two mathematical structures should be stressed: $\Hom(-,Y)$ is an additive
functor whereas $[\to Y\rangle$ is a poset, and it is 
the collection of these posets $[\to Y\rangle$
which demonstrates the global directedness.  
	\medskip 

In general, the right factorization lattice $[\to Y\rangle$ 
is very large and does not satisfy
any chain condition. The main idea of Auslander is 
to write $[\to Y\rangle$ as the filtered union of the subsets
${}^C[\to Y\rangle$, where ${}^C[\to Y\rangle$ is given by those maps $f$
which are ``right $C$-determined''. These posets are again lattices and they are of finite
height, we call ${}^C[\to Y\rangle$ the {\it right $C$-factorization lattice} for $Y$. 
Since the concept of ``right determination'' looks
(at least at first sight) technical and unattractive, let us first describe the set
${}^C[\to Y\rangle$ only in the important case when $C$ is a generator: in this
case, ${}^C[\to Y\rangle$ consists of the (right equivalence classes of the
right minimal) maps $f$ ending in $Y$ with kernel in $\add\tau C$ (we denote by
$\tau = D\Tr$ and $\tau^{-} = \Tr D$
the Auslander-Reiten translations). Here is Auslander's first main assertion:
$$
 (1)\qquad  [\to Y\rangle = \bigcup\nolimits_C {}^C[\to Y\rangle, 
$$
where $C$ runs through all isomorphism classes of $\Lambda$-modules;
or, in the formulation of Auslander: any map in $\mo\Lambda$ is right
determined by some module $C$.  Note that the inclusion of the lattice 
${}^C[\to Y\rangle$ into the lattice $[\to Y\rangle$ preserves meets, but
usually not joins.

Auslander's second main assertion describes the right $C$-factorization 
lattice ${}^C[\to Y\rangle$ as follows:
There is a lattice isomorphism 
$$
 (2)\qquad \eta_{CY}: {}^C[\to Y\rangle \longrightarrow \mathcal S\Hom(C,Y), 
$$
where $\Hom(C,Y)$ is considered as an $\End(C)^{\text{op}}$-module, and where $\mathcal SM$
denotes the submodule lattice of a module $M$. Actually, the map
$\eta_{CY}$ is easy to describe, namely $\eta_{CY}(f) = \Img \Hom(C,f)$
for $f$ a morphism ending in $Y$. The essential assertion is the surjectivity
of $\eta_{CY}$, thus to say that any submodule of $\Hom(C,Y)$ is of the form
$\Img\Hom(C,f)$ for some $f$. 
 
What is the relevance? 
As we have mentioned, usually the lattice $[\to Y\rangle$ 
itself will not satisfy any chain conditions, but all the lattices ${}^C[\to Y\rangle$
are of finite height and often can be displayed very nicely: according to
(2) we deal with the submodule lattice $\mathcal SM$ of some finite length module
$M$ over an artin algebra (namely over $\Gamma(C) = \End(C)^{\text{op}}$) and it is easy
to see that any submodule lattice arises in this way.
Using the Auslander bijections $\eta_{CY}$,
one may transfer properties of submodule lattices to the right $C$-factorization
lattices ${}^C[\to Y\rangle,$ this will be one of the aims of this paper. Given a
submodule $U$ of $\Hom(C,Y)$, let $f$ be a right $C$-determined map ending in $Y$ 
such that 
$\eta_{CY}(f) = U$. The composition series of the factor module $\Hom(C,Y)/U$
correspond to certain factorizations of $f$ (to the ``maximal $C$-factorizations''),
and we may define the $C$-type of $f$ so that it is equal to the 
dimension vector of the module $\Hom(C,Y)/U$ (recall that the dimension
vector of a module $M$ has as coefficients the Jordan-H\"older multiplicities of the
various simple modules occurring in $M$).

Submodule lattices have interesting combinatorial features, and it seems to 
interesting that Auslander himself looked mainly at
combinatorial properties (for example at waists in submodule lattices). 
But we should stress that we really are in the realm of algebraic geometry.
Thus, let us assume for a moment
that $\Lambda$ is a $k$-algebra where $k$ is an algebraically
closed field. If $M$ is a finite-dimensional $\Lambda$-module,
the set $\mathcal SM$ of all submodules of $M$ is the disjoint union of the 
sets $\mathbb G_{\mathbf e}(M)$ consisting of all
submodules of $M$ with fixed dimension vector $\mathbf e$. 
It is well-known that $\mathbb G_{\mathbf e}(M)$
is in a natural way a projective variety, called nowadays a {\it quiver Grassmannian}. 
Given $\Lambda$-modules $C$ and $Y$, the Auslander bijections draw the
attention on the $\End(C)^{\text{op}}$-module $M = \Hom(C,Y)$, let $\mathbf d$ be its dimension vector and 
let $\mathbf e, \mathbf e'$ be dimension vectors with $\mathbf e+\mathbf e' = \mathbf d.$
The quiver Grassmannians 
$\mathbb G_{\mathbf e'}(\Hom(C,Y))$ corresponds under the Auslander bijection $\eta_{CY}$
to the set ${}^C[\to Y\rangle^{\mathbf e}$ of
all right equivalence classes of right $C$-determined maps which end in $Y$ and have
type $\mathbf e$. We call ${}^C[\to Y\rangle^{\mathbf e}$ an {\it Auslander variety.} 
These Auslander varieties have to be considered as an important tool for studying the right
equivalence classes of maps ending in a given module.
	\medskip 

We end this summary by an outline in which way the Auslander bijections (2)
incorporate the existence of minimal right
almost split maps: We have to look at the special case where $Y$ is indecomposable
and $C = Y$ and to deal with the submodule $\rad(Y,Y)$ of $\Hom(Y,Y)$. 
The bijection (2) yields an element $f: X\to Y$
in ${}^Y[\to Y\rangle$ such that $\eta_{YY}(f) = \rad(Y,Y)$; to 
say that $f$ is right $Y$-determined means that $f$ is right
almost split. 
	\bigskip\bigskip\bigskip 

\centerline{\huge \bf I. The setting.}
	
\section{The right factorization lattice $[\to Y\rangle$.}
\label{sec:2}

Let $Y$ be a $\Lambda$-module. 
Let $\bigsqcup_X \Hom(X,Y)$ be the class of all homomorphisms $f: X \to Y$ with
arbitrary modules $X$ (such homomorphisms will be said to be the homomorphisms 
{\it ending in} $Y$).
We define a preorder $\preceq$ on this class as follows: Given $f: X \to Y$
and $f': X' \to Y$, we write $f \preceq f'$ provided there is a homomorphism $h: X \to X'$
such that $f = f'h$ (clearly, this relation is reflexive and transitive). As usual,
such a preorder defines an equivalence relation (in our setting, we call it {\it right equivalence})
by saying that $f,f'$ are {\it right equivalent} provided we have both $f \preceq f'$ and
$f'\preceq f$, and it induces a poset relation $\le$ on the set $[\to Y\rangle$ of right
equivalence classes of homomorphisms ending in $Y$.
Given a morphism $f: X \to Y$,
we denote its right equivalence class by $[f\rangle$ and by definition
$[f\rangle \le [f'\rangle$ if and only if $f \preceq f'.$  
As we will see in proposition \ref{lattice}, the poset $[\to Y\rangle$ is a lattice, thus we will call it
the {\it right factorization lattice} for $Y$.

It should be stressed that $[\to Y\rangle$ is a set, not only a class: namely, the
isomorphism classes of $\Lambda$-modules form a set and for every module $X$, the
homomorphisms $X \to Y$ form a set; we may choose a representative from each
isomorphism class of $\Lambda$-modules and 
given a homomorphism $f: X \to Y$, then there is an isomorphism 
$h: X' \to X$ where $X'$ is such a representative, and $f$ is right equivalent to
$fh.$
	
Recall that a map $f: X \to Y$ is said to be {\it right minimal} provided any direct
summand $X'$ of $X$ with $f(X') = 0$ is equal to zero. If $f: X \to Y$ is a
morphism and $X = X'\oplus X''$ such that $f(X'') = 0$ and $f|X': X' \to Y$
is right minimal, then $f|X'$ is called a {\it right minimalisation of $f$.}
The kernel of a right minimalisation of $f$ will be called the {\it intrinsic kernel
of $f$}, it is unique up to isomorphism. 
	\medskip

\begin{Prop}\label{minimal} Every right equivalence class $[f\rangle$ 
in $[\to Y\rangle$ contains a right
minimal morphism, namely $[f'\rangle$, where $f'$ is a right minimalisation of $f$.
Given right minimal morphisms $f:  X \to Y$ and
$f': X' \to Y$, then $f,f'$ are right equivalent if and only if there is an
isomorphism $h: X \to X'$ such that $f = f'h$.
\end{Prop}
	\medskip
\begin{proof} 
Let $f: X \to Y$ be a homomorphism ending in $Y$. Write $X = X_1\oplus X_2$
such that $f(X_2) = 0$ and $f|X_1: X_1 \to Y$ is right minimal. Let 
$u: X_1\to X_1\oplus X_2$
be the canonical inclusion, $p: X_1\oplus X_2 \to X_1$ the canonical projection. Then
$pu = 1_{X_1}$ and $f = fup$ (since $f(X_2) = 0$). We see that $fu \preceq f$
and $f = fup \preceq fu$, thus $f$ and $fu$ are right equivalent and
$fu = f|X_1$ is right minimal. If the right minimal morphisms
$f: X \to Y$ and $f': X' \to Y$ are right equivalent, then there are 
morphisms $h: X \to X'$ and $h': X' \to Y$ such that $f = f'h$ and $f' = fh'$.
But $f = fh'h$ implies that $h'h$ is an automorphism, and $f' = f'hh'$ implies
that $hh'$ is an automorphism, thus $h, h'$ have to be isomorphisms (see \cite{[ARS]} I.2).
\end{proof}
	\medskip

\noindent
{\bf Remark.} Monomorphisms $X \to Y$ are always right minimal, and the right equivalence
classes of monomorphisms ending in $Y$ may be identified with the submodules of $Y$
(here, we identify the right equivalence class of the monomorphism $f: X \to Y$
with the image of $X$).
	
\begin{Prop}\label{lattice}
The poset $[\to Y\rangle$ is a lattice with zero and one. 
Given $f_1: X_1\to Y$ and $f_2: X_2\to Y$, say with pullback 
$g_1: X \to X_1, g_2: X \to X_2$, the 
meet of $[f_1\rangle$ and $[f_2\rangle$ is given by the map $f_1g_1: X \to Y$, the
join of $[f_1\rangle$ and $[f_2\rangle$ is given by $[f_1,f_2]: X_1\oplus X_2 \to Y$.
\end{Prop}

\begin{proof}[a trivial verification] Write $f = f_1g_1 = f_2g_2$. 
We have $f = f_1g_1 \preceq f_1$ and
$f = f_2g_2 \preceq f_2$, thus $[f\rangle \le [f_1\rangle$ and
$[f\rangle \le [f_2\rangle$. If $f': X' \to Y$ is a morphism with
$[f'\rangle \le [f_1\rangle$ and $[f'\rangle \le [f_2\rangle$, then 
$f' \preceq f_1$ and $f' \preceq f_2$, thus there are morphisms $\phi_i$ with
$f' = f_i\phi_i$, for $i=1,2$ Since $f_1\phi_1 = f_2\phi_2$, the pullback property
yields a morphism $\phi: X' \to X$ such that $\phi_i = g_i\phi$ for $i=1,2$.
Thus $f' = f_1\phi_1 = f_1g_1\phi = f\phi$ shows that $f' \preceq f$, thus
$[f'\rangle \le [f\rangle.$ This shows that $[f\rangle$ is the meet of 
$[f_1\rangle$ and $[f_2\rangle$.

Second, denote the canonical inclusion maps $X_i \to X_1\oplus X_2$ by $u_i$,
for $i=1,2$, thus $[f_1,f_2]u_i = f_i$ and therefore 
$[f_i\rangle \le [[f_1,f_2]\rangle$ for $i=1,2.$ 
Assume that there is given a morphism $g: X'' \to  Y$ with
$[f_i\rangle \le [g\rangle$ for $i=1,2$. This means that there are morphisms
$\psi_i: X_i \to X''$ such that $f_i = g\psi_i$ for $i=1,2$. Let 
$\psi = [\psi_1,\psi_2]: X_1\oplus X_2 \to X''$ (with $\psi u_i = \psi_i$).
Then $[f_1,f_2] = g[\psi_1,\psi_2] = g\psi$ shows that $[f_1,f_2] \preceq g$,
thus $[[f_1,f_2]\rangle \le [g\rangle.$ This shows that $[[f_1,f_2]\rangle$
is the join of $[f_1\rangle$ and $[f_2\rangle$.

It is easy to check that the map $0 \to Y$ is the zero element of $[\to Y\rangle$
and that the identity map $Y \to Y$ is its unit element.
\end{proof} 
	
It should be stressed that {\it if $f_1: X_1\to Y$ and $f_2: X_2\to Y$ are right minimal, 
say with pullback $g_1: X \to X_1, g_2: X \to X_2$, then neither 
the map $f_1g_1$ nor the direct sum map $[f_1,f_2]: X_1\oplus X_2 \to Y$ 
will be right minimal, in general.}
Thus if one wants to work with right minimal maps, one has to right minimalise the 
maps in question. Here are corresponding examples:
	\medskip 

\noindent 
{\bf Examples 1.} 
Let $\Lambda$ be the path algebra of the quiver 
$$
 a \longleftarrow b
$$ 
of type $\mathbb A_2$.

All path algebras of quivers considered in the paper will have coefficients 
in an arbitrary field $k$, unless we specify some further conditions. 
When dealing with the path algebra of a quiver $\Delta$, 
and $x$ is a vertex of $\Delta$, we denote by $S(x)$ (or also just by $x$) the
simple module corresponding to $x$, by $P(x)$ and $Q(x)$ the projective cover or
injective envelope of $S(x)$, respectively. 

Take as maps $f_1, f_2$ the canonical projection $f_1 = f_2: P(b) \to S(b)$, 
this is a right minimal
map. The pullback $U$ of $f_1$ and $f_2$ is a submodule of $P(b)\oplus P(b)$ 
which is isomorphic to $S(a)\oplus P(b)$. 
Since any map $S(a) \to S(b)$ is zero, there is no right minimal map $U \to S(b)$. 

Also, the map $[f_1,f_2]: P(b)\oplus P(b) \to S(b)$ is not right minimal,
since we have $\dim\Hom(P(b),S(b)) = 1.$ 
	\medskip

As we have seen, the poset $[\to Y\rangle$ is a lattice. What will be important in the
following discussion is the fact that we deal with a meet-semilattice 
(these are the posets such that any pair of elements has a meet). Note that all the 
semilattices 
which we deal with turn out to be lattices, however the poset maps to be considered 
will preserve meets, but usually not joins, thus we really work in the category
of meet-semilattices.
	
\begin{Prop} The lattice $[\to Y\rangle$ is modular.
\end{Prop}
	
\begin{proof} 
Let $f_i:X_i\to Y$ be maps with target $Y$, where $1\le i \le 3$, 
such that $f_1\preceq f_3$. We want to show that 
$$
   ([f_1\rangle \vee [f_2\rangle)\wedge [f_3\rangle\ \le \ 
    [f_1\rangle \vee ([f_2\rangle)\wedge [f_3\rangle),
$$
the reverse inequality being trivial. Since  $f_1\preceq f_3$, there is $h: X_1\to X_3$ 
such that $f_1 = f_3h.$ 

First, let us construct an element $f$ in $[f_1\rangle \vee ([f_2\rangle)\wedge [f_3\rangle).$
A pullback diagram
$$
{\beginpicture
\setcoordinatesystem units <2cm,1.2cm>
\arr{0.25 1}{0.75 1}
\arr{0.25 0}{0.75 0}
\arr{0 .7}{0 .3}
\arr{1 .7}{1 .3}

\put{$P$} at 0 1
\put{$X_2$} at 1 1
\put{$X_3$} at 0 0
\put{$Y$} at 1 0

\put{$\ssize p_2$} at 0.5 1.2
\put{$\ssize f_3$} at 0.5 .2
\put{$\ssize p_3$} at -.15 .5
\put{$\ssize f_2$} at 1.15 .5
\endpicture}
$$

yields a map $f_3p_3$ such that $[f_3p_3\rangle = [f_2\rangle)\wedge [f_3\rangle,$ thus
the map $f = [f_3h,f_3p_3]: X_1\oplus P \to Y$ belongs to 
$[f_3h\rangle \vee ([f_2\rangle)\wedge [f_3\rangle)$

Next, we construct an element $f'$ in $([f_1\rangle \vee [f_2\rangle)\wedge [f_3\rangle$.
The map $[f_3h,f_2]: X_1\oplus X_2 \to Y$ belongs to $[f_1\rangle \vee [f_2\rangle$,
now we form the pullback 
$$
{\beginpicture
\setcoordinatesystem units <2.5cm,1.2cm>
\arr{0.25 1}{0.75 1}
\arr{0.25 0}{0.75 0}
\arr{0 .7}{0 .3}
\arr{1.1 .7}{1.1 .3}

\put{$P'$} at 0 1
\put{$X_1\oplus X_2$} at 1.1 1
\put{$X_3$} at 0 0
\put{$Y$} at 1.1 0

\put{$ \bmatrix g_1\cr g_2\endbmatrix$} at 0.5 1.5
\put{$ f_3$} at 0.5 .2
\put{$ g_3$} at -.15 .5
\put{$ \bmatrix f_3h& f_2\endbmatrix$} at 1.5 .5
\endpicture}
$$
the map $f' = f_3g_3 = [f_3h,f_2]\bmatrix g_1\cr g_2\endbmatrix$ is an element of
$([f_3h\rangle \vee [f_2\rangle)\wedge [f_3\rangle$.

It follows from $f_3g_3 = [f_3h,f_2]\bmatrix g_1\cr g_2\endbmatrix = f_3hg_1 + f_2g_2$
that $f_2g_2 = f_3(g_3-hg_1)$, thus the pair of maps $g_2, g_3-hg_1$ factors through
the pullback $P$ of $f_2,f_3$. Thus there is $g: P' \to P$ such that $p_2g = g_2$
and $p_3g = g_3-hg_1.$ It follows that $g_3 = hg_1+p_3g$ and therefore
$$
 f = f_3g_3 = f_3hg_1+f_3p_3g = [f_3h,f_3p_3]\bmatrix g_1\cr g\endbmatrix = f'\bmatrix g_1\cr g\endbmatrix.
$$
This shows that $f \preceq f'$.
\end{proof}

\noindent
{\bf Examples 2. Failure of the chain conditions.}
Here are examples which show that in general  {\it $[\to Y\rangle$ neither satisfies the 
ascending nor the descending chain
condition.} 
Let $\Lambda$ be the Kronecker algebra, this is the path algebra of the quiver
$$
{\beginpicture
\setcoordinatesystem units <1.5cm,1cm>
\put{$a$} at 0 0
\put{$b$} at 1 0
\arr{0.8 0.1}{0.2 0.1}
\arr{0.8 -.1}{0.2 -.1}
\endpicture}
$$
The $\Lambda$-modules are also called {\it Kronecker
modules} (basic facts concerning the Kronecker modules will be recalled in section \ref{sec:14}).
Let $Y = S(b)$, the simple injective module. 

We denote by $Q_n$ the indecomposable preinjective module of length $2n+1$
(thus $Q_0 = S(b) = Q(b),\ Q_1 = Q(a)$). There is a chain of epimorphisms 
$$
   \cdots \longrightarrow Q_2 \overset {f_2} \longrightarrow
 Q_1 \overset {f_1} \longrightarrow   Q_0 = Y,
$$
thus we have the descending chain
$$
 \cdots [f_1f_2f_3\rangle < [f_1f_2\rangle < [f_1\rangle < [1_Y\rangle
$$
in $[\to Y\rangle.$ 

Here, we can assume that all the kernels $f_n:  Q_{n} \to Q_{n-1}$ 
are equal to $R$, where $R$ is a fixed indecomposable module of length $2$. 
Also, if the ground field $k$ is infinite, then there is such a chain of epimorphisms 
such that all the kernels are pairwise different and of length $2$.
In the first case, the kernels of the maps $f_1f_2\cdots f_n$ are all indecomposable
(namely of the form $R[n]$ for $n\in \mathbb N$), in the second, they are direct sums
of pairwise non-isomorphic modules of length 2.
	\bigskip

In order to look at the ascending chain condition, let $P_n$ be indecomposable
preprojective of length $2n+1$. For $i\ge 1$, there are 
epimorphisms $f_i: P_i \to Y$ and monomorphisms $u_i: P_i \to P_{i+1}$
such that $f_{i+1}u_{i} = f_{i}.$ Thus, there is a commutative diagram of maps
$$
{\beginpicture
\setcoordinatesystem units <1.5cm,1.3cm>
\put{$P_1$} at 0 0 
\put{$P_2$} at 1 0 
\put{$\cdots$} at 2 0 
\put{$P_i$} at 3 0 
\put{$P_{i+1}$} at 4 0
\put{$u_1$} at 0.6 0.2
\put{$u_i$} at 3.4 0.2
\arr{0.3 0}{0.7 0}
\arr{3.3 0}{3.7 0}
\arr{4.3 0}{4.7 0}
\put{$\cdots$} at 5 0 

\plot 1.3 0 1.5 0 /
\arr{2.5 0}{2.7 0}
\put{$Y$} at 4 -1 
\arr{0.3 -.2}{3.75 -1} 
\arr{1.3 -.2}{3.8 -.9} 
\arr{3.2 -.2}{3.9 -.8} 
\arr{4 -.2}{4 -.8} 
\put{$f_1$} at  1.5 -.7 
\put{$f_2$} at 2.3 -.3
\put{$f_i$} at 3.3 -.5
\put{$f_{i+1}$} at 4.25 -.5
\endpicture}
$$
and we obtain an ascending chain in $[\to Y\rangle.$ 
$$
 [f_1\rangle < [f_2\rangle < [f_3\rangle < \cdots.
$$

Thus, looking at the sequence of maps
$$
 P_1 \overset {u_1} \longrightarrow P_2 \overset {u_2} \longrightarrow \cdots 
 \longrightarrow Y
$$
we obtain an ascending chain in $[\to Y\rangle.$ 
$$
 [f_1\rangle < [f_2\rangle < [f_3\rangle < \cdots.
$$
	\medskip 

There are similar chains in $[\to Y\rangle$ consisting of elements $[f: X \to Y\rangle$
with $X$ regular Kronecker modules. 
	\bigskip

\noindent 
{\bf Remark.} If we fix $Y$ (as we usually do),
we may consider the class of right minimal morphisms $f: X \to Y$ as the
objects of a category, with maps from $(f: X \to Y)$ to $(f': X'\to Y)$
being given by the maps $h: X \to X'$ such that $f = f'h$. According to
Proposition \ref{minimal}, this category is a groupoid (this means that all morphisms
are isomorphisms), and any connected component of the category is just the class
of all the right minimal maps which belong to a right equivalence class. 
If we work with a skeleton of the category $\mo\Lambda$, then we only have
to consider the sets 
$$
 \rAut(f) = \{h\in \End(X) \mid fh = h\},
$$
this is a subgroup of $\Aut(X)$; we may call it the {\it right automorphism
group} of $f$. The classification problem for the right minimal maps ending in $Y$
is divided in this way into two problems: to determine, on the one hand, 
the structure of the right factorization lattice $[\to Y\rangle$ and, on the 
other hand, to determine $\rAut(f)$ for every 
right minimal map $f$ ending in $Y$. This provides a nice separation of 
the local symmetries and the global directedness, 
as mentioned at the beginning of the paper. 

\section{Morphisms determined by modules: Auslander's First Theorem.}
\label{sec:3}

Here is the decisive definition. Let $f: X \to Y$ be a morphism and $C$ a
module. Then $f$ is said to be {\it right $C$-determined} (or right determined by $C$)
provided the following condition
is satisfied: given any morphism
$f': X' \to Y$ such that $f'\phi$ factors through $f$ for
all $\phi: C \to X'$, then $f'$ itself factors through $f$. 
Thus one deals with the following diagrams:
$$
{\beginpicture
\setcoordinatesystem units <1.3cm,1cm>
\put{\beginpicture
\put{$X$} at 0.3 0
\put{$Y$} at 1.5 0
\put{$X'$} at 0 1
\put{$f\strut$} at .85 -.22
\put{$f'$} at 0.75 0.85
\arr{0.5 0}{1.3 0}
\arr{.2 .9}{1.3 .2}

\put{$\phi$} at -.5 1.2
\put{$\phi'$} at -.4 0.1
\put{$C$} at -1 1
\arr{-.7 1}{-.3 1}
\setdashes <1mm>
\arr{-.7 0.7}{.1 0.1}

\endpicture} at 0 0
\put{\beginpicture
\put{$X$} at 0.3 0
\put{$Y$} at 1.5 0
\put{$X'$} at 0 1
\put{$f\strut$} at 0.85 -.22
\put{$f'$} at 0.75 0.85
\arr{0.5 0}{1.3 0}
\arr{.2 .9}{1.3 .2}

\put{$h$} at -.02 0.5
\setdashes <1mm>
\arr{.05 0.75}{.25 0.25}

\endpicture} at 4 0

\endpicture}
$$
The existence of the dashed arrow $\phi'$ on the left for all possible
maps $\phi: C \to X'$ shall imply the existence of the dashed arrow $h$ on the right
(of course, the converse implication always holds true: If $f' = fh$ for some morphism
$h$, then $f'\phi = f(h\phi)$ for all morphisms $\phi: C \to X'$). 

\begin{Prop}\label{determined}
Let $: X \to Y$ be a morphism, let $C, C'$ be modules.

\item{\rm (a)} Assume that $\add C = \add C'$. Then
$f$ is right $C$-determined if and only if $f$ is right $C'$-determined.
\item{\rm (b)}  
If $f$ is right $C$-determined, then $f$ is also right $(C\oplus C')$-determined.
\end{Prop}
	
\begin{proof}
Trivial verification.
\end{proof}

We denote by ${}^C[\to Y\rangle$ the set of the right equivalence 
classes of the morphisms ending in $Y$ which are
right $C$-determined. We will see below that also ${}^C[\to Y\rangle$ is a lattice,
thus we call it the {\it right $C$-factorization lattice} for $Y$. 
	\bigskip

Note that ${}^C[\to Y\rangle$ is usually not closed under predecessors or
successors inside $[\to Y\rangle$. But there is the following important property:
	
\begin{Prop}\label{closed-meets}
The subset ${}^C[\to Y\rangle$ of $[\to Y\rangle$ is closed
under meets.
\end{Prop}

\begin{proof} Let $f_1: X_1 \to Y$ and $f_2: X_2 \to Y$ be right $C$-determined. 
As we know, the meet of $[f_1\rangle$ and $[f_2\rangle$ is given by
forming the pullback of $f_1$ and $f_2$. Thus assume that $X$ is the pullback with
maps $g_1: X \to X_1$ and $g_2: X \to X_2$ and let $f = f_1g_1 = f_2g_2.$ 
We want to show that $f$ is right $C$-determined. 
Thus, assume that there is
given $f': X' \to Y$ such that  for
any $\phi: C \to X'$, there exists $\phi': C \to X$  such that
$f'\phi = f\phi'$. Then we see that for any $\phi: C \to X'$, we have
$f'\phi = f\phi = f_1(g_1\phi)$, thus $f'\phi$ factors through $f_1$. Since
$f_1$ is right $C$-determined, it follows that $f'$ factors through $f_1$, say
$f' = f_1h_1$ for some $h_1: X' \to X_1$. Similarly, for any $\phi: C \to X'$,
the morphism $f'\phi$ factors through $f_2$ and therefore 
$f' = f_2h_2$ for some $h_2: X' \to X_2$. Now $f_1h_1 = f' = f_2h_2$ implies that
there is $h: X' \to X$ such that $g_1h = h_1,$ and $g_2h = h_2$. Thus
$f' = f_1h_1 = f_1g_1h   = fh$ shows that $f'$ factors through $f$.
\end{proof} 
	
We should stress that ${}^C[\to Y\rangle$ usually is not closed under joins,
see the examples at the end of the section. One of these examples is chosen
in order to convince the reader that this is not at all a drawback, 
but an important feature if we want to work with lattices of finite height. 
	
\begin{Thm}[Auslander's First Theorem]\label{first}
For any $\Lambda$-module $Y$, one has
\Rahmen{
 [\to Y\rangle = \bigcup\nolimits_C {}^C[\to Y\rangle,
}
where $C$ runs through all the $\Lambda$-modules (or just through representatives of all
multiplicity-free $\Lambda$-modules) and this is a filtered union of meet-semilattices.
\end{Thm}

By definition, the sets ${}^C[\to Y\rangle$ are subsets of $[\to Y\rangle$.
By \ref{determined}(a), we know that ${}^C[\to Y\rangle$ only depends on $\add C$, thus we
may restrict to look at representatives of multiplicity-free $\Lambda$-modules $C$.
Proposition \ref{determined}(b) asserts that both ${}^C[\to Y\rangle$ and ${}^{C'}[\to Y\rangle$ are 
contained
in ${}^{C\oplus C'}[\to Y\rangle$, thus we deal with a filtered union.
According to \ref{closed-meets}, we deal with embeddings of meet-semilattices. 
The essential assertion of Theorem \ref{first} is 
that any morphism is right determined by some module, the usual formulation of 
Auslander's First Theorem.
A discussion of this assertion and its proof follows.
	\medskip 

There is a precise formula which yields for $f$ the smallest possible
module $C(f)$ which right determines $f$. We will call it the minimal right
determiner of $f$, any other right determiner of $f$ will have $C(f)$ as
a direct summand. 
	
We need another definition.  An indecomposable projective module $P$ is said
to {\it almost factor through $f$,} provided there is a commutative diagram
of the following form
$$
{\beginpicture
\setcoordinatesystem units <2cm,1.2cm>
\arr{0.35 1}{0.75 1}
\arr{0.25 0}{0.75 0}
\arr{0 .7}{0 .3}
\arr{1 .7}{1 .3}

\put{$\rad P$} at 0 1
\put{$P$} at 1 1
\put{$X$} at 0 0
\put{$Y$} at 1 0

\put{$\ssize \iota$} at 0.5 1.2
\put{$\ssize f$} at 0.5 .2
\put{$$} at -.15 .5
\put{$\ssize \eta$} at 1.15 .5
\endpicture}
$$
where $\iota$ is the inclusion map, 
such that the image of $\eta$ is not contained in the image of $f$.
Let us mention the following: {\it If the indecomposable projective module $P$
almost factors through $f$, then $P/\rad P$ embeds into the cokernel $\Cok(f).$}
Namely, given a map $\eta: P \to Y$ 
such that the image of $\eta$ is not contained in the image of $f$, as well as
the commutative diagram above, we may complete the diagram by adding the cokernels 
of the horizontal maps:
$$
{\beginpicture
\setcoordinatesystem units <2cm,1.2cm>
\arr{0.35 1}{0.75 1}
\arr{0.25 0}{0.75 0}
\arr{1.25 1}{1.55 1}
\arr{1.25 0}{1.55 0}
\arr{2.45 1}{2.75 1}
\arr{2.45 0}{2.75 0}
\arr{0 .7}{0 .3}
\arr{1 .7}{1 .3}
\arr{2 .7}{2 .3}

\put{$\rad P$} at 0 1
\put{$P$} at 1 1
\put{$X$} at 0 0
\put{$Y$} at 1 0
\put{$P/\rad P$} at 2 1
\put{$0$} at 3 1
\put{$\Cok(f)$} at 2 0
\put{$0$} at 3 0

\put{$\ssize \iota$} at 0.5 1.2
\put{$\ssize f$} at 0.5 .2
\put{$$} at -.15 .5
\put{$\ssize \eta$} at 1.15 .5
\put{$\ssize \eta'$} at 2.15 .5
\endpicture}
$$
Since the image of $\eta$ is not contained in the image of $f$, we see
that $\eta'$ is non-zero, thus $P/\rad P$ is a submodule of $\Cok(f).$

\begin{Thm}[Determiner formula of Auslander-Reiten-Smal\o]\label{determiner} 
Let $f$ be a morphism ending in $Y$.
Let $C(f)$ be the direct sum of the
indecomposable modules of the form $\tau^{-}K$, where $K$ is an indecomposable 
direct summand of the intrinsic kernel of $f$ and of the
indecomposable projective modules which almost factor through $f$, one from each
isomorphism class. 
Then $f$ is right $C$-determined if and only if $C(f) \in  \add C$.
\end{Thm}

The theorem suggests to call $C(f)$ the {\it minimal right determiner of $f$}.
For the proof of Theorem \ref{determiner}, see \cite{[ARS]} and also \cite{[R9]}.
	
\begin{Cor} 
Any morphism $f$ is right $C$-determined 
by some $C$, for example by the module
\Rahmen{\tau^{-} \Ker(f) \oplus P(\soc\Cok(f)).}
\end{Cor}
	
\begin{proof} 
We have to show that $C(f)$ is a direct summand of
$\tau^{-} \Ker(f) \oplus P(\soc\Cok(f))$. The intrinsic kernel of
$f$ is a direct summand of $\Ker(f)$, thus if $K$ is an indecomposable 
direct summand of the intrinsic kernel of $f$, then $\tau^{-}K$ is a
direct summand of $\tau^{-} \Ker(f).$ Now assume that $S$ is a simple
module such that $P(S)$  almost factors through $f$. Then $S$ is a submodule
of $\Cok(f)$, thus $P(S)$ is a direct summand of $P(\soc\Cok(f)).$
\end{proof}

\begin{Cor}[Auslander]
The module $\tau^{-}\Ker(f)\oplus \Lambda$ right determines $f$.
\end{Cor} 
	
\begin{Cor}\label{projective}
Let $P$ be a projective module and $f: X \to Y$
a right minimal morphism. 
Then $f$ is right $P$-determined if and only if $f$ is a monomorphism
and the socle of the cokernel of $f$ is generated by $P$.
\end{Cor}
	
\begin{proof} 
This is an immediate consequence of the determiner formula: First, assume that
$f$ is right $P$-determined. Then the intrinsic kernel of $f$ has to be zero.
Since we assume that $f$ is right minimal, $f$ must be a monomorphism. If $S$
is a simple submodule of the cokernel of $f$, then $P(S)$ almost factors through
$f$, thus $P(S)$ is a direct summand of $P$. This shows that the socle of the
cokernel of $f$ is generated by $P$. Conversely, assume that $f$ is a monomorphism
and the socle of the cokernel of $f$ is generated by $P$. Since $f$ is a
monomorphism, $C(f)$ is the direct sum of all 
indecomposable projective modules $P'$ which almost factor through $f$. Such
a module $P'$ is the projective cover of a simple submodule of $\Cok(f)$.
Since $P$ generates the socle of the cokernel of $f$, it follows that $P'$ is
a direct summand of $P$. Thus $C(f)$ is in $\add P$, therefore 
$f$ is right $P$-determined. 
\end{proof} 

\begin{Cor}\label{monos}
A right minimal 
morphism $f: X \to Y$ is a monomorphism if and only if it is
right $\Lambda$-determined.
\end{Cor}
	
\noindent 
{\bf Example 3.} 
{\it The subset ${}^C[\to Y\rangle$ of $[\to Y\rangle$ is usually not closed
under joins,} as the following example shows: Let $\Lambda$ be the path algebra
of the quiver of type $\mathbb A_3$ with two sources, namely of 
$$
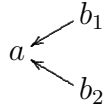
{\beginpicture
\setcoordinatesystem units <1cm,.5cm>
\put{$a$} at 0 0
\put{$b_1$} at 1 1
\put{$b_2$} at 1 -1
\arr{0.75 0.85}{0.2 0.2} 
\arr{0.75 -.85}{0.2 -.2} 
\endpicture}
$$
Let $f_i: P(b_i) \to Q(a)$ be non-zero maps for $i=1,2$,
these are monomorphisms, thus they are right $\Lambda$-determined.
The join of $[f_1\rangle$ and $[f_2\rangle$ in $[\to Y\rangle$
is given by the map $[f_1,f_2]: P(b_1)\oplus P(b_2) \to Q(a)$. Clearly, this
map is right minimal, but it is not injective. Thus $[f_1,f_2]$ is not
right $\Lambda$-determined. 
	\medskip

\noindent 
{\bf Example 4.} This example indicates that in general, 
it may not be advisable to ask for 
closure under joins. Consider again the Kronecker algebra $\Lambda$ as
exhibited in example 2. Let $Y = Q(a)$ and $C = \Lambda$,
thus all the maps in ${}^C[\to Y\rangle$ are given by inclusion maps
$f: X \to Y$, where $X$ is a submodule of $Y$. In fact, we may identify 
${}^C[\to Y\rangle$ with the submodule lattice of $Y$, it 
is a lattice of
height $3$ which looks as
follows:
$$
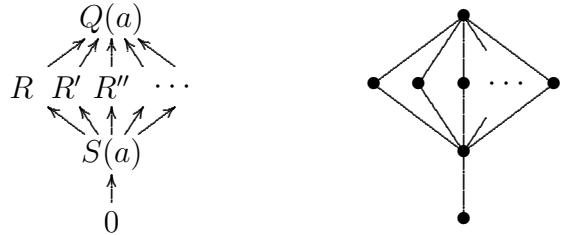
{\beginpicture
\setcoordinatesystem units <1.2cm,.9cm>
\put{\beginpicture

\put{$Q(a)$} at 2 2
\put{$R$} at 1 1 
\put{$R'$} at 1.5 1  
\put{$R''$} at 2 1  
\put{$S(a)$} at 2 0
\put{$0$} at 2 -1  
\put{$\cdots$} at 2.7 1 
\arr{1.7 0.3}{1.3 0.7}
\arr{2.3 0.3}{2.7 0.7}
\arr{2.7 1.3}{2.3 1.7}
\arr{1.3 1.3}{1.7 1.7}
\arr{1.85 0.3}{1.65 0.7}  
\arr{1.65 1.3}{1.85 1.7}
\arr{2 0.3}{2 0.7}
\arr{2 1.3}{2 1.7}
\arr{2.15 0.3}{2.35 0.7}
\arr{2.35 1.3}{2.15 1.7}
\arr{2 -0.7}{2 -.3}
\endpicture} at 0 0
\put{\beginpicture
\multiput{$\bullet$} at  1 1  2 0  3 1  2 2   1.5 1  2 1  2 -1  /
\put{$\cdots$} at 2.5 1 
\plot 1 1  2 0  3 1  2 2  /
\plot 1 1  2 2 /
\plot 2 0  1.5 1  2 2 /
\plot 2 0  2 1  2 2 /
\plot 2 0  2.25 .5   /
\plot 2.25 1.5  2 2 /
\plot 2 -1  2 0 /
\endpicture} at 4 0
\endpicture}
$$
Here, the modules  $R,R',R'',\dots$ are the indecomposable
representations of length 2, one from each isomorphism class and all
the arrows are inclusion maps.

The join in ${}^C[\to Y\rangle$ 
of two different maps $f_1,f_2$ in the height 2 layer is just the identity map $Y \to Y$,
whereas the join of $f_1,f_2$ 
in $[\to Y\rangle$ is the direct sum map $[f_1,f_2]: R_1\oplus R_2 \to Y.$ 
More generally, if there are given
$n$ pairwise different regular modules $R_1,\dots, R_n$ of length 2
with inclusion maps $f_i: 
R_i \to Y$, then the join in $[\to Y\rangle$ 
is the direct sum map 
$[f_1,\dots,f_n]: R_1\oplus\cdots\oplus R_n \to Y$. Let us stress
that all these direct sum maps are right minimal (thus here we deal
with a cofork as defined in section \ref{sec:13}).  
Thus, if the base field $k$ is infinite, the smallest subposet of 
$[\to Y\rangle$ closed under meets and joins and containing the inclusion maps
$R \to Y$ with $R$ regular of length 2 will have infinite height. 
	
\begin{Prop}\label{non-zero-map}
Let $f: X \to Y$ be a morphism. If $C'$
is an indecomposable direct summand of $C(f)$, then $\Hom(C',Y) \neq 0.$
\end{Prop}
	
\begin{proof} By definition, there are two kinds of indecomposable direct summands of $C(f)$,
the non-projective ones are of the form $\tau^{-}K'$, where $K'$ is an indecomposable
direct summand of the intrinsic kernel of $f$, the remaining ones are the
indecomposable projective modules which almost factor through $f$. Of course,
if $P$ is an indecomposable projective module which almost factors through $f$,
then $\Hom(P,Y) \neq 0$. 

Thus, we have to consider a module of the form $C' = \tau^{-}K'$
with $K'$ an indecomposable direct summand of the intrinsic kernel $K$ of $f$ and 
note that $K'$ cannot be injective. 
Let us denote by $u: K \to X$ and $u': K'\to K$ the inclusion maps. Since $u'$ is split mono, there is $r: K \to K'$ with $ru' = 1_{K'}.$ Let 
$$ 
  0 \longrightarrow K' \overset\mu\longrightarrow  M  \overset\epsilon
 \longrightarrow C' \longrightarrow  0
$$ 
be the Auslander-Reiten sequence
starting with $K'$. Since the composition $uu'$ is not split mono, there is 
a map $\phi: M \to X$ with $\phi\mu = uu'.$ Thus, 
there is the following commutative diagram with exact rows
$$
{\beginpicture
\setcoordinatesystem units <2cm,1.2cm>
\arr{-0.7 1}{-0.3 1}
\arr{-0.7 0}{-0.3 0}
\arr{0.35 1}{0.75 1}
\arr{0.25 0}{0.75 0}
\arr{1.25 1}{1.75 1}
\arr{1.25 0}{1.75 0}
\arr{2.25 1}{2.75 1}
\arr{0 .7}{0 .3}
\arr{1 .7}{1 .3}
\arr{2 .7}{2 .3}

\put{$0$} at -1 1
\put{$0$} at -1 0
\put{$K'$} at 0 1
\put{$M$} at 1 1
\put{$K$} at 0 0
\put{$X$} at 1 0
\put{$C'$} at 2 1
\put{$0$} at 3 1
\put{$Y$} at 2 0

\put{$\ssize \mu$} at 0.5 1.2
\put{$\ssize u$} at 0.5 .2
\put{$\ssize \epsilon$} at 1.5 1.2
\put{$\ssize f$} at 1.5 .2
\put{$\ssize u'$} at -.15 .5
\put{$\ssize \phi$} at 1.15 .5
\put{$\ssize \phi'$} at 2.15 .5
\endpicture}
$$
If we assume that $\phi'= 0$, then $f\phi = 0$, thus $\phi$ factors through the
kernel of $f$, say $\phi = u\phi''.$ Consequently, $uu' = \phi\mu = u\phi''\mu$.
But $u$ is injective, thus $u' = \phi''\mu$ and therefore $1_{K'} = 
ru' = r\phi''\mu$. But this means that $\mu$ is split mono, a contradiction.
It follows that $\phi' \neq 0$, thus $\Hom(C',Y)\neq 0.$ 
\end{proof} 

\noindent 
{\bf Remark.} Proposition \ref{non-zero-map} asserts that all the indecomposable direct summands $C'$
of the minimal right determiner $C(f)$ of a map $f: X\to Y$ satisfy $\Hom(C',Y)\neq 0$.
Actually, according to \cite{[ARS]}, Proposition XI.2.4 
(see also \cite{[R9]}), such a module $C'$
is equipped with a distinguished non-zero map $C'\to Y$ which is said to 
``almost factor through'' $f$. At the beginning of this section we gave a corresponding definition in the special case when $C'$ is projective. 
See also the Remark 3 at the end of section \ref{sec:4}. 

\section{The Auslander bijection. Auslander's Second Theorem.}
\label{sec:4}

Let $C,Y$ be objects. Let $\Gamma(C)  = \End(C)^{\text{op}}$. We always will consider 
$\Hom(C,Y)$ as a $\Gamma(C)$-module. For any module $M$, we denote by $\mathcal  SM$ the
set of all submodules (it is a lattice with respect to intersection and sum of
submodules).

Define 
$$
 \eta_{CY}: \bigcup\nolimits_X \Hom(X,Y) \ \longrightarrow\ \mathcal  S(\Hom(C,Y))
$$ 
by $\eta_{CY}(f) = \Img\Hom(C,f) = f\cdot \Hom(C,X)$ for $f: X \to Y$
(note that $f\cdot \Hom(C,X)$ clearly is a $\Gamma(C)$-submodule). 
	\medskip 

Here is a reformulation of the definition of $\eta_{CY}$.
	
\begin{Prop} Let $f: X \to Y$. Then $\eta_{CY}(f)$ is the set 
of all $h\in \Hom(C,Y)$ which factor through $f$. This subset of $\Hom(C,Y)$
is a $\Gamma(C)$-submodule.
\end{Prop}

\begin{proof}We have mentioned already, that
$\eta_{CY}(f)$ is a $\Gamma(C)$-submodule of $\Hom(C,Y)$.
Also, if $h\in \eta_{CY}(f) = f\Hom(C,X)$, then $h$ factors through $f$.
And conversely, if $h$ 
factors through $f$, then $h$ belongs to $f\Hom(C,X) = \eta_{CY}(f).$
\end{proof} 
	
\begin{Lem}
If $X = X_0\oplus X_1$ and $f(X_0) = 0$, then $\eta_{CY}(f) = \eta_{CY}(f|X_1).$
\end{Lem}
	
\begin{proof} 
Let $X = X_0\oplus X_1$ and write $f = \bmatrix f_0&f_1\endbmatrix
= \bmatrix 0&f_1\endbmatrix$ with 
$f_i: X_i \to Y.$ Then, for $\phi_i: C \to X_i$, we have
$\bmatrix 0&f_1\endbmatrix \bmatrix \phi_0 \cr \phi_1\endbmatrix = f_1\phi_1.$
Thus
$$
 \eta_{CY}(f) = f\Hom(C,X_0\oplus X_1) = f_1\Hom(C,X_1) = \eta_{CY}(f_1).
$$
\end{proof}

In particular: {\it If $f_1$ is a right minimal version of $f$, 
then $\eta_{CY}(f) = \eta_{CY}(f_1)$.} Thus, 
$\eta_{CY}$ is constant on right equivalence
classes and we can define $\eta_{CY}([f\rangle) = \eta_{CY}(f)$. We obtain 
in this way a map
$$
 \eta_{CY}: [\to Y\rangle \to \mathcal  S(\Hom(C,Y)).
$$ 
Of special interest is the restriction of $\eta_{CY}$ to ${}^C[\to Y\rangle.$

\begin{Prop} Let $C,Y$ be modules. The map
$$
 \eta_{CY}: {}^C[\to Y\rangle \to \mathcal  S\Hom(C,Y).
$$
is injective and preserves meets. As a consequence, it preserves and reflects the ordering.
\end{Prop}

\begin{proof} A trivial verification: First, let us show that $\eta_{CZ}$ is injective.
Consider maps $f: X \to Y$ and $f': X' \to Y$ such that $f\Hom(C,X) = f'\Hom(C,X').$
Since $f$ is right $C$-determined and $f'\Hom(C,X') \subseteq f\Hom(C,X)$,
we see that $f'\in f\Hom(X',X)$. Since 
$f'$ is right $C$-determined and $f\Hom(C,X) \subseteq f'\Hom(C,X')$,
we see that $f\in f\Hom(X,X')$. But this means that $f'\preceq f \preceq f'$, thus 
$[f\rangle = [f'\rangle.$ 
\end{proof} 

Next, consider the following pullback diagram
$$
{\beginpicture
\setcoordinatesystem units <2.5cm,1.2cm>
\arr{0.25 1}{0.75 1}
\arr{0.25 0}{0.75 0}
\arr{0 .7}{0 .3}
\arr{1 .7}{1 .3}

\put{$X$} at 0 1
\put{$X_1$} at 1 1
\put{$X_2$} at 0 0
\put{$Y$} at 1 0

\put{$\ssize g_1$} at 0.5 1.2
\put{$\ssize f_2$} at 0.5 .2
\put{$\ssize g_2$} at -.15 .5
\put{$\ssize f_1$} at 1.15 .5
\endpicture}
$$
such that $f_1,f_2$ both are right $C$-determined. Let $f = f_1g_1 = f_2g_2,$
thus the meet of $[f_1\rangle$ and $[f_2\rangle$  is $[f\rangle$.
Now $f = f_1g_1$ shows that $f\Hom(C,X) \subseteq f_1\Hom(C,X_1).$
Similarly, $f = f_2g_2$ shows that $f\Hom(C,X) \subseteq f_2\Hom(C,X_2).$
Both assertions together yield 
$$
 f\Hom(C,X) \subseteq f_1\Hom(C,X_1) \cap f_2\Hom(C,X_2).
$$ 
Conversely, take an element in
$f_1\Hom(C,X_1) \cap f_2\Hom(C,X_2),$ say $f_1\phi_1 = f_2\phi_2$
with $\phi_i: C \to X_i$, for $i=1,2$. The pullback property yields a morphism
$\phi: C \to X$ such that $g_i\phi = \phi_i$ for $i=1,2$. Therefore
$ f_1\phi_1 = f_1g_1\phi = f\phi$ belongs to $f\Hom(C,X)$. Thus, 
$$
 f_1\Hom(C,X_1) \cap f_2\Hom(C,X_2) \subseteq f\Hom(C,X),
$$ 
and therefore
$$
 f_1\Hom(C,X_1) \cap f_2\Hom(C,X_2) = f\Hom(C,X).
$$ 

In general, assume that $L,L'$ are posets with meets and  $\eta: L \to L'$
is a set-theoretical map which preserves meets. 
Then $a\le b$ in $L$ implies $\eta(a) \le \eta(b)$ in $L'$. Namely,
$a \le b$ gives $a\wedge b = a$, thus $\eta(a)\wedge \eta(b) = \eta(a)$
and therefore $\eta(a) \le \eta(b)$. Conversely, if $a,b$ are arbitrary elements
in $L$ with $\eta(a) \le \eta(b)$, let $c = a\wedge b$. Then $\eta(c) =
\eta(a)\wedge \eta(b) = \eta(a)$. Thus, if $\eta$ is injective, then $c = a$,
and $a\wedge b = a$ implies $a \le b$. Altogether, we see that $\eta_{CY}$
preserves and reflects the ordering. 
	
Auslander's Second Theorem (as established in \cite{[A1]}) asserts:
	
\begin{Thm}[Auslander's Second Theorem] The map 
$$
 \eta_{CY}: [\to Y\rangle \to \mathcal  S(\Hom(C,Y))
$$ 
is surjective.
\end{Thm}
	
Altogether we see: {\it The map $\eta_{CY}$ defined by $\eta_{CY}(f) = \Img\Hom(C,f)$
yields a lattice isomorphism}
\Rahmen{\eta_{CY}: \ {}^C[\to Y\rangle \longrightarrow \mathcal  S\Hom(C,Y).}
	
Better: {\it The composition
\Rahmen{{}^C[\to Y\rangle \quad \subseteq \quad  [\to Y\rangle \quad 
\overset{\eta_{CY}}\longrightarrow
\quad  \mathcal  S_{\Gamma(C)}\Hom(C,Y)}
of the inclusion map and the 
map $\eta_{CY}$ defined by $\eta_{CY}(f) = \Img\Hom(C,f)$ is 
a lattice isomorphism.}
		\medskip

\noindent 
{\bf Convention.} In the following, several examples of Auslander bijections
will be presented. When looking at the submodule lattice $\mathcal  S M$ of a module $M$,
we usually will mark (some of) the elements of $\mathcal  SM$ by bullets
$\bullet$ and connect comparable elements by a solid lines. Here, going upwards 
corresponds to the inclusion relation. 

For the corresponding lattices ${}^C[\to Y\rangle$, we often will mark an
element $[f: X \to Y\rangle$ (with $f$ a right minimal map)
by just writing $X$ and we will connect neighboring pairs $[f: X \to Y\rangle \le 
[f': X' \to Y\rangle$ by drawing an (upwards) arrow $X \to Y$. On the other hand, 
sometimes it seems to be more appropriate to refer to the right minimal map $f: X \to Y$
with kernel $K'$ and image $Y'$ by using the short exact sequence notation 
$K' \to X \to Y'$.
	\medskip 

Note that the lattice $\mathcal  S\Hom(C,Y)$ has two distinguished elements,
namely $\Hom(C,Y)$ itself as well as its zero submodule. 
Under the bijection $\eta_{CY}$ the total submodule $\Hom(C,Y)$ corresponds to
the identity map $1_Y$ of $Y$, this is not at all exciting. But of interest 
seem to be the maps in $\eta_{CY}^{-1}(0)$, we will discuss them in 
 this will be discussed in Proposition \ref{zero}
	\bigskip

\noindent 
{\bf The special case $C= {}_\Lambda\Lambda$.} It is worthwhile to draw the attention
on the special case when $C = {}_\Lambda\Lambda$.

\begin{Prop}
The special case of the Auslander bijection $\eta_{\Lambda Y}$
is the obvious identification of both
${}^\Lambda[\to Y\rangle$ and $\mathcal  S\Hom(\Lambda,Y)$
with $\mathcal  SY$.
\end{Prop}
	
\begin{proof}  First, consider ${}^\Lambda[\to Y\rangle$:
The determiner formula asserts: a right minimal morphism is right 
$\Lambda$-determined
if and only if it is a monomorphism. Thus ${}^\Lambda[\to Y\rangle$
is just the set of right equivalence classes of monomorphisms ending in $Y$,
and the map $f \mapsto \Img(f)$ yields an identification between 
the set of right equivalence classes of monomorphisms ending in $Y$ and the
submodules of $\Lambda$.

Next, we deal with $\mathcal  S \Hom(\Lambda,Y)$.
Note that $\Gamma({}_\Lambda\Lambda) = \End({}_\Lambda\Lambda)^{\text{op}} = \Lambda$ and
there is a canonical identification $\epsilon: \Hom(\Lambda,Y) \simeq Y$ 
(given by $\epsilon(h) = h(1)$ for $h\in \Hom(\Lambda,Y)$), 
thus $\mathcal  S\epsilon: \mathcal  S \Hom(\Lambda,Y) \simeq \mathcal  S Y$
(with $\mathcal  S\epsilon(U) = \{h(1)\mid h\in U\}$ for $U$ a submodule of
$\Hom(\Lambda,Y)$).

The Auslander bijection $\eta_{\Lambda,Y}$ attaches to $f: X \to Y$ the
submodule $f\Hom(\Lambda,X)$ and there is the following commutative diagram:
$$
\hbox{\beginpicture
\setcoordinatesystem units <2.5cm,1.5cm>
\put{${}^\Lambda[\to Y\rangle$} at 0 0 
\put{$\mathcal  S \Hom(\Lambda,Y)$} at 2 0
\put{$\mathcal  SY$} at 1 -.9 
\arr{0.2 -0.2}{0.8 -.8}
\arr{0.4 0}{1.5 0}
\arr{1.8 -0.2}{1.2 -.8}
\put{$\eta_{\Lambda Y}$} at 1 .2
\put{$\Img$} at 0.4 -.6
\put{$\mathcal  S\epsilon$} at 1.6 -.6

\put{$\Hom(\Lambda,Y)$} at 3.3 0
\put{$Y$} at 3.3 -1
\arr{3.3 -.2}{3.3 -.8}
\put{$\epsilon$} at 3.4 -.5
\endpicture}
$$
Namely, for $f: X \to Y$ we have 
$$
\hbox{\beginpicture
\setcoordinatesystem units <1cm,.7cm>
\put{$(\mathcal  S\epsilon)\eta_{\Lambda Y}(f) =
(\mathcal  S\epsilon)(f\Hom(\Lambda,X))$} at 0 1
\put{$ = \{fh(1)\mid h\in \Hom(\Lambda,X)\} =
 \{f(x)\mid x\in X\} = \Img(f).$} at 4.15 0 
\endpicture}
$$
\end{proof}

As a consequence, we see that {\it all possible submodule lattices 
$\mathcal  SY$ occur as images under the Auslander bijections.} This assertion can
be strengthened considerably, as we want to show now.
	\bigskip 

By definition, an artin algebra $\Lambda$ is an artin $k$-algebra for some
commutative artinian ring $k$ (this means that $\Lambda$ is a $k$-algebra and that it
is finitely generated as a $k$-module). Such an algebra is said to be 
{\it strictly wild} (or better {\it strictly $k$-wild}),
provided for any artin $k$-algebra $\Gamma$, there is a full exact embedding
$\mo \Gamma \to \mo \Lambda$. If $M$ is a $\Lambda$-module and $M'$ is a $\Lambda'$-module,
a {\it semilinear isomorphism} from $M$ to $M'$ is a pair $(\alpha,f)$, where 
$\alpha: \Lambda\to \Lambda'$ is an algebra isomorphism, and $f: M \to M'$  is an
isomorphism of abelian groups such that $f(\lambda m) = \alpha(\lambda)f(m)$ for all
$\lambda\in \Lambda$ and $m\in M$. It is clear that any semilinear isomorphism from $M$ to $M'$
induces a lattice isomorphism $\mathcal  SM \to \mathcal  SM'.$

\begin{Prop}\label{wild}
Let $\Lambda$ be an artin $k$-algebra which is strictly
$k$-wild. Let $\Gamma$ be an artin $k$-algebra and $M$ a $\Gamma$-module. Then there
are $\Lambda$-modules $C,Y$ such that the $\Gamma(C)$-module $\Hom(C,Y)$ is
semilinearly isomorphic to $M$. 
Thus there is a lattice isomorphism ${}^C[\to Y\rangle \to \mathcal  SM$.
\end{Prop}

\begin{proof} 
Let $F: \mo\Gamma \to \mo\Lambda$ be a full embedding (we do not need that it is exact).
Let $C = F({}_\Gamma\Gamma)$ and $Y = F(M)$. Let $\alpha: \Gamma = \End({}_\Gamma\Gamma)^{\text{op}} \to 
\End(C)^{\text{op}} = \Gamma(C)$ as well as $f: M = \Hom({}_\Gamma\Gamma,M) \to \Hom(C,M)$ both be
given by applying the functor $F$. Since $F$ is a full embedding, $\alpha$ is an algebra isomorphism and
$f$ is an isomorphism of abelian groups. The functoriality of $F$ asserts that we also have
$f(\gamma m) = \alpha(\gamma)f(m)$ for all 
$\gamma\in \Gamma$ and $m\in M$. This shows that the pair $(\alpha,f)$ is a semilinear
isomorphism. 
\end{proof}

\noindent 
{\bf Remark.} If $F: \mo\Gamma \to \mo\Lambda$ is a full embedding functor, and $C,Y$ are $\Gamma$-modules,
then we obtain a bijection
$$
{\beginpicture
\setcoordinatesystem units <1cm,1cm>
\put{${}^C[\to y\rangle$\strut} [r] at 0 0
\put{$\mathcal  S\Hom(C,Y)$\strut} [l] at 1 0
\put{$\mathcal  S\Hom(F(C),F(Y))$\strut}  at 6.3 0
\put{${}^{F(C)}[\to F(Y)$,\strut} [l] at 10 0
\arr{0.2 0}{0.8 0}
\arr{8.3 0}{9.8 0}
\arr{0.2 0}{0.8 0}
\put{$\ssize \eta_{CY}$} at 0.5 0.3
\put{$\ssize F$} at 3.9 0.3
\arr{3.6 0}{4.3 0} 
\put{$\ssize \eta_{F(C),F(Y)}^{-1}$} at 9 0.3
\endpicture}
$$
but even if $F$ is exact, such a bijection will not be given by applying directly $F$. Namely, if
$f: X \to Y$ is right minimal and right $C$-determined, then the kernel of $f$ belongs to $\add \tau C$,
thus the kernel of $F(f)$ belongs to $\add F(\tau C)$, whereas the intrinsic kernel of any 
right $F(C)$-determined map has to belong to $\add \tau F(C)$ and the $\Lambda$-modules $F(\tau C)$ and $\tau F(C)$
may be very different, as the obvious embeddings of the category of $n$-Kronecker modules into the category of
$(n+1)$-Kronecker modules (using for one arrow the zero map) show.
	\medskip 

Note that under a full exact embedding functor 
$F: \mo\Gamma \to \mo\Lambda$, submodule lattices are
usually not preserved: given a $\Gamma$-module $M$, the functor
$F$ yields an embedding of $\mathcal  S({}_\Gamma M)$ 
into $\mathcal  S({}_\Lambda F(M))$, but usually
this is a proper embedding. Actually, for any finite-dimensiona algebra $\Lambda$, 
there are submodule lattices
$\mathcal  S({}_\Gamma M)$ which cannot be realized as the submodule lattice
of any $\Lambda$-module.  Namely, assume that the length of the
indecomposable projective $\Lambda$-modules is bounded by $t$ and 
take a finite-dimensional algebra $\Gamma$ with a local 
$\Gamma$-module $M$ of length $t+1$. Then
$\mathcal  S M$ is a modular lattice of height $t+1$ with a unique element
of height $t$ (the radical of the module $M$). If $\mathcal  S({}_\Lambda Y)$
is of the form $\mathcal  SM$, then $Y$ has to be a local $\Lambda$-module of
length $t+1$, thus a factor module of an indecomposable projective
$\Lambda$-module. But by assumption, the indecomposable projective
$\Lambda$-modules have length at most $t$.
	\bigskip

\noindent
{\bf Remarks. 1.} 
Let $Y = \bigoplus Y_i,$ then the subsets $\Hom(C,Y_i)$ of $\Hom(C,Y)$ are actually $\Gamma(C)$-submodules and
there is an isomorphism of $\Gamma(C)$-modules
$\Hom(C,Y) \simeq \bigoplus_i\Hom(C,Y_i)$.
Thus $\eta_{CY}$ maps the lattice ${}^C[\to Y\rangle$ bijectively onto 
the submodule lattice $\mathcal  S(\bigoplus_{i}\Hom(C,Y_i)$.
 The lattice  
$\mathcal  S(\bigoplus_{i}\Hom(C,Y_i)$ contains 
$\prod_i\mathcal  S \Hom(C,Y_i)$ as a sublattice and both have the same height. However, 
$\prod_i\mathcal  S \Hom(C,Y_i)$ may be a proper sublattice of $\mathcal  S(\bigoplus_{i}\Hom(C,Y_i)$,
since isomorphisms of subfactors of the various modules $\Hom(C,Y_i)$ yield diagonals
in $\mathcal  S(\bigoplus_{i}\Hom(C,Y_i)$.
	\medskip 

\noindent
{\bf 2.} 
When dealing with the Auslander bijections $\eta_{CY}: {}^C[\to Y\rangle \to \mathcal  S\Hom(C,Y)$,
we always can assume that $C$ is multiplicity-free and supporting,
here {\it supporting} means that $\Hom(C_i,Y) \neq 0$ for any indecomposable
direct summand $C_i$ of $C$. Namely, let $C'$ be the direct sum  of all indecomposable
direct summands $C_i$ of $C$ with $\Hom(C_i,Y) \neq 0$, one from each isomorphism class.
Then, on the one hand, ${}^C[\to Y\rangle = {}^{C'}[\to Y\rangle$ (since a map $f$ ending in $Y$
is right $C$-determined if and only if it is right $C'$-determined. 
On the other hand, there is an idempotent $e\in \Gamma(C)$ such that $e\Gamma(C)e = \Gamma(C')$
and $e\Hom(C,Y)e = \Hom(C',Y)$, and there is a lattice isomorphism
$\mathcal  S\Hom(C,Y) \to \mathcal  S\Hom(C',Y),$ given by $U \mapsto eU$, where $U$ is a submodule of $\Hom(C,Y)$.
	\medskip

\noindent 
{\bf 3.} Both objects ${}^C[\to Y\rangle$ and $\mathcal  S\Hom(C,Y)$ related by the
Auslander bijection $\eta_{CY}$ concern morphisms ending in $Y$. Of course, 
in Proposition \ref{non-zero-map} we have seen already that all the indecomposable 
direct summands $C'$
of the minimal right determiner $C(f)$ of a map $f: X\to Y$ satisfy $\Hom(C',Y)\neq 0.$

Looking at ${}^C[\to Y\rangle$, we deal with morphisms ending in $Y$ and which are right 
$C$-determined.
Looking at $\Hom(C,Y)$, we deal with maps ending in $Y$ and starting in $C$.
One should be aware that a right minimal map ending in $Y$ and right 
$C$-determined usually will {\bf not} start at $C$, thus the relationship
between the elements of ${}^C[\to Y\rangle$ and the submodules of $\Hom(C,Y)$
is really of interest! Note however that in case we deal with a map $f: C\to Y$
which is right $C$-determined (and starts in $C$), then 
$$
  \eta_{CY}(f) = f\Hom(C,C)
$$
is just the $\Gamma(C)$-submodule of $\Hom(C,Y)$ generated by $f$. 
	\bigskip

We use the next two sections in order to transfer well-known properties
of the lattice of submodules of a finite length module to the right
$C$-factorization lattices, in particular the Jordan-H\"older theorem.
In section \ref{sec:5}, we introduce the right $C$-length of a right $C$-determined
map $f$ ending in $Y$, it corresponds to the the length of the factor module
$\Hom(C,Y)/\eta_{CY}(f).$ In section \ref{sec:6} we will define the $C$-type of $f$
as the dimension vector of $\Hom(C,Y)/\eta_{CY}(f).$

\section{Right $C$-factorizations and right $C$-length.}
\label{sec:5}
	
The Auslander bijection asserts that the lattice ${}^C[\to Y \rangle$
is a modular lattice of finite height, thus 
there is a Jordan-H\"older Theorem for ${}^C[\to Y \rangle$; it can be
obtained from the corresponding Jordan-H\"older Theorem for the submodule lattice
$\mathcal  S\Hom(C,Y)$. In sections 5 and 6, we are going to formulate the assertions for
${}^C[\to Y \rangle$ explicitly. Here we consider composition series of submodules
and factor modules of ${}^C[\to Y \rangle$.
	\medskip 

Let $h_i: X_i\to X_{i-1}$ be maps, where $1\le i \le t,$ with 
composition $f = h_1\dots h_t$.
The sequence $(h_1,h_2,\dots, h_t)$ is called a 
{\it right $C$-factorization of $f$ of length $t$}
provided the maps $h_i$ are non-invertible and the compositions
$f_i = h_1\cdots h_i$ are right minimal and right $C$-determined,
for $1\le i \le t.$
It sometimes may be helpful to deal also with right $C$-factorizations of
length $0$; by definition these are just the identity maps (or, if you prefer, 
the isomorphisms). 
	\medskip

If $(h_1,\dots,h_t)$ is a right $C$-factorization of a map $f$, 
then any integer sequence
$0 = i(0) < i(1) < \cdots < i(s) = t$ defines a sequence of maps
$(h'_1,h'_2,\dots, h'_s)$ with $h'_j = h_{i(j-1)+1}\cdots h_{i(j)}$
for $1\le j \le s$. We use the following lemma inductively, in order to show 
that $(h'_1,h'_2,\dots, h'_s)$ is again a right $C$-factorization of $f$
and we say
that $(h_1,h_2,\dots, h_t)$ is a {\it refinement} of 
$(h'_1,h'_2,\dots, h'_s)$. In particular, any right $C$-factorization
$(h_1,\dots,h_t)$ of $f$ is a refinement of $f$.

\begin{Lem} If $(h_1,\dots,h_t)$ is a right $C$-factorization of length $t\ge 2$, 
then
$$
 (h_1,\dots, h_{i-1},h_ih_{i+1}, h_{i+2},\dots,h_t)
$$ 
is a
right $C$-factorization (of length $t-1$).
\end{Lem}
	
\begin{proof}  We only have to check that
$h_ih_{i+1}$ cannot be invertible. Assume $h_ih_{i+1}$ is invertible.
Then $h_i$ is a split epimorphism. Since $f_i = h_1\cdots h_i$ is
right minimal, it follows that $h_i$ is invertible, a contradiction.
\end{proof} 
	
We say that a right $C$-factorization $(h_1,h_2,\dots, h_t)$ is 
{\it maximal} provided it does not have a refinement of length $t+1$.
	
\begin{Prop}\label{factorization}
 If $(h_1,\dots,h_t)$ is a right $C$-factorization of a map $f$, 
and $f_i = h_1\cdots h_i$ for $0\le i \le t$, then 
$$
  \eta_{CY}(f_t) \subset \cdots \subset \eta_{CY}(f_1) \subset \eta_{CY}(f_0) = \Hom(C,Y)
$$
is a chain of proper inclusions of submodules and any such chain is obtained in this way.
The refinement of right $C$-factorizations corresponds to the refinement of 
submodule chains.
\end{Prop}
	
\begin{proof} 
This is a direct consequence of Auslander's Second Theorem. 
\end{proof} 
	
\begin{Cor} Any right $C$-factorization $(h_1,\dots,h_t)$
has a refinement which is a maximal right $C$-factorization and
all maximal right $C$-factorizations of $(h_1,\dots,h_t)$
have the same length.
\end{Cor} 

\begin{proof} This follows from \ref{factorization} and the Jordan-H\"older theorem. 
\end{proof} 
	
In particular, any right minimal right $C$-determined map $f$ has
a refinement which is a maximal right $C$-factorization, say 
$(h_1,\dots,h_t)$ and its length $t$ will be called the {\it right $C$-length}
of $f$, we write $|f|_C$ for the right $C$-length of $f$.	
There is the following formula:
	
\begin{Prop}\label{length}
Let $f: X \to Y$ be right minimal and right
$C$-determined. Then
\Rahmen{|f|_C = |\Hom(C,Y)| - |\eta_{CY}(f)|,}
where $|\Hom(C,Y)|$ denotes the length of the $\Gamma(C)$-module
$\Hom(C,Y)$ and $|\eta_{CY}(f)|$ the length of its $\Gamma(C)$-submodule
$\eta_{CY}(f)$.
\end{Prop}
	
\noindent 
{\bf The right equivalence class $\eta^{-1}_{CY}(0).$} 
As we have mentioned in section \ref{sec:4}, it is of interest to determine the maps
in the right equivalence class $\eta_{CY}^{-1}(0).$
	
\begin{Prop}\label{zero}
Let $C,Y$ be modules. 
Up to right equivalence, there is a unique right $C$-determined map $f$ ending in $Y$
with $|f|_C$ maximal. The submodule  $\eta_{CY}(f)$ of $\Hom(C,Y)$ is the zero module.
If $f'$ is any right $C$-determined map ending in $Y$, then $f = f'h$ 
for some $h$.
\end{Prop}
	
\begin{proof} The lattice ${}^C[\to Y\rangle$ has a unique zero element, namely $\eta_{CY}^{-1}(0)$.
Let $\eta_{CY}^{-1}(0) = [f\rangle$ for some right minimal map $f$. Then, 
the right $C$-length of $f$ has to be maximal and $[f\rangle \le [f'\rangle$
for any right $C$-determined map $f'$ ending in $Y$. 
\end{proof} 

In general it seems to be quite difficult to describe the maps $f$
such that $[f\rangle = \eta_{CY}^{-1}(0).$ But one should
be aware that such a map $f$ always does exist: any pair $C,Y$ of $\Lambda$-modules
determines uniquely up to right equivalence a map $f$ 
ending in $Y$, namely the right minimal, right $C$-determined map $f$ with $\eta_{CY}(f) = 0.$ 

\begin{Prop}
Let $C,Y$ be modules. The set $\eta_{CY}^{-1}(0)$ is the right equivalence class of the zero map
$0 \to Y$ if and only if $P(\soc Y)$ belongs to $\add C$.
\end{Prop}

\begin{proof} 
This is a direct consequence of Corollary \ref{projective}. 
\end{proof} 
	\bigskip 

\noindent 
{\bf The special case of $C$ being projective.} 
For an arbitrary projective module $C$, there is the following description 
of the right $C$-length of a right minimal, right $C$-determined morphism $f$. 
Here, we denote by 
$[M:S]$ the Jordan-H\"older multiplicity of the simple module $S$ in
the module $M$, this is 
the number of factors in a composition series of $M$ which are isomorphic to 
$S$.	
	
\begin{Prop} Let $C$ be projective. 
The right minimal, right $C$-determined maps $f: X \to Y$ are
up to right equivalence just the inclusion maps of submodules $X$ of $Y$
such that the socle of $Y/X$ is generated by $C$. 

If $f: X \to Y$ is right minimal and right $C$-determined, 
 then $f$ is injective and 
\Rahmen{|f|_C = \sum_{P(S)|C}[\Cok(f):S].}
The minimal element $\eta_{CY}^{-1}(0)$ of ${}^C[\to Y\rangle$
is the inclusion map $X \to Y$, where $X$ is the intersection of the kernels of
all maps $Y \to Q(S)$, where $S$ is a simple module with $P(S)$ a direct summand
of $C$.
\end{Prop}
	
\begin{proof} Let $\mathcal  Q$ be the set of modules $Q(S)$, 
where $S$ is a simple module with $P(S)$ a direct summand
of $C$. Let $X$ be the intersection of the kernels of 
all maps $Y \to Q$ with $Q\in \mathcal  Q.$
Since $Y$ is of finite length, there are finitely many maps
$g_i: Y \to Q(S_i)$ with $Q(S_i)\in \mathcal  Q$, say $1\le i \le m$, such that
$X = \bigcap_{i=1}^m \Ker(g_i).$ Then $Y/X$ embeds into $\bigoplus_{i=1}^m
Q(S_i)$, thus its socle is generated by $C$. It follows that the inclusion map
$X \to Y$ is right $C$-determined. On the other hand, if $X' \to Y$ is
right minimal and right $C$-determined, then it is a monomorphism, thus we can
assume that it is an inclusion map. In addition, we know that the socle of $Y/X'$
is generated by $C$, thus $Y/X'$ embeds into a finite direct sum of modules in $\mathcal  Q$.
It follows that $X'$ is the intersection of some maps 
$Y \to Q$, where $Q\in \mathcal  Q$, thus $X \subseteq X'.$
\end{proof}

There is the following consequence:
{\it The ${}_\Lambda\Lambda$-length of any inclusion map $X \to Y$}
(such a map is obviously right minimal and right $\Lambda$-determined) 
{\it is precisely the length of $Y/X.$}
	\bigskip

For further results concerning the right $C$-length of maps, see section \ref{sec:9}.

\section{The right $C$-type of a right $C$-determined map.}
\label{sec:6}

Recall that we consider
$\Hom(C,Y)$ as a $\Gamma(C)$-module, where $\Gamma(C) = \End(C)^{\text{op}}.$ 
The indecomposable
projective $\Gamma(C)$-modules are of the form $\Hom(C,C_i)$, where $C_i$ is an
indecomposable direct summand of $C$, thus the simple $\Gamma(C)$-modules are of the
form $S(C_0) = \tp \Hom(C,C_0).$ 

Given an artin algebra $\Gamma$, we denote by $K_0(\Gamma)$ 
its Grothendieck group (of all $\Gamma$-modules modulo all exact sequences), it is the
free abelian group with basis the set of isomorphism classes $[S]$ of 
the simple $\Gamma$-modules $S$.
Given a $\Gamma$-module $M$, we denote by $\bdim M$ the corresponding element in $K_0(\Gamma)$,
called the {\it dimension vector} of $M$. Of course, $\bdim M$ can be 
written as an integral linear combination $\bdim M = \sum_{[S]} [M:S][S],$ 
where the coefficient of $[S]$ is just the Jordan-H\"older multiplicity $[M:S]$ of $S$ in $M$.
The elements of $K_0(\Gamma)$ with
non-negative coefficients will be said to be the $\Gamma$-dimension vectors.
If $\mathbf e$ is a $\Gamma$-dimension vector and $M$ is a $\Gamma$-module, we denote
by $\mathcal S_{\mathbf e}M$ the subset of $\mathcal  SM$ consisting of all submodules of $M$ with
dimension vector $\mathbf e.$
	\medskip

Let us return to the artin algebra $\Gamma(C)$, where $C$ is a $\Lambda$-module. 
Thhe Grothendieck group $K_0(\Gamma(C))$ is the free abelian group with basis the set of modules
$S(C_i)$, where $C_i$ runs through a set of representatives of the isomorphism classes
of the indecomposable direct summands $C_i$ of $C$. We are interested here in the
dimension vectors of $\Hom(C,Y)$ and of its factor modules. Actually, 
we want to attach to each right $C$-determined map ending in $Y$ its
right $C$-type $\btype_C(f)$ so that $\btype_C(f) = \bdim \Hom(C,Y)/\eta_{CY}(f).$ 
We start with pairs of neighbors in the right $C$-factorization lattice ${}^C[\to Y\rangle$,
since they correspond under $\eta_{CY}$ to the composition factors of $\Hom(C,Y)$.
	\bigskip

Let $f = hf'$ and $f'$ be right minimal, right $C$-determined maps.
We say that the pair $(f,f')$ is a pair of {\it $C$-neighbors} provided 
$|f|_C = |f'|_C+1.$
Note that the pair $(f,f')$ in $^C[\to Y\rangle$ is a pair 
of $C$-neighbors provided 
$[f\rangle < [f'\rangle$
and there is no $f''$ with $[f\rangle < [f''\rangle < [f'\rangle$
(of course, it is the condition $[f\rangle < [f'\rangle$ 
which implies that there is 
a map $h$ with $f = f'h$). 	
	\bigskip

{\bf Remark.} Let us consider a composition  $f = f'h$, 
where $f'$ both are right minimal and right $C$-determined.
It can happen that $h$ is also right minimal
and right $C$-determined, but $f = f'h$ is not right $C$-determined.
Also it can happen that both maps $f'$ and $f = f'h$ are right minimal and right 
$C$-determined, whereas $h$ is not right $C$-determined. Here are corresponding examples.
	\medskip

\noindent 
{\bf Examples 5.} We consider the path algebra $\Lambda$ of the 
linearly directed quiver $\Delta$ of type $\mathbb A_3$
$$
 a \overset{\alpha}\longrightarrow  b 
 \overset{\beta}\longrightarrow  c\ .
$$
Let $X = Q(a), X' = Q(b), Y = Q(c) = S(c).$ There are non-zero maps
$$
   Q(a) \overset{h}\longrightarrow   Q(b) 
 \overset{f'}\longrightarrow  S(c)\ ,
$$
and we let $f = f'h.$ All three maps $f',h,f$ are surjective and
right minimal. The kernel
of $f'$ is the simple module $S(b)$, the kernel of $h$ is the simple 
module $S(a)$ and the kernel of $f$ is $P(b)$. 

First, let $C = S(b)\oplus S(c)$, thus $\tau C = S(a)\oplus S(b)$ and both
$f'$ and $h$ are right $C$-determined, whereas $f$ is not right
$C$-determined.

Second, let $C = Q(b)\oplus S(c),$ thus $\tau C = P(b)\oplus S(b)$. Then
both $f$ and $f'$ are right $C$-determined, whereas $h$ is not
right $C$-determined. 
	\bigskip 

Let $f: X \to Y$ and $f': X' \to Y$ such that 
$(f,f')$ is a pair of neighbors.
We say that $(f,f')$ is {\it of type} $C_0$ (or better of type $[S(C_0)]$) 
where $C_0$ is
an indecomposable direct summand of $C$, provided there is a map $\phi: C_0 \to X'$
such that $f'\phi$ does not factor through $f$. Such a summand $C_0$ must exist, since
otherwise $f'$ would factor through $f$, due to the fact that $f$ is right 
$C_0$-determined.
The following proposition shows that $C_0$ is uniquely determined.
	\medskip

\begin{Prop} If $(f,f')$ is a 
pair of $C$-neighbors of type $C_0$, then 
$\eta_{CY}(f')/\eta_{CY}(f)$ 
is isomorphic to the simple $\Gamma(C)$-module $S(C_0) = \tp \Hom(C,C_0).$ Thus, 
the type of a pair of $C$-neighbors is well-defined.
\end{Prop}
 
Thus, if $(f,f')$ is a pair of $C$-neighbors of type $C_0$,
we may write $\btype_C(f,f') = [S(C_0)] \in K_0(\Gamma(C)).$ 

\begin{proof}  Let  $\phi: C_0 \to X'$ be a map 
such that $f'\phi$ does not factor through $f$. We obtain a homomorphism
of $\Gamma(C)$-modules
$$
 \Hom(C,f'\phi): \Hom(C,C_0) \to \Hom(C,Y)
$$
which maps into $f'\Hom(C,X')$ (since the image consists of the maps $f'\phi\psi$
with $\psi: C \to C_0$), 
We claim that $\Hom(C,f'\phi)$ does not map
into $f\Hom(C,X)$. Assume, for the contrary, that $\Hom(C,f'\phi)$ maps 
into $f\Hom(C,X)$. Choose $m: C_0\to C$ and $e: C\to C_0$ with $em = 1$. By assumption,
the element $\Hom(C,f'\phi)(e) = f'\phi e$ belongs to $f\Hom(C,X)$, thus there is 
$\phi': C \to X$ with $f'\phi e = f\phi'$ and therefore
$$
   f'\phi = f'\phi em = f\phi'm
$$
shows that $f'\phi$ factor through $f$, a contradiction. 

Thus, the image of $\Hom(C,f'\phi)$ is
a $\Gamma(C)$-submodule of $\Img \Hom(C,f')$ which is not
contained in $\Img(C,f')$ and which is an epimorphic image of the projective module
$\Hom(C,C_0)$. Since we know that  $\Img \Hom(C,f)$ is a maximal submodule of
$\Img \Hom(C,f')$, it follows that
$$
 \Img \Hom(C,f')/\Img \Hom(C,f') \simeq \tp \Img \Hom(C,f'\phi) \simeq \tp \Hom(C,C_0).
$$
This is what we wanted to show. 
\end{proof} 
	
We note the following: {\it If $(f,f')$ is a pair of $C$-neighbors and
$f = f'h$, then $h$ may be neither injective nor surjective.} Let us 
exhibit examples with $f' = 1_Y$.
	\medskip 

\noindent 
{\bf Example 6.} As in the examples 5, let $\Delta$ be the 
linearly directed quiver of type $\mathbb A_3$ and take now as $\Lambda$
the path algebra of $\Delta$ modulo the zero relation $\alpha\beta$. 

Let $Y = P(c)$ and $C = S(b)$, then
$$
 {}^{S(b)}[\to P(c)\rangle \quad \leftrightarrow \quad \mathcal  S\Hom(S(b),P(c))
$$
are lattices with precisely two elements: in ${}^{S(2)}[\to P(3)\rangle$,
there is the right equivalence class of the 
identity map $f' = 1_{P(c)}$ as well as the right equivalence class of any
non-zero map $f: P(b)\to P(c)$.  Note that $f$ is right minimal and 
right $S(b)$-determined, and it is neither mono nor epi.
	\bigskip

Now consider a right minimal right $C$-determined map $f$ ending in $Y$.
As we have mentioned, we want to attach to $f$ an element 
$\btype_C(f) \in K_0(\Gamma(C)).$

\begin{Prop} Let $C,Y$ be $\Lambda$-modules. Let $f$ be right minimal and right $C$-determined
map ending in $Y$ with maximal right $C$-factorization $h_1,\dots,h_t$. Write
$f_i = h_1\cdots h_i,$ for $0\le i \le t$. Then 
\Rahmen{\btype_C f = \sum_{i=1}^t \btype_C(f_i,f_{i-1})}
is a well-defined element of $K_0(\Gamma(C))$ and we have 
\Rahmen{\btype_C(f) = \bdim \Hom(C,Y)/\eta_{CY}(f) = \mathbf d - \bdim \eta_{CY}(f),}
with $\mathbf d = \bdim\Hom(C,Y)$.
\end{Prop}
	
\begin{proof}  Unter the Auslander bijection $\eta_{CY}$, the chain
$$
 [f\rangle = [f_t\rangle <  [f_{t-1}\rangle < \cdots <  [f_0\rangle
$$
is mapped to a chain of submodules
$$
 \eta_{CY}(f) = \eta_{CY}(f_t) \subset \eta_{CY}(f_{t-1}) 
 \subset \cdots \subset \eta_{CY}(f_0) = \Hom(C,Y)
$$
with simple factors $\eta_{CY}(f_{i-1})/\eta_{CY}(f_i)$,
thus we obtain in this way a composition series of $\Hom(C,Y)/\eta_{CY}(f).$
The Jordan-H\"older theorem for $\Hom(C,Y)/\eta_{CY}(f)$ asserts that 
this yields the dimension vector $\bdim \Hom(C,Y)/\eta_{CY}(f)$, independent
of the choice of the composition series. 
\end{proof}

For any $\Gamma(C)$-dimension vector $\mathbf e$, let us denote by ${}^C[\to Y\rangle^{\mathbf e}$
the set of all elements $[f\rangle$ in ${}^C[\to Y\rangle$ such that $\btype_C(f) = \mathbf e.$
Note that ${}^C[\to Y\rangle^{\mathbf e}$ is non-empty only in case $\mathbf e \le \bdim \Hom(C,Y)$,
thus we obtain a decomposition 
\Rahmen{{}^C[\to Y\rangle = \ \bigsqcup_{\mathbf e}\  {}^C[\to Y\rangle^{\mathbf e}}
into a finite number of disjoint subsets. 
	\medskip

\begin{Prop}\label{fixed-dim}
 The Auslander bijection $\eta_{CY}$ yields a bijection
\Rahmen{{}^C[\to Y\rangle^{\mathbf e} \overset{1:1}
\longrightarrow \ \mathcal  S_{\mathbf d-\mathbf e}\Hom(C,Y)}
for every $\Gamma(C)$-dimension vector $\mathbf e$.
\end{Prop}
	\bigskip

Assume now that $k$ is an algebraically closed field and that $\Lambda$ and $\Gamma$ are $k$-algebra. 
If $M$ is a $\Gamma$-module and $\mathbf e$ a dimension vector for $\Gamma$, we write
$\mathbb G_{\mathbf e}M$ instead of $\mathcal  S_{\mathbf e}M$.  
Note that $\mathbb G_{\mathbf e}M$ is in a natural way
an algebraic variety, it is 
called a {\it quiver Grassmannian.}
Namely, all the $\Gamma$-modules with dimension vector $\mathbf e$ have the same $k$-dimension,
say $e$ (if $\mathbf e = \sum_{[S]} e_S[S]$, then $e = \sum_{[S]} e_S\dim_k S$). Denote by 
$\mathbb G_e({}_kM)$ the usual Grassmannian of all
$e$-dimensional subspaces $U$ of the vector space ${}_kM.$ Using Pl\"ucker coordinates one knowns
that $\mathbb G_e({}_kM)$ is a closed subset of a projective space, 
thus $\mathbb G_e({}_kM)$ is
a projective variety. Now $\mathbb G_{\mathbf e} M$ is a subset of $\mathbb G_e({}_kM)$ defined by the
vanishing of some polynomials (which express the fact that we consider submodules $U$ with a fixed
dimension vector), thus also $\mathbb G_{\mathbf e}M$ is an algebraic variety and indeed a projective
variety (but usually not even connected).
	\medskip

Proposition \ref{fixed-dim} can be reformulated as follows:
	\medskip 

\begin{Prop}The Auslander bijection $\eta_{CY}$ yields a bijection
\Rahmen{{}^C[\to Y\rangle^{\mathbf e} \overset{1:1}
\longrightarrow \ \mathbb G_{\mathbf d-\mathbf e}\Hom(C,Y)}
for every $\Gamma(C)$-dimension vector $\mathbf e$.
\end{Prop}
	\medskip

In particular, we see that the set ${}^C[\to Y\rangle^{\mathbf e}$ is 
a projective variety: these {\it Auslander varieties} (as they should be called) furnish
an important tool for studying the right equivalence classes of maps ending in a given module.
As we have mentioned at the end of section \ref{sec:2}, the study of the set of right minimal maps
ending in a fixed module $Y$ can be separated nicely into that of the local symmetries
described by the right automorphism groups and that of 
the global directedness given by the right factorization
lattice. Auslander's first theorem describes the right factorization lattice as the filtered
union of the right $C$-factorization lattices, and, as we now see, these right $C$-factorization
lattices are finite disjoint unions of (transversal) subsets which are 
projective varieties, the Auslander varieties.
	\bigskip 

\noindent 
{\bf Remark.} It seems that quiver Grassmannians first have been studied
by Schofield \cite{[Sc]} and Crawley-Boevey \cite{[CB]} in order to deal with
generic properties of quiver representations. In 2006, Caldero and
Chapoton \cite{[CC]} observed that quiver Grassmannians can be used
effectively in order to analyse the structure of cluster algebras
as introduced by Fomin and Zelevinsky. Namely, it turns out that 
cluster variables can be described using the Euler
characteristic of quiver Grassmannians. In this way 
quiver Grassmannians are now an indispensable tool for studying
cluster algebras and quantum cluster algebras.
We should add that quiver Grassmannians were also used (at least
implicitly) in the study of quantum groups, see for example
the calculation of Hall polynomials in \cite{[R5]}. 
A large number of papers is presently devoted to special properties of 
quiver Grassmannians. 

There is the famous assertion that any projective variety is a quiver
Grassmannian, see the paper \cite{[Re]} by Reineke (answering in this way a
question by Keller)
as well as blogs by Le Brujn \cite{[L]} (with a contribution by Van den Bergh) and
by Baez \cite{[B]}.
Actually, the construction as proposed by Van den
Bergh in Le Bruyn's blog 
is much older, it has been mentioned explicitly 
already in 1996 by Hille \cite{[Hi]} dealing with moduli spaces
of thin representations  (see the example at the end of that paper), 
and it can be traced back to earlier considerations
of Huisgen-Zimmermann dealing with moduli spaces
of serial modules, even if they
were published only later 
(see \cite{[Hu]}, Theorem G, but also \cite{[BH]}, Corollary B,
and \cite{[DHW]}, Example 5.4).
It follows from Proposition \ref{wild} above
that given a strictly wild algebra $\Lambda$
and any projective variety $V$, 
there are $\Lambda$-modules $C,Y$ and a dimension vector $\mathbf e$ such
that ${}^C[\to Y\rangle^{\mathbf e}$ is isomorphic to  $V$. 
We will show in \cite{[R10]} that this holds true for all controlled wild algebras.

\section{Maps of right $C$-length 1.}
\label{sec:7}

By definition, a right minimal right $C$-determined map $f$ has right $C$-length 1
provided $f$ is not invertible and given any factorization $f = f'h$ 
with $f'$ right minimal right $C$-determined, then one of the maps $f', h$ is invertible.
Let us denote by ${}^C[\to Y\rangle^1$ the set of 
right equivalence classes of the maps ending in $Y$
which have right $C$-length 1. 

Warning: An irreducible map $f$ is of course right minimal, but if $f$ is irreducible and 
right $C$-determined, we may have $|f|_C > 1.$ For example, consider the Kronecker quiver, 
take $C = {}_\Lambda\Lambda$. The irreducible map $f: P_0 \to P_1$ has $|f|_C = 2$ (note 
the factorizazion $P_0 \subset \rad P_1 \subset P_1$).
	\bigskip

Here is an immediate consequence of Proposition \ref{length}. 

\begin{Cor} Let $f: X \to Y$ be right minimal and right
$C$-determined. Then $|f|_C = 1$ if and only if 
$\eta_{CY}(f)$ is a maximal $\Gamma(C)$-submodule of $\Hom(C,Y)$.
\end{Cor}
	
If we denote by $\mathcal  S_m\Hom(C,Y)$ the set of maximal
submodules of $\Hom(C,Y)$, then 
the restriction of $\eta_{CY}$ furnishes a bijection
\Rahmen{{}^C[\to Y\rangle^1\ \overset{1:1}
\longrightarrow \ \mathcal  S_m\Hom(C,Y).}
	\bigskip

In order to analyze maps of right $C$-length 1, we will need the following lemma.

\begin{Lem}\label{induced}
Assume that $f: X \to Y$ and $f': X'\to Y$ are
epimorphisms  
with $f = f'h$. Then we have the following commutative diagram with
exact rows:
$$
{\beginpicture
\setcoordinatesystem units <1.7cm,1.3cm>
\arr{-.75 1}{-.25 1}
\arr{0.25 1}{0.75 1}
\arr{1.25 1}{1.75 1}
\arr{2.25 1}{2.75 1}

\arr{-.75 0}{-.25 0}
\arr{0.25 0}{0.75 0}
\arr{1.25 0}{1.75 0}
\arr{2.25 0}{2.75 0}

\arr{0 .3}{0 .7}
\arr{1 .3}{1 .7}
\plot 2.03 0.7  2.03 0.3 /
\plot 1.97 0.7  1.97 0.3 /

\multiput{$0$} at -1 0  -1 1  3 0  3 1 /

\put{$K'$} at 0 1
\put{$X'$} at 1 1
\put{$Y$} at 2 1
\put{$K$} at 0 0
\put{$X$} at 1 0
\put{$Y$} at 2 0

\put{$\ssize f'$} at 1.5 1.2
\put{$\ssize f$} at 1.5 .2
\put{$\ssize h'$} at .15 .5
\put{$\ssize h$} at 1.15 .5
\endpicture}
$$
If $f$ is right minimal and $h'$ is a split epimorphism, 
then also $f'$ is right minimal.
\end{Lem}

{\bf Remark.} Observe that it is not enough to assume that $h'$ is an epimorphism. As
an example, take 
the indecomposable injective Kronecker module  $X = Q_1$ 
of length $3$, let $K$ be a submodule of length 2, and $K' = K/\soc.$
Then $Q_1 \to Q_1/K$ is right minimal. But the induced sequence is just the short exact
sequence $K/\soc \to Q_1/\soc \to Y$ which splits.

\begin{proof}
Denote the kernel of $h'$ by $K''$, thus we can assume
that $K = K'\oplus K''$ such that $h'$ is the canonical projection $K \to K'$ with
kernel $K''$. 
Assume that $X' = U\oplus V$, where $U$ is contained in the kernel of $f'$, thus
$U \subseteq K'.$ Since $X' = X/K''$, there are submodules $U', V'$ of $X'$
both containing $K''$ such that $U'+V' = X'$ and $U'\cap V' = K''$, with
$U = U'/K''$ and $V = V'/K''.$

Consider $U'' = U'\cap K'$, this is a submodule of the kernel $K$ of $f$.
Also, $U' = K''+ U''$ (using the modular law). Thus we have
$U''\cap V' = U'\cap K' \cap V' \subseteq K'\cap K'' = 0$ and
$U''+V' = U''+V'+K'' = U'+V' = X$. This shows that $U''$ is a direct summand of $X$
which is contained in the kernel of $f$. Since $f$ is right minimal, we see that
$U'' = 0.$ Since $U' = K'' + U'' = K''$, it follows that $U = 0.$
\end{proof} 

\begin{Cor}\label{indecomposable-kernel}
Let $C$ be a module. Let $f: X \to Y$ be a right minimal
right $C$-determined  epimorphism with $|f|_C = 1$. Then
the kernel of $f$ is indecomposable.
\end{Cor}
	
\begin{proof}
Let $K = \tau C.$ The kernel of $f$ has to be non-zero, thus assume it is
decomposable, say equal to $K_1\oplus K_2$ with non-zero modules $K_1,K_2\in  \add K.$ 
The canonical projection $p_1: K_1\oplus K_2
\to K_1$ yields a commutative diagram with exact rows:
$$
{\beginpicture
\setcoordinatesystem units <2.2cm,1.3cm>
\arr{-.75 1}{-.4 1}
\arr{0.4 1}{0.75 1}
\arr{1.25 1}{1.75 1}
\arr{2.25 1}{2.75 1}

\arr{-.75 0}{-.25 0}
\arr{0.25 0}{0.75 0}
\arr{1.25 0}{1.75 0}
\arr{2.25 0}{2.75 0}

\arr{0 .7}{0 .3}
\arr{1 .7}{1 .3}
\plot 2.03 0.7  2.03 0.3 /
\plot 1.97 0.7  1.97 0.3 /

\multiput{$0$} at -1 0  -1 1  3 0  3 1 /

\put{$K_1\oplus K_2$} at 0 1
\put{$X$} at 1 1
\put{$Y$} at 2 1
\put{$K_1$} at 0 0
\put{$X'$} at 1 0
\put{$Y$} at 2 0

\put{$\ssize f$} at 1.5 1.2
\put{$\ssize f'$} at 1.5 .2
\put{$\ssize p_1$} at .15 .5
\put{$\ssize p'_1$} at 1.15 .5
\endpicture}
$$
Since $f$ is right minimal and $p_1$ is a split epimorphism, lemma \ref{induced} asserts
that $f'$ is right minimal. Since $f'$ is also right $C$-determined, we see that
$|f|_C \ge 2,$ a contradiction. 
\end{proof} 

\begin{Prop}\label{length-1}
Let $C$ be indecomposable and non-projective and let
$K = \tau C$. If
$$
 \epsilon:\qquad  0 \longrightarrow K^t \longrightarrow X 
 \overset f \longrightarrow Y \longrightarrow  0
$$
is an exact sequence, then $[f\rangle$ belongs to ${}^C[\to Y\rangle^1$ if and
only if $t=1$ and the equivalence class $[\epsilon]$ is a non-zero element of the 
$\Gamma(K)$-socle of $\Ext^1(Y,K).$
\end{Prop}

\begin{proof}  First, assume that $[f\rangle$ belongs to ${}^C[\to Y\rangle^1$. 
According to \ref{indecomposable-kernel}, 
we must have $t=1.$ Also, since $f$ is not split epi, we see that 
$[\epsilon]$ is a non-zero
element of $\Ext^1(Y,K)$. Assume that $[\epsilon]$ does not belong to the
$\Gamma(K)$-socle of $\Ext^1(Y,K)$. Then there is a nilpotent endomorphism $\phi$ 
of $K$ such that the induced exact sequence does not split. Thus there is
a commutative diagram with exact rows
$$
{\beginpicture
\setcoordinatesystem units <1.7cm,1.3cm>
\arr{-.75 1}{-.25 1}
\arr{0.25 1}{0.75 1}
\arr{1.25 1}{1.75 1}
\arr{2.25 1}{2.75 1}

\arr{-.75 0}{-.25 0}
\arr{0.25 0}{0.75 0}
\arr{1.25 0}{1.75 0}
\arr{2.25 0}{2.75 0}

\arr{0 .7}{0 .3}
\arr{1 .7}{1 .3}
\plot 2.03 0.7  2.03 0.3 /
\plot 1.97 0.7  1.97 0.3 /

\multiput{$0$} at -1 0  -1 1  3 0  3 1 /

\put{$K$} at 0 1
\put{$X$} at 1 1
\put{$Y$} at 2 1
\put{$K$} at 0 0
\put{$X'$} at 1 0
\put{$Y$} at 2 0

\put{$\ssize f$} at 1.5 1.2
\put{$\ssize f'$} at 1.5 .2
\put{$\ssize \phi$} at .15 .5
\put{$\ssize \phi'$} at 1.15 .5
\endpicture}
$$
with the upper row being the sequence $\epsilon$. Since the lower sequence
does not split, the map $f'$ is right minimal. The kernel shows that $f'$
is also $C$-determined. Since $\phi'$ is not invertible, 
the factorization $f = f'\phi'$  shows that $(\phi',f')$ is a $C$-factorization
of $f$ of length at least $2$, thus $|f|_C \ge 2,$ a contradiction. This shows
that $[\epsilon]$ belongs to the $\Gamma(K)$-socle of $\Ext^1(Y,K).$

Conversely, assume that $t=1$ and $[\epsilon]$ is a non-zero element of the
$\Gamma(K)$-socle of $\Ext^1(Y,K).$ Now $f$ is right minimal, right $C$-determined,
and not an isomorphism, thus $|f|_C \ge 1.$ Assume that $|f|_C \ge 2,$
thus there is a $C$-factorization $(h_1,h_2)$ of $f$. Since $f = h_2h_1$ is surjective,
also $h_2$ is surjective. Since $h_2$ is right minimal, right $C$-determined and
not invertible, its kernel has to be of the form $K^t$ for some $t\ge 1$.
Thus there is a commutative diagram with exact rows of the following form
$$
{\beginpicture
\setcoordinatesystem units <1.8cm,1.3cm>
\arr{-.75 1}{-.4 1}
\arr{0.4 1}{0.75 1}
\arr{1.25 1}{1.75 1}
\arr{2.25 1}{2.75 1}

\arr{-.75 0}{-.25 0}
\arr{0.25 0}{0.75 0}
\arr{1.25 0}{1.75 0}
\arr{2.25 0}{2.75 0}

\arr{0 .7}{0 .3}
\arr{1 .7}{1 .3}
\plot 2.03 0.7  2.03 0.3 /
\plot 1.97 0.7  1.97 0.3 /

\multiput{$0$} at -1 0  -1 1  3 0  3 1 /

\put{$K$} at 0 1
\put{$X$} at 1 1
\put{$Y$} at 2 1
\put{$K'$} at 0 0
\put{$X'$} at 1 0
\put{$Y$} at 2 0

\put{$\ssize f$} at 1.5 1.2
\put{$\ssize h_2$} at 1.5 .2
\put{$\ssize h_1'$} at .15 .5
\put{$\ssize h_1$} at 1.15 .5
\endpicture}
$$
where again the upper row is $\epsilon$. 
Write $h'_1 = (\phi_1,\dots,\phi_t)$ with endomorphisms $\phi_i: K \to K$.
If all $\phi_i$ belong to the radical of $\Gamma(K)$, then all the sequences
induced from $\epsilon$ by the maps $\phi_i$ split, thus also the lower sequence
splits, since it is induced from $\epsilon$ by $h_1'.$ Thus, at least one of the
maps $\phi_i$ has to be invertible and therefore $h_1'$ is a split monomorphism.

If $t > 1$, then the lower sequence splits off a sequence 
$0 \to K \overset 1  \rightarrow K \to 0 \to 0$,
but this means that $h_2$ is not right minimal. Thus $t=1.$ But then $h_1'$ is
an automorphism, thus $h_1$ is invertible, a contradiction. This shows that 
$|f|_C = 1.$
\end{proof} 

We say that an epimorphism $f$ is {\it epi-irreducible,} provided for any
factorization $f = f'f''$ with $f''$ a proper epimorphism, the map $f'$ is a split
epimorphism (the dual concept of mono-irreducible maps has been considered in \cite{[R6]}).
	
\begin{Prop}
Let $C$ be indecomposable, non-projective and let 
$K = \tau C$. If $f: X \to Y$ is an epi-irreducible epimorphism with kernel $K = \tau C$, 
then  $f$ belongs to ${}^C[\to Y\rangle^1$.
\end{Prop}

\begin{proof}
According to \ref{length-1}, we have to show that the given exact sequence 
$$
 \epsilon:\qquad  0 \longrightarrow K \longrightarrow X 
 \overset f \longrightarrow Y \longrightarrow 0
$$
belongs to the $\Gamma(K)$-socle of $\Ext^1(Y,K)$. Thus, let $\phi$ be a
non-invertible endomorphism of $K$, write it in the form $\phi=\phi'\phi''$
with $\phi''$ epi and $\phi'$ mono. Consider the exact sequence induced from
$\epsilon$ by $\phi''$
$$
{\beginpicture
\setcoordinatesystem units <1.7cm,1.3cm>
\arr{-.75 1}{-.25 1}
\arr{0.25 1}{0.75 1}
\arr{1.25 1}{1.75 1}
\arr{2.25 1}{2.75 1}

\arr{-.75 0}{-.25 0}
\arr{0.25 0}{0.75 0}
\arr{1.25 0}{1.75 0}
\arr{2.25 0}{2.75 0}

\arr{0 .7}{0 .3}
\arr{1 .7}{1 .3}
\plot 2.03 0.7  2.03 0.3 /
\plot 1.97 0.7  1.97 0.3 /

\multiput{$0$} at -1 0  -1 1  3 0  3 1 /

\put{$K$} at 0 1
\put{$X$} at 1 1
\put{$Y$} at 2 1
\put{$K$} at 0 0
\put{$X'$} at 1 0
\put{$Y$} at 2 0

\put{$\ssize f$} at 1.5 1.2
\put{$\ssize f'$} at 1.5 .2
\put{$\ssize \phi''$} at .15 .5
\put{$\ssize f''$} at 1.15 .5
\endpicture}
$$
Since $\phi''$ is a proper epimorphism, also $f''$ is a proper epimorphism,
thus $f'$ is a split epimorphism. But this implies that also the 
exact sequence induced from
$\epsilon$ by $\phi$ splits. 
\end{proof} 

We will need some basic facts concerning the Gabriel-Roiter measure of
finite length modules, see \cite{[R6]}. The Gabriel-Roiter measure of a module $M$
will be denoted by $\gamma(M).$
We recall that any indecomposable module $M$
which is not simple has a Gabriel-Roiter submodule $M'$, this is a certain indecomposable
submodule of $M$ and the embedding $M' \to M$
is called a {\it Gabriel-Roiter inclusion.} Recall that a Gabriel-Roiter inclusion
$M' \to M$ is mono-irreducible: this means that for any proper submodule $M''$ of $M$
with $M' \subseteq M''$, the inclusion $M' \subseteq M''$ splits. 
As a consequence, given any nilpotent endomorphism $f$ of $M/M'$, the sequence
induced from $0 \to M' \to M \to M/M' \to 0$ using $f$ splits. Also, it follows that
the cokernel $M/M'$ of a Gabriel-Roiter inclusion is indecomposable (and not projective).

Of course, we may use duality
and consider a Gabriel-Roiter submodule $U$ of $DM$, the corresponding 
projection $M = D^2M \to DU$ will be called a {\it co-Gabriel-Roiter projection.}
By duality, a co-Gabriel-Roiter projection is an epi-irreducible epimorphism
with (non-injective) indecomposable kernel. 

\begin{Cor}\label{co-GR}
Let $M$ be an indecomposable module which is not simple
and let $f: M \to Y$ be a co-Gabriel-Roiter projection, say with kernel $K$. 
Let $C = \tau^{-}K.$ Then $f$ is right minimal, right $C$-determined and
$|f|_C = 1,$ thus $[f\rangle$ 
belongs to ${}^C[\to Y\rangle^1$.
\end{Cor}

\noindent 
{\bf Remark.} If $M$ is indecomposable and not simple, we also may consider
a Gabriel-Roiter submodule $U$ of $M$, say with projection $p: M \to M/U = Y'$
and consider $C' = \tau^{-}U.$ Then $p$ is right minimal and right $C'$-determined,
however in general there is not a fixed number $t$ such that
$[p\rangle$ belongs to ${}^{C'}[\to Y'\rangle^t$ or to 
${}^{C'}[\to Y'\rangle_t$. A typical example is example 8 presented in the
next section. The two modules $P(b)$ and $\tau^{-}S(a)$ both have 
$P(a)$
as a Gabriel-Roiter submodule with factor module $S(b)$. 
Let $C' =  \tau^{-}P(a).$ The projection
$P(b) \to S(b)$ belongs to ${}^{C'}[\to S(b)\rangle^2 = 
{}^{C'}[\to S(b)\rangle_0$ whereas the projection
$\tau^{-}S(a) \to S(b)$ belongs to ${}^{C'}[\to S(b)\rangle^1 =
{}^{C'}[\to S(b)\rangle_1.$   

\section{Epimorphisms in $[\to Y\rangle$.}
\label{sec:8}

The set of right equivalence classes $[f\rangle$ where $f$ is an epimorphism
is obviously a coideal of the lattice $[\to Y\rangle,$ we denote it
by $[\to Y\rangle_{\text{epi}}$. Since the pullback of an epimorphism is again
an epimorphism, we see that $[\to Y\rangle_{\text{epi}}$ is closed under meets.
Also, for any module $C$, the subset ${}^C[\to Y\rangle_{\text{epi}}$ of 
${}^C[\to Y\rangle$ consisting of the right equivalence classes of all
right $C$-determined epimorphisms ending in $Y$ is a coideal which is closed
under meets. Since ${}^C[\to Y\rangle$ is a lattice of finite height, we
see that ${}^C[\to Y\rangle_{\text{epi}}$ has a unique minimal element, say
$[f_0\rangle$ and our first aim will be to describe $\eta_{CY}(f_0).$
	\bigskip 

Before we deal with this question, let us point out in which way
the projectivity or non-projectivity of indecomposable direct summands of $C$
are related to the fact that right minimal right $C$-determined morphisms $f$
are mono or epi. 
If $f$ is a monomorphism, then $f$ is right minimal and right
$\Lambda$-determined
(see \ref{monos}), 
thus right $C$-determined for some projective module $C$. 
Conversely, if $C$ is projective, then any right minimal,
right $C$-determined morphism is a monomorphism (see \ref{projective}). Namely, if $K$
is an indecomposable  direct summand of the kernel of $f$, 
where $f$ is right minimal, then $K$ is not injective and
$\tau^{-}K$ is a direct summand of any module $C$ such that $f$ is right
$C$-determined. Of course, since $K$ is not injective, $\tau^{-}K$
is an indecomposable non-projective module. One should be aware that a 
morphism may be right $C$-determined for some module $C$ without any indecomposable
projective direct summand, without being surjective. 
	\medskip 

\noindent 
{\bf Example 7.} As in example 6, we take as $\Lambda$ 
the path algebra of the linearly directed quiver $\Delta$ of type $\mathbb A_3$, 
modulo the zero relation $\alpha\beta$.
Again, let $Y = P(c)$ and $C = S(b)$. As we have mentioned already, the
non-zero maps
$f: P(b) \to P(c)$ are not surjective, but right $S(b)$-determined, and, of course, $S(b)$
is not projective. (As we will see below, it is essential for this feature
that the kernel of $f$ has injective dimension at least 2; for a general discussion
of maps which are not surjective, but right $C$-determined by a 
module $C$ without any
indecomposable projective direct summands, we refer to \cite{[R9]}).
		\bigskip 

\noindent
{\bf The submodule $\Hom(C,\mathcal  P,Y)$ of $\Hom(C,Y).$} We denote by 
$\Hom(C,\mathcal  P,Y)$ the set of morphisms $C \to Y$ which factor
through a projective module. Note that $\Hom(C,\mathcal  P,Y)$ is a $\Gamma(C)$-submodule of
$\Hom(C,Y)$.
	
\begin{Prop}\label{surjective}
Assume that $f: X \to Y$ is right 
$C$-determined.
Then $f$ is surjective if and only if $\eta_{CY}(f) \supseteq \Hom(C,\mathcal  P,Y)$.
\end{Prop}

\begin{proof} One direction is a trivial verification: 
First assume that $f$ is surjective. Let $h$ belong to $\Hom(C,\mathcal  P,Y)$,
thus $h = h_2h_1$ where $h_1: C \to P)$ and $h_2: P \to Y$ with $P$ projective.
Since $f$ is surjective and $P$ is projective, there is $h_2': P \to X$
such that $h_2 = fh_2'.$ Thus shows that $h = h_2h_1 = fh_2'h_1$ belongs to
$f\Hom(C,X) = \eta_{CY}(f).$

The converse is more interesting, here we have to use that $f$ is right
$C$-determined. We assume that $\eta_{CY}(f) \supseteq \Hom(C,\mathcal  P,Y)$.
Let $p: P(Y) \to Y$ be a projective cover of $Y$.
Consider an arbitrary morphism $\phi: C \to P(Y)$. The composition
$p\phi$ belongs to $\Hom(C,\mathcal  P,Y)$, thus to $\eta_{CY}(f) = f\Hom(C,X)$. 
Since $f$ is right $C$-determined, it follows that $p$ itself
factors through $f$, say $p = fp'$ for some $p': P(Y) \to X$. Now the
composition $fp' = p$ is surjective, there $f$ has to be surjective.
\end{proof} 	

Let us denote by ${}^C[\to Y\rangle_{\text{epi}}$ the subset of ${}^C[\to Y\rangle$
given by all elements $[f\rangle$ with $f$ an epimorphism. 
	
\begin{Prop} The restriction of the Auslander bijection 
$\eta_{CY}$ yields a poset isomorphism  	
$${}^C[\to Y\rangle_{\text{epi}}\quad \longrightarrow \quad
 \mathcal  S(\Hom(C,Y)/\Hom(C,\mathcal  P,Y)) = \mathcal  S \underline{\Hom}(C,Y)
$$
such that the following diagram commutes:
\Rahmen{\beginpicture
\setcoordinatesystem units <1cm,1.2cm>
\put{${}^C[\to Y\rangle_{\text{epi}}$} at 0 1
\put{$\mathcal  S \underline{\Hom}(C,Y)$} at 4 1
\put{${}^C[\to Y\rangle$} at 0 0 
\put{$ \mathcal  S\Hom(C,Y))$} at 4 0
\arr{1.2 0}{2.5 0}
\arr{1.2 1}{2.5 1}
\arr{0 0.7}{0 0.3}
\arr{4 0.7}{4 0.3}
\put{$\eta_{CY}$} at 1.9 0.15
\endpicture}
Here, the vertical maps are the canonical inclusions.
\end{Prop}
	
\begin{proof}
It is well-known that given a module $M$ and a submodule $M'$, then
the lattice of submodules of the factor module $M/M'$
is canonically isomorphic to the lattice of the submodules $U$ of $M$ satisfying
$M' \subseteq U.$ This is the vertical map on the right. 
\end{proof}

More generally, dealing with a morphism $f$ which is right 
$C$-determined, 
we can recover the image of $f$ as follows:
	
\begin{Prop} Let $f: X \to Y$ be right $C$-determined. Then one recovers the image
of $f$ as the largest submodule $Y'$ of $Y$ (with inclusion map $u: Y' \to Y$)
such that $u\Hom(C,\mathcal  P,Y') \subseteq f\Hom(C,X).$
\end{Prop}
	
\begin{proof} 
Let $Y'$ be the image of $f$ with inclusion map $u$ and $uf' = f$
(with $f'$ surjective).
First of all, we show that $u\Hom(C,\mathcal  P,Y') \subseteq f\Hom(C,X).$
Let $\phi': C \to \Lambda$ and $\phi'': \Lambda \to Y'$ (the maps $\phi''\phi'$
obtained in this way generate $\Hom(C,\mathcal  P,Y')$ additively). We want to show
that $u\phi''\phi$ factors through $f$. Since $f': X \to Y'$ is surjective,
there is $\psi: \Lambda \to X$ such that $\phi'' = f'\psi$ (since $\Lambda$ is
projective). Thus $u\phi''\phi' = uf'\psi\phi' = f\psi\phi'.$ Thus
$u\phi''\phi'$ factors through $f.$

On the other hand, let $u'': Y'' \to Y$ be a submodule of $Y$ such that
$u''\Hom(C,\mathcal  P,Y'') \subseteq f\Hom(C,X)$. Let $p: P(Y'') \to Y''$ be a projective cover.
Consider the map $f' = u''p:  P(Y'') \to Y.$ It has the property that for all
maps $\phi: C \to P(Y'')$ the composition $f'\phi$ factors through $f$
(namely $f'\phi = u''p\phi$ belongs to  $u''\Hom(C,\mathcal  P,Y'') \subseteq f\Hom(C,X)$).
But $f$ is right $C$-determined, thus we conclude that $f'$ factors through
$\alpha,$ say $f' = f\phi'$ for some $\phi': C \to P(Y'').$ Thus the image $Y''$
of $f'$ is contained in the image $Y'$ of $f$. This is what we wanted to prove.
\end{proof}

We recover in this way Proposition \ref{surjective}. 
Namely, if $f$ is surjective, then $Y$ is the image
of $f$, thus $Y$ is one of the submodule $Y'$ with $u\Hom(C,\mathcal  P,Y') \subseteq f\Hom(C,X)$,
thus $\Hom(C,\mathcal  P,Y) \subseteq f\Hom(C,Y).$ 

Conversely, if $\Hom(C,\mathcal  P,Y) \subseteq f\Hom(C,Y),$
then $Y$ is one of the submodules $Y'$ with $u\Hom(C,\mathcal  P,Y') \subseteq f\Hom(C,X)$ and 
therefore the image of $f$ contains $Y$, thus is equal to $Y$. This shows:
{\it If $f$ is right $C$-determined, then $f$ is surjective if
and only if $\Hom(C,\mathcal  P,Y) \subseteq f\Hom(C,X).$}
	
\begin{Cor}\label{epi}
Let $C, Y$ be modules.
\item{\rm(a)} All maps in $^C[\to Y\rangle$ are epimorphisms
if and only of $\Hom(C,\mathcal  P,Y) = 0.$
\item{\rm(b)} $\Hom(C,\mathcal  P,Y) = \Hom(C,Y)$ if and only if the only element
$[f\rangle$ in ${}^C[\to Y\rangle$ with $f$ surjective is the
right equivalence class of the identity map $Y \to Y,$
if and only if any surjective map ending in $Y$ with kernel in $\add \tau C$ splits. 
\end{Cor}
	
\noindent 
{\bf Kernels with injective dimension at most $1$.}
	
\begin{Prop}\label{inj-dim-1}
Let $K$ be a module and $C = \tau^{-}K$. The following
conditions are equivalent:
\item{\rm(i)} The injective dimension of $K$ is at most $1$.
\item{\rm (ii)} If $Y$ is any module, then all maps in 
   ${}^C[\to Y\rangle$ are epimorphisms.
\item{\rm(iii)} We have $\Hom(C,\mathcal  P,Y) = 0$ for all modules $Y$.
\end{Prop} 
	
\begin{proof}  Recall from \cite{[R4]}, 2.4 that $K$ has injective dimension at most $1$ if and
only if $\Hom(C,\Lambda) = 0.$ Thus, if $K$ has injective dimension at most $1$
and $Y$ is an arbitrary module, then $\Hom(C,\mathcal  P,Y) = 0,$ this shows that (i) implies (iii).
Conversely, assume the condition (iii), thus $\Hom(C,\mathcal  P,Y) = 0$ for all modules $Y$. 
If the injective dimension of $K$ would be at least $2$, then $\Hom(C,\Lambda) \neq 0.$
But $\Hom(C,\mathcal  P,\Lambda) = \Hom(C,\Lambda)$. This contradiction shows that (iii) implies
(i). For the equivalence of (ii) and (iii) see \ref{epi}(a). 
\end{proof} 
	
\begin{Cor} Let $\Lambda$ be hereditary and $C$ a module without any
indecomposable projective direct summand. Then any right $C$-determined  
morphism $f: X \to Y$ is an epimorphism.
\end{Cor}

\begin{proof} Let $K = \tau C$. Since $C$ has no indecomposable projective direct
summand, it follows that $C = \tau^{-} K.$ Since $\Lambda$ is hereditary,
the injective dimension of any module is at most $1$. Since the injective dimension
of $K$ is at most $1$, it follows from the proposition that all 
right $C$-determined maps are epimorphisms.
\end{proof}

Of course, we also can show directly that $\Hom(C,\mathcal  P,Y) = 0.$ 
Namely, let $g: C \to Y$ be in $\Hom(C,\mathcal  P,Y)$.
Then $g = g_2g_1$ with $g_1: C \to P$, where $P$ is a projective module. 
The image $P'$ of $g_1$ is a submodule of $P$, thus, since 
$\Lambda$ is hereditary, the module $P'$ is also projective. Thus, we have a
surjective map $C \to P'$ with $P'$ projective. Such a map splits. This shows that
$P'$ is isomorphic to a direct summand of $C$. It follows that $P' = 0$ and therefore
$g = 0.$ 
	\bigskip

\noindent 
{\bf Riedtmann-Zwara degenerations.}
Recall that $M'$ is a Riedtmann-Zwara degeneration
of $M$ if and only if there is an exact sequence of the form
$$
 0 \to K \to K\oplus M \to M' \to 0,
$$
or, equivalently, if and only if there is an exact sequence of the form
$$
 0 \to M' \to M\oplus L \to L \to 0;
$$
(in both sequences we can assume that the maps $K \to K$ and $L \to L$,
respectively, are in the radical).

In terms of the Auslander bijection, we may deal with these
data in several different ways: namely, we may look at the right
equivalence classes of both $[M \to M'\rangle$ and $[K\oplus M \to M'\rangle$ 
in $[\to M'\rangle$ as well as at the right equivalence class 
$[M' \to M\rangle$ in $[\to M\rangle.$ 
In case we deal with $[K\oplus M \to M'\rangle$, one should be aware that
this map $[K\oplus M \to M'\rangle$ is the join of the two maps
$[K \to M'\rangle$ and $[M \to M'\rangle$ in $[\to M'\rangle$. 
	
In addition, we also may concentrate on the possible maps $K \to K$ and
$L \to L$ (sometimes called steering maps). 
	\medskip

When dealing with epimorphisms in ${}^C[\to Y\rangle$, Riedtmann-Zwara
degenerations play a decisive role, as the following proposition shows:
	
\begin{Prop}
{\it Let $f: X \to Y$ and $f': X' \to Y$ be epimorphisms
with isomorphic kernels. If $[f\rangle \le [f'\rangle$, then
$X'$ is a Riedtmann-Zwara degeneration of $X$.}
\end{Prop}

\begin{proof} 
Let $h: X \to X'$ with $f = f'h.$ Let $u: K \to X$ and
$u': K \to X'$ be the kernel maps. Since $f = f'h$, there is $h': K \to K$
such that $u'h' = hu,$ thus we deal with the following commutative diagram with
exact sequences:
$$
{\beginpicture
\setcoordinatesystem units <1.7cm,1.3cm>
\arr{-.75 1}{-.25 1}
\arr{0.25 1}{0.75 1}
\arr{1.25 1}{1.75 1}
\arr{2.25 1}{2.75 1}

\arr{-.75 0}{-.25 0}
\arr{0.25 0}{0.75 0}
\arr{1.25 0}{1.75 0}
\arr{2.25 0}{2.75 0}

\arr{0 .7}{0 .3}
\arr{1 .7}{1 .3}
\plot 2.03 0.7  2.03 0.3 /
\plot 1.97 0.7  1.97 0.3 /

\multiput{$0$} at -1 0  -1 1  3 0  3 1 /

\put{$K$} at 0 1
\put{$X$} at 1 1
\put{$Y$} at 2 1
\put{$K$} at 0 0
\put{$X$} at 1 0
\put{$Y$} at 2 0

\put{$\ssize u$} at 0.5 1.2
\put{$\ssize u'$} at 0.5 .2
\put{$\ssize f$} at 1.5 1.2
\put{$\ssize f'$} at 1.5 .2
\put{$\ssize h'$} at .15 .5
\put{$\ssize h$} at 1.15 .5
\endpicture}
$$
The diagram shows that the lower exact sequence is induced from the upper one
by $h'$. But this means that the following sequence is exact:
$$
{\beginpicture
\setcoordinatesystem units <3cm,1.3cm>
\put{$0$} at 0.3 0
\put{$K$} at 1 0
\put{$K\oplus X$} at 2 0
\put{$X'$} at 3 0
\put{$0$} at 3.7 0
\arr{0.5 0}{0.8 0}
\arr{1.15 0}{1.7 0}
\arr{2.3 0}{2.85 0}
\arr{3.2 0}{3.5 0}
\put{$\bmatrix u\cr -h'\endbmatrix$} at 1.4 0.5
\put{$\bmatrix h& u'\endbmatrix$} at 2.6 0.3

\endpicture}
$$
This is a Riedtmann-Zwara sequence, thus $X'$ is a 
Riedtmann-Zwara degeneration of $X$. 
\end{proof} 

\noindent
{\bf Example 8.} 
Let $\Lambda$ be given by the quiver with one loop $\alpha$ at 
the vertex $a$ and an arrow $b \to a$, with the relation $\alpha^2 = 0.$
$$
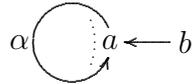
{\beginpicture
\setcoordinatesystem units <1cm,1cm>
\multiput{} at 0 0.5  0 -.5 /
\put{$a$} at 0 0 
\put{$b$} at 1 0 
\circulararc 310 degrees from -0.02 0.2 center at -.5 0
\arr{0.8 0}{0.2 0}
\arr{-0.04 -.25}{-0.02 -.2}
\setdots <1mm>
\setquadratic
\plot -0.25 .3  -.2 0  -.25 -.3 /
\put{$\alpha$} at -1.2 0
\endpicture}
$$

Let $K = P(a)$ and $Y = S(b)$, thus $\dim\Ext^1(Y,K) = 2$.
The module $C = \tau^{-}K$ is of length 4 with socle $a$ and top $b\oplus b$, thus
$\dim\Hom(C,Y) = 2.$ Note 
that $\Gamma(C) = \End(C)^{\text{op}} = k[t]/t^2$ and that $\Hom(C,Y)$ as a
$\Gamma(C)$-module is cyclic. 

The universal cover of the Auslander-Reiten quiver of $\Lambda$ looks as follows:
$$
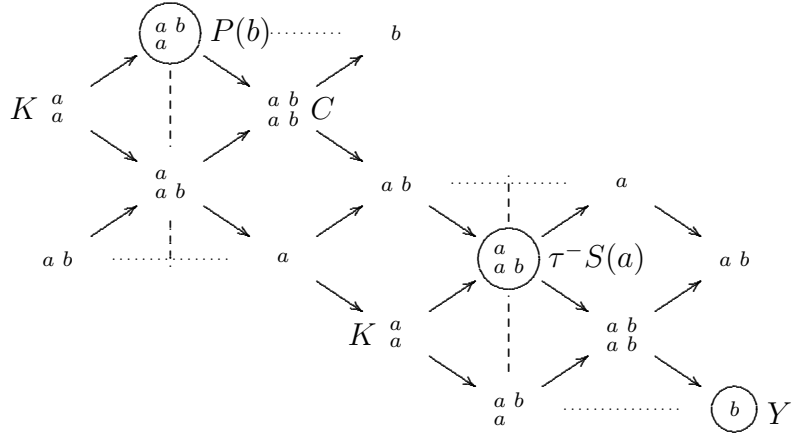
{\beginpicture
\setcoordinatesystem units <1.5cm,1cm>
\put{$\smallmatrix a \cr a \endsmallmatrix$} at -1 4
\put{$\smallmatrix a&b \endsmallmatrix$} at -1 2

\put{$\smallmatrix a&b\cr
                   a \endsmallmatrix$} at 0 5
\put{$\smallmatrix a\cr
                   a&b \endsmallmatrix$} at 0 3
\put{$\smallmatrix a & b\cr
                   a & b \endsmallmatrix$} at 1 4
\put{$\smallmatrix a \endsmallmatrix$} at 1 2
\put{$\smallmatrix b \endsmallmatrix$} at 2 5
\put{$\smallmatrix a&b \endsmallmatrix$} at 2 3
\put{$\smallmatrix a\cr
                   a \endsmallmatrix$} at 2 1
\put{$\smallmatrix a \cr
                   a&b \endsmallmatrix$} at 3 2
\put{$\smallmatrix a&b\cr
                   a \endsmallmatrix$} at 3 0

\put{$\smallmatrix a & b\cr
                   a & b \endsmallmatrix$} at 4 1
\put{$\smallmatrix a \endsmallmatrix$} at 4 3

\put{$\smallmatrix b \endsmallmatrix$} at 5 0
\put{$\smallmatrix a&b \endsmallmatrix$} at 5 2

\put{$Y$} at 5.4 -.05 

\arr{-.7 4.3}{-.3 4.7}
\arr{-.7 3.7}{-.3 3.3}
\arr{-.7 2.3}{-.3 2.7}

\arr{0.3 4.7}{0.7 4.3}
\arr{0.3 3.3}{0.7 3.7}
\arr{0.3 2.7}{0.7 2.3}

\arr{1.3 4.3}{1.7 4.7}
\arr{1.3 3.7}{1.7 3.3}
\arr{1.3 2.3}{1.7 2.7}
\arr{1.3 1.7}{1.7 1.3}

\arr{2.3 2.7}{2.7 2.3}
\arr{2.3 1.3}{2.7 1.7}
\arr{2.3 0.7}{2.7 0.3}

\arr{3.3 2.3}{3.7 2.7}
\arr{3.3 1.7}{3.7 1.3}
\arr{3.3 0.3}{3.7 0.7}

\arr{4.3 2.7}{4.7 2.3}
\arr{4.3 1.3}{4.7 1.7}
\arr{4.3 0.7}{4.7 0.3}

\circulararc 360 degrees from 0.27 5 center at 0 5 
\circulararc 360 degrees from 3.27 2 center at 3 2 
\circulararc 360 degrees from 5.2 0 center at 5 0 

\setdots <1mm>
\plot 0.9 5  1.5 5 /
\plot 2.5 3  3.5 3 /
\plot -.5 2  .5 2 /
\plot 3.5 0  4.5 0 /
\setdashes <1mm>
\plot 0 4.5  0 3.5 /
\plot 0 2.5  0 1.9 /
\put{} at 0 5.5 
\plot 3 0.5  3 1.5 /
\plot 3 2.5  3 3.1 /
\multiput{$K$} at -1.3 4  1.7 1  /
\put{$P(b)$} [l] at 0.35 5 
\put{$\tau^{-}S(a)$} [l] at 3.35 2 
\put{$C$} at 1.35 4
\endpicture}
$$
The region in-between the dashed lines is a fundamental domain
of the Auslander-Reiten quiver of $\Lambda$; the Auslander-Reiten quiver of $\Lambda$
is obtained by identifying these lines in order to form a Moebius strip.

The encircled vertices in the Auslander-Reiten quiver yield ${}^C[\to Y\rangle$. 
Here is ${}^C[\to Y\rangle$ as well as $\mathcal  S\Hom(C,Y)$:
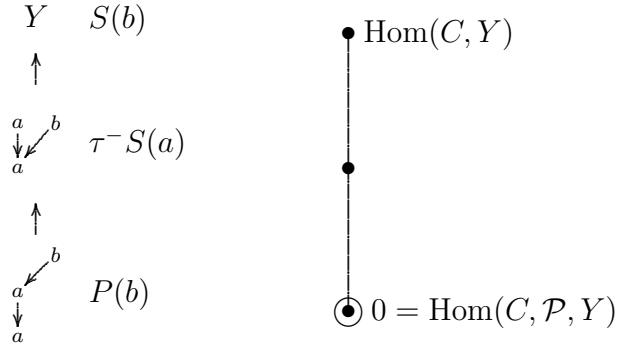
$$
{\beginpicture
\setcoordinatesystem units <1cm,2cm>
\put{\beginpicture

\put{$Y$} at 0 1.85
\put{$S(b)$} [l] at .7 1.83

\put{\beginpicture
\setcoordinatesystem units <.5cm,.5cm>
\multiput{$\ssize a$} at 0 -.2  0 1 /
\put{$\ssize b$} at 1 1
\arr{0 0.7}{0 0.1}
\arr{0.8 .8}{0.2 .1}
\endpicture} at 0 1
\put{$\tau^{-}S(a)$} [l] at .7 1

\put{\beginpicture
\setcoordinatesystem units <.5cm,.5cm>
\multiput{$\ssize a$} at 0 -.2  0 1 /
\put{$\ssize b$} at 1 2
\arr{0 0.7}{0 0.1}
\arr{0.8 1.8}{0.2 1.2}
\endpicture} at 0 0 
\put{$P(b)$} [l] at .7 0

\arr{0 0.4}{0 0.6}
\arr{0 1.4}{0 1.6}
\endpicture} at 0 0
\put{\beginpicture

\multiput{$\bullet$} at 0 0  0 0.95  0 1.85 /
\plot 0 0  0 1.85 /
\put{$\Hom(C,Y)$} [l] at 0.2 1.85
\put{$$} [l] at 0.3 0.95
\put{$0 = \Hom(C,\mathcal  P,Y)$} [l] at 0.3 0
\put{$\bigcirc$} at 0 0

\endpicture} at 5 0
\endpicture}
$$

\section{Modules $K$ with semisimple endomorphism ring.}
\label{sec:9}
 
We start with a well-known
characterization of such modules.

\begin{Lem}\label{semisimple-kernel}
Let $K$ be a module. The following conditions are equivalent:
\item{\rm (i)} The endomorphism ring of $K$ is semisimple.
\item{\rm (ii)} There are pairwise orthogonal bricks $K_1,\dots,K_n$ such that
 $\add K = \add\{K_1,\dots,K_n\}$.
\end{Lem}
	
\begin{Prop} Let $K$ be a module with semisimple endomorphism ring.
Let $f = f'h$, where $f, f'$ are right minimal epimorphisms with  
kernels in $\add K$, starting at the same module $X$. 
Then $h$ is surjective and its kernel belongs to $\add K$.
\end{Prop}

\begin{proof} 
By assumption, there is a commutative diagram with exact rows
$$
{\beginpicture
\setcoordinatesystem units <1.7cm,1.3cm>
\arr{-.75 1}{-.25 1}
\arr{0.25 1}{0.75 1}
\arr{1.25 1}{1.75 1}
\arr{2.25 1}{2.75 1}

\arr{-.75 0}{-.25 0}
\arr{0.25 0}{0.75 0}
\arr{1.25 0}{1.75 0}
\arr{2.25 0}{2.75 0}

\arr{0 .7}{0 .3}
\arr{1 .7}{1 .3}
\plot 2.03 0.7  2.03 0.3 /
\plot 1.97 0.7  1.97 0.3 /

\multiput{$0$} at -1 0  -1 1  3 0  3 1 /

\put{$K$} at 0 1
\put{$X$} at 1 1
\put{$Y$} at 2 1
\put{$K'$} at 0 0
\put{$X$} at 1 0
\put{$Y$} at 2 0

\put{$\ssize u$} at 0.5 1.2
\put{$\ssize u'$} at 0.5 .2
\put{$\ssize f$} at 1.5 1.2
\put{$\ssize f'$} at 1.5 .2
\put{$\ssize h'$} at .15 .5
\put{$\ssize h$} at 1.15 .5
\endpicture}
$$
with inclusion maps $u,u'$ such that both $K,K'$ 
belong to $\add K.$ Since $K,K'$ belong to $\add K$
and the endomorphism ring of $K$ is semisimple,
there is a submodule $K''$ of $K'$
such that $K' = \Img(h')\oplus K''$. Let us denote by $u'': K'' \to K'$ the
inclusion map and by $q': K' \to K'/\Img(h') = K''$ the canonical projection,
thus $q'u'' = 1_{K''}.$ Now $q': K' \to K''$ is the cokernel of $h'$,
and we can identify $K''$ and we can complete the diagram above by inserting
the cokernels of $h$ and $h'$ as follows:
$$
{\beginpicture
\setcoordinatesystem units <1.7cm,1.3cm>
\arr{-.75 1}{-.25 1}
\arr{0.25 1}{0.75 1}
\arr{1.25 1}{1.75 1}
\arr{2.25 1}{2.75 1}

\arr{-.75 0}{-.25 0}
\arr{0.25 0}{0.75 0}
\arr{1.25 0}{1.75 0}
\arr{2.25 0}{2.75 0}

\arr{0 .7}{0 .3}
\arr{1 .7}{1 .3}
\plot 2.03 0.7  2.03 0.3 /
\plot 1.97 0.7  1.97 0.3 /

\multiput{$0$} at -1 0  -1 1  3 0  3 1 /

\put{$K$} at 0 1
\put{$X$} at 1 1
\put{$Y$} at 2 1
\put{$K'$} at 0 0
\put{$X$} at 1 0
\put{$Y$} at 2 0

\put{$\ssize u$} at 0.5 1.2
\put{$\ssize u'$} at 0.5 .2
\put{$\ssize f$} at 1.5 1.2
\put{$\ssize f'$} at 1.5 .2
\put{$\ssize h'$} at .15 .5
\put{$\ssize h$} at 1.15 .5

\arr{0 -.3}{0 -.7}
\arr{1 -.3}{1 -.7}
\arr{0 -1.3}{0 -1.7}
\arr{1 -1.3}{1 -1.7}
\put{$\ssize q'$} at .15 -.5
\put{$\ssize q$} at 1.15 -.5
\put{$K''$} at 0 -1 
\put{$K''$} at 1 -1 
\put{$0$} at 0 -2 
\put{$0$} at 1 -2 
\plot 0.3 -1.03 0.7 -1.03 /
\plot 0.3 -.97 0.7 -.97 /
\endpicture}
$$
Since $qu' = q'$, we see that $qu'u'' = q'u'' = 1_{K''}$,
thus $K''$ is a direct summand of $X$ which lies inside the kernel of $f'$.
Since we assume that $f'$ is right minimal, it follows that $K'' = 0$, thus
$h'$ is surjective. Therefore also $h$ is surjective. 

On the other hand, the kernel of $h$ can be identified with the kernel of $h'$,
and again using that $K,K'$ belong to $\add K$ 
and that the endomorphism ring of $K$ is semisimple,
we see that the
kernel of $h$ belongs to $\add K.$
\end{proof}

\begin{Prop}\label{length-KRS}
Let $K$ be a module 
and assume that the endomorphism ring of $K$ is semisimple. Let $C = \tau^{-}K$
and assume that $\Hom(C,\mathcal  P,Y) = 0.$ 
If $f: X \to Y$ is right minimal and right $C$-determined, 
then $f$ is surjective and
\Rahmen{|f|_C = \mu(\Ker(f))}
where $\mu(M)$ is the Krull-Remak-Schmidt number (the 
number of direct summands when $M$ is written as a
direct sum of indecomposable modules.

Also, $\eta_{CY}^{-1}(0)$ is given by the
universal extension from below
$$
 0 \to K' \to X \to Y \to 0
$$
with $K' \in \add K$.
\end{Prop} 

\begin{proof} 
Since we assume that $\Hom(C,\mathcal P,Y) = 0,$ all
right minimal
right $C$-determined maps $f: X \to Y$ are surjective, and the kernel
of such a map is in $\add K.$ Let $|f|_C = n$ and consider a 
(maximal) chain
$$
 X_n 
 \overset {h_n}  \longrightarrow   X_{n-1} 
 \overset {h_{n-1}}  \longrightarrow   \cdots 
 \overset {h_2}  \longrightarrow   X_1 
 \overset {h_1}  \longrightarrow   X_0 = Y
$$
of non-invertible maps such that the compositions
$f_t = h_1\cdots h_t$ for $1\le t \le n$ are right minimal and right 
$C$-determined. Now all the maps $f_t$ are right minimal epimorphisms 
with kernels in $\add K$, thus, according to
proposition 7.2, also the maps $h_t$ are epimorphisms with kernels $K_t$ in $\add K$.
Since we assume that the endomorphism ring of $K$ is semisimple, 
we see that the kernel of $f_t$ is just $\bigoplus_{i=1}^t K_t$.
Now let $K_t = K'\oplus K''$
be a direct decomposition with $K'$ indecomposable,
and let $h: X_t\to X_t/K'$ be the canonical projection. Since $K'$ is 
contained in the kernel of $h_t$, we can factor $h_t$ through $h$
and obtain a map $h'_t: X_t/K' \to X_{t-1}$ with kernel $X''$ such that
$h_t = h'_th$.
Altogether we have obtained a refinement 
$$
 X_n 
 \overset {h_n} \longrightarrow  X_{n-1} 
 \overset {h_{n-1}} \longrightarrow  \cdots 
 \longrightarrow  X_t 
 \overset h \longrightarrow  X_t/K'  
 \overset {h_t'} \longrightarrow  X_{t-1} 
 \longrightarrow
 \cdots 
 \overset {h_2} \longrightarrow  X_1 
 \overset {h_1} \longrightarrow  X_0 = Y
$$
with $h_t = h'_th$. Let $f' = f_{t-1}h'_t = h_1\cdots h_{t-1} h'_t$. 
We apply proposition \ref{semisimple-kernel} to $f_t$ and $f'$. Note that 
$f_t= f_{t-1}h_t = f_{t-1}h'_th = f'h$ and the corresponding map
from the kernel of $f_t$ to the kernel of $f'$ is just the
split epimorphism $K_t \to K_t/K'$. Thus \ref{semisimple-kernel} asserts that
$f'$ is right minimal (and of course also right $C$-determined).
The maximality of the chain $(h_1,\dots,h_t)$ implies that $h'_t$
has to be invertible, thus $K'' = 0$. This shows that $K_t$ is
indecomposable. Altogether, we see that $\mu(\Ker(f_t)) = t,$
thus $\mu(\Ker(f)) = n = |f|_C.$

The last assertion is obvious: If 
$$
 0 \longrightarrow  K' \longrightarrow  X \overset f \longrightarrow  Y 
 \longrightarrow  0
$$
is the universal extension from below with $K' \in \add K$, then 
$[f\rangle$ belongs to ${}^C[\to Y\rangle$ and any other extension of $Y$
from below with kernel in $\add K$ is induced from it. Thus $[f\rangle$
has to be the zero element of the lattice ${}^C[\to Y\rangle$.
\end{proof} 

\begin{Cor}	 
Let $K$ be a module with injective dimension at most $1$
and assume that the endomorphism ring of $K$ is semisimple. Let $C = \tau^{-}K$.
If $f: X \to Y$ is right minimal and right $C$-determined, 
then $f$ is surjective and
$|f|_C = \mu(\Ker(f))$. Also, $\eta_{CY}^{-1}(0)$ is given by the
universal extension from below using modules in $\add K$.
\end{Cor}
	
\begin{proof}  Combine Proposition \ref{inj-dim-1} and Proposition \ref{length-KRS}.
\end{proof}

\noindent 
{\bf Example 9.} Consider the 3-subspace quiver
$$
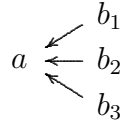
{\beginpicture
\setcoordinatesystem units <1.2cm,.6cm>
\put{$a$} at 0 0
\put{$b_1$} at 1 1
\put{$b_2$} at 1 0
\put{$b_3$} at 1 -1
\arr{0.7 0.8}{0.3 0.3}
\arr{0.7 -.8}{0.3 -.3}
\arr{0.7 0}{0.3 0}
\endpicture}
$$
We consider the indecomposable modules which are neither projective,
nor injective: these are the modules
$N(i) = \tau Q(b_i)$, for $1\le i \le 3$ and $M = \tau Q(a)$.

Let $Y = Q(a)$ and  
$C = N(1)\oplus N(2)\oplus N(3)$.
Then  ${}^C[\to Q(a)\rangle$  looks as follows:
$$
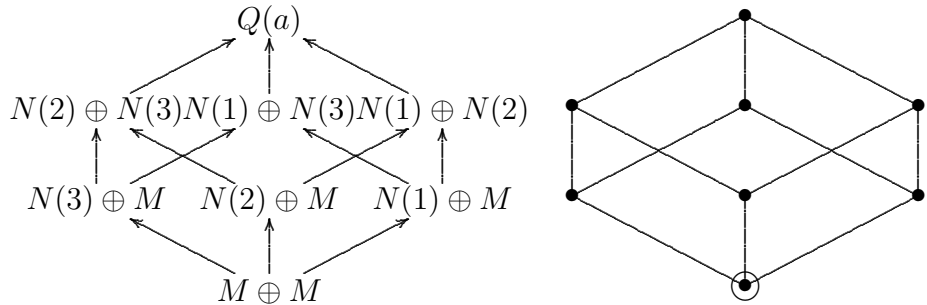
{\beginpicture
\setcoordinatesystem units <2.3cm,1.2cm>
\put{\beginpicture
\put{$Q(a)$} at 1 3
\put{$N(2)\oplus N(3)$} at 0 2
\put{$N(1)\oplus N(3)$} at 1 2 
\put{$N(1)\oplus N(2)$} at 2 2
\put{$N(3)\oplus M$} at 0 1
\put{$N(2)\oplus M$} at 1 1
\put{$N(1)\oplus M$} at 2 1
\put{$M\oplus M$} at 1 0
\arr{0.8 0.2}{0.2 0.8}
\arr{1 0.2}{1 0.8}
\arr{1.2 0.2}{1.8 0.8}

\arr{0.8 1.2}{0.2 1.8}
\arr{1.2 1.2}{1.8 1.8}

\arr{0.2 1.2}{0.8 1.8}
\arr{1.8 1.2}{1.2 1.8}

\arr{0 1.2}{0 1.8}
\arr{2 1.2}{2 1.8}

\arr{0.2 2.2}{0.8 2.8}
\arr{1 2.2}{1 2.8}
\arr{1.8 2.2}{1.2 2.8}
\endpicture} at 0 0 

\put{\beginpicture
\multiput{$\bullet$} at 1 3  0 2  1 2  2 2  0 1  1 1  2 1  1 0 /
\plot 1 0  2 1  2 2  1 3  0 2  0 1  1 0  1 1  2 2 /
\plot 1 1  0 2 /
\plot 0 1  1 2  2 1 /
\plot 1 2  1 3 /
\put{$\bigcirc$} at 1 0 
\endpicture} at 2.75 0

\endpicture}
$$
Note that here we have $K = \tau C = P(B_1)\oplus P(b_2)\oplus P(b_3).$

\section{Comparison with Auslander-Reiten theory.}
\label{sec:10}

Let $C = Y$ be indecomposable and consider the Auslander bijection
for $C = Y$:
$$
 {}^Y[\to Y\rangle \quad \longleftrightarrow \quad \mathcal  S\Hom(Y,Y).
$$
The subspace $\rad(Y,Y)$ of $\Hom(Y,Y)$ on the right corresponds to the 
right almost split map ending in $Y$. Two possible cases have to 
be distinguished: 

If $Y = P$ is projective, then we get a morphism which is
right determined by a projective module, thus we must get a monomorphism.
Of course, what we obtain is just the embedding of the radical $\rad P$ 
into $P$.

If $Y$ is not projective, then we get an epimorphism with kernel in
$\add \tau Y$. Actually, we get the epimorphism of the Auslander-Reiten
sequence ending in $Y$, thus the kernel is precisely $\tau Y.$
	\bigskip

What we see is that {\it the minimal right almost split map ending in $Y$ 
is a waist in ${}^Y[\to Y\rangle$} and it just corresponds to the waist
$\rad(Y,Y) \subset \End(Y).$ 
	\bigskip

More generally, we have:

\begin{Prop}  Let 
$Y$ be indecomposable and $Y$ a direct summand of $C$, and consider
$$
 {}^C[\to Y\rangle \quad \longleftrightarrow \quad \mathcal  S\Hom(C,Y).
$$
Then $\mathcal  S\Hom(C,Y)$ is a local module
(with maximal submodule $\Img\Hom(C,g)$, where
$g$ is minimal right almost split ending in $Y$).
\end{Prop}

\begin{proof} 
 Let $\tau Y \longrightarrow \mu Y \overset g \longrightarrow  Y$ 
be the Auslander-Reiten sequence
ending in $Y$. Then the right-equivalence class $[g\rangle$ of
$g$ belongs to $^C[\to Y\rangle$ and every map $X \to Y$
which is not a split epimorphism, factors through $g$. This means
that every element of $^C[\to Y\rangle$ different from the identity map $Y \to Y$
is less or equal to $[g\rangle$. This means that $^C[\to Y\rangle$ 
has a unique maximal
submodule, namely $\Hom(C,g).$
\end{proof}

{\bf The Auslander-Reiten formula $\Ext^1(Y,K) \simeq  D\underline{\Hom}(\tau^{-}K,Y).$}
The Auslander bijection provides a bijection between the submodules of 
$\underline{\Hom}(\tau^{-}K,Y)$ with the right equivalence classes of surjective maps 
$M \to Y$ with kernel in $\add(K)$ (namely, the submodules of $\underline{\Hom}(\tau^{-}K,Y)$
just correspond bijectively to the submodules of $\Hom(\tau^{-}K,Y)$ which
contain $\Hom(\tau^{-}K,\mathcal  P,Y)$). Consider the following triangle: 
$$
\hbox{\beginpicture
\setcoordinatesystem units <3cm,1.3cm>

\put{\{$K \overset{f}\to M \overset{g}\to Y\mid \text{short exact}\}$} at 0.3 2.05
\put{$\Ext^1(Y,K)$} at 0 1
\put{${}^{\tau^{-}K}[\to Y\rangle_{\text{epi}}$} at 0 0
\arr{0 1.7}{0 1.3}
\arr{0  .7}{0  .3}
\arr{0.6 1.5}{1.4 0.5}
\arr{0.6 0}{1.4 0}
\put{$\mathcal  S \underline{\Hom}(\tau^{-}K,Y)$} at 2 0
\put{$\ssize g\ \mapsto\  g\Hom(\tau^{-}K,M)$} at 1.4 1.3
\put{$(f,g)$} at -1 2
\put{$[f,g]$} at -1 1
\put{$[g\rangle$} at -1 0
\arr{-1 1.6}{-1 1.4}
\arr{-1  .6}{-1  .4}
\plot -1.01 1.6  -.99 1.6 /
\plot -1.01  .6  -.99  .6 /
\put{$\eta_{\tau^{-}K,Y}$} at 1 .15
\endpicture}
$$
Here, the two left hand columns concern the 
canonical way of attaching to a short exact sequence $\epsilon = (f,g)$ 
the corresponding element $[\epsilon] = [f,g]$ in $\Ext^1$; if $(f,g)$ and
$(f',g')$ are short exact sequences with $[f,g] = [f',g']$ in
$\Ext^1$, then the maps $f,f'$ are right equivalent, thus $[f,g] \mapsto 
[g\rangle$ is a well-defined map
$\Ext^1(Y,K) \to {}^{\tau^{-}K}[\to Y\rangle_{\text{epi}}$ and we obtain
a commutative triangle as shown. 

Having another look at the left columns of the triangle, the reader should obverse
that the split short exact sequence which yields the zero element of $\Ext^1$
gives the unit element of the lattice ${}^{\tau^{-}K}[\to Y\rangle_{\text{epi}}$,
namely the identity map $1: Y \to Y.$ 

\begin{Prop}
Let $g: M \to Y$ be surjective with kernel $K$ and $C =\tau^{-}K $.
Then an element $h\in \Hom(C,Y)$ belongs to the set
$\eta_{C,Y}(g)$ if and only if 
the induced sequence $h_*([f,g])$ splits. Also, the set 
$\eta_{C,Y}(g)$ is a $\Gamma(C)$-submodule of $\Hom(C,Y)$, where
$\Gamma(C) = \End(C)^{\text{op}}$.
\end{Prop}

\begin{proof} Given an element $h\in \Hom(C,Y)$, the induced sequence $h_*([f,g])$
splits if and only if $h$ factors through $g$. 
But we have already noted that $\eta_{C,Y}(g)$ 
is the set of all elements $h\in \Hom(C,Y)$ which factor through $g$.
And we know that $\eta_{C,Y}(g)$ is a $\Gamma(C)$-submodule of $\Hom(C,Y)$.
\end{proof} 
	
Now let us invoke the
Auslander-Reiten formula: $\Ext^1(Y,K)  \simeq D\underline{\Hom}(\tau^{-}K,Y).$
Here we use that we deal with an artin algebra $\Lambda$, thus the center $k$ of $\Lambda$
is a (commutative) artinian ring and $D$ is the duality functor $\mo k \to \mo k$
given by a minimal cogenerator of $\mo k$. 

There are the following two horizontal bijections as well
as the vertical map on the left:
$$
\hbox{\beginpicture
\setcoordinatesystem units <2.5cm,1.5cm>

\put{$\Ext^1(Y,K)$} at 0 1
\put{${}^{\tau^{-}K}[\to Y\rangle_{\text{epi}}$} at 0 0
\arr{0  .7}{0  .3}
\arr{0.6 0}{1.4 0}
\put{$\mathcal  S \underline{\Hom}(\tau^{-}K,Y)$} at 2 0
\put{$[f,g]$} at -1 1
\put{$[g\rangle$} at -1 0
\arr{-1  .6}{-1  .4}
\plot -1.01  .6  -.99  .6 /
\put{$\eta_{\tau^{-}K,Y}$} at 1 .15

\put{$\simeq$} at 1 1
\put{$D\underline{\Hom}(\tau^{-}K,Y)$} at 2 1
\setdashes <1.2mm>
\arr{2 0.7}{2 0.3}
\endpicture}
$$
We have inserted 
a dashed arrow on the right which corresponds to the composition of the three given 
maps
$$
 D\underline{\Hom}(\tau^{-}K,Y) \longrightarrow 
  \Ext^1(Y,K) \longrightarrow 
{}^{\tau^{-}K}[\to Y\rangle_{\text{epi}} \longrightarrow
  \mathcal  S \underline{\Hom}(\tau^{-}K,Y)
$$
It seems to be of interest to describe in detail this composition! One may 
conjecture that
here one attaches to a linear map $\alpha\in D\underline{\Hom}(\tau^{-}K,Y)$
the largest $\Gamma$-submodule of $\underline{\Hom}(\tau^{-}K,Y)$ lying in the
kernel of $\alpha$.
	\bigskip 

Let us now assume that $K$ is indecomposable so that $\add K$ consists of direct
sums of copies of $K$. Given a short exact sequence
$K \overset{f}\to M \overset g \to Y$, the map $g$ is right minimal if and only if the sequence
does not split. On the other hand, given 
$[g\rangle \in {}^{\tau^{-}K}[\to Y\rangle_{\text{epi}}$ with
$g$ right minimal, the kernel of $g$ 
is the direct sum of say $t$ copies of $K$ and $t\le 1$ just means that
$[g\rangle$ is obtained from a short exact sequence $K \overset f \to  M \overset g \to  Y$.
It follows that under the assumption that $K$ is indecomposable, it is easy to
identify the elements of $\mathcal  S\Hom(\tau^{-}K,Y)$ which are 
images of the elements of $\Ext^1(Y,K)$ under $\eta_{\tau^{-}K,Y}$.
	\bigskip 

\noindent 
{\bf 
Comparison.} In which way are the Auslander bijections
better than the Auslander-Reiten formula?
What is the advantage of the Auslander bijection compared to the Auslander-Reiten
formula?
	\medskip

1) We do not only deal with the set $\underline{\Hom}(C,Y)$ but with all of
$\Hom(C,Y)$. To extend such a bijection as given by the Auslander-Reiten formula 
to a larger setting should always be of
interest. But also note that the set $\underline{\Hom}(C,Y)$ depends on the
module category which we consider, not just on the modules themselves.
	\medskip

2) The duality is replaced by a covariant bijection. 
	\medskip

3) The usual Auslander-Reiten picture concerns indecomposable modules, and almost split sequences, thus indecomposable modules and irreducible maps.
In the language of the Auslander bijection, we only deal with $C = Y$ indecomposable and only with the submodule $\rad(C,C) \subset \Hom(C,C)$, whereas
\begin{itemize} 
\item we should not restrict to indecomposable modules,
\item and not to the condition $C = Y$,
\item and we want to deal with all submodules of $\Hom(C,Y)$, not just the radical subspace. 
\end{itemize} 
	\bigskip

Concerning the 
Auslander-Reiten-theory, there is an essential difference whether $C$ is
projective or not. If $C$ is projective, 
we obtain an inclusion map, whereas if $C$ is
not projective, then we obtain an extension. --- This feature dominates also 
the Auslander bijections: First of all, there is
the extreme case of  $C$ being projective, then we consider submodules
(and we consider arbitrarily ones, not just the radicals of the indecomposable
projective modules). In general 
we deal with extensions ending in a submodule $Y'$ of $Y$; if $C$ is a generator, then we deal with all possible extensions of submodules of $Y$
from below using modules in $\add\tau C$. 
	\bigskip

4) The Auslander-Reiten theory only deals with the 
factor category of $\mo \Lambda$ modulo
the infinite radical. The Auslander bijection takes care of {\bf any} morphism.
	\bigskip

5) Families of modules do not play any role in the Auslander-Reiten theory. 
As we will see soon, families of modules are an essential features in the
frame of the Auslander bijections.
	\bigskip 
	\bigskip 
	\bigskip 

\centerline{\huge \bf II. Families of modules.}
	\bigskip  

\section{The modules present in ${}^C[\to Y\rangle$ are of bounded length.}
\label{sec:11}

We say that a module $M$ is {\it present in}
${}^C[\to Y\rangle$ provided there exists a right minimal
map $f: M \to Y$ which is right $C$-determined; similarly, we say
that $M$ is {\it present in} ${}^C[\to Y\rangle^t$ provided there exists a right minimal
map $f: M \to Y$ which is right $C$-determined and $|f|_C = t.$

\begin{Prop}\label{bound}
There is a constant $\lambda = \lambda(\Lambda)$ such that for any
pair of modules $C, Y$ with $C\neq 0$ and 
any right minimal right $C$-determined map $X \to Y,$ we have 
\Rahmen{  |X|\ \le \ \lambda\,|C|^2\,|Y|.}
\end{Prop} 

\begin{proof}  Write $C = \bigoplus C_i$ with indecomposable 
direct summands $C_i$ and let $K_i = \tau C_i$. Note that 
$|K_i| \le d^2|C_i|$, 
where $d = |{}_\Lambda\Lambda|$ (see for example \cite{[R1]}). 
Of course,  there is also the weaker bound $|K_i| \le d^2|C|.$ 

Let $\Lambda$ be an artin $k$-algebra, where $k$ is a commutative
artinian ring. Since there are only finitely many simple modules and since $\Ext^1(Y,X)$
is a $k$-module of finite length for all $\Lambda$-modules $X,Y$ of finite length,
the length of the $k$-modules $\Ext^1(S,S')$, 
where  $S,S'$ are simple $\Lambda$-modules, is bounded, thus let $e$ be the maximum. 

The long exact $\Hom$-sequences imply that
the length of the $k$-module $\Ext^1(Y,X)$ is bounded by $e|X||Y|,$ 
for any $\Lambda$-modules $X,Y$ of finite length.

Now, let $f: X\to Y$ be right minimal and right $C$-determined. There is
a short exact sequence $K' \to X \to Y'$ with $Y'$ a submodule of $Y$, 
such that the module $K'$ belongs to $\add\tau C$ and such that $f$ is the
composition of $X\to Y'$ and the inclusion map $Y' \to Y.$
With $f$ also the map $X \to Y'$ is right minimal. 

Write $K' = \bigoplus_{i=1}^s K_i^{t_i}$.
The $\Ext$-Lemma
of \cite{[R6]} asserts that $t_i$ is bounded by the $k$-length of $\Ext^1(Y',K_i)$,
since the map $X \to Y'$ is right minimal. Thus 
$$
  t_i \le e|K_i||Y'| \le d^2e|C_i||Y'| \le d^2e|C_i||Y|,
$$ 
and therefore $\sum t_i \le d^2e|C||Y|.$
Thus
$$
 |K'| = \sum t_i|K_i| \le \sum t_i d^2|C| = d^2|C|\sum t_i \le d^4e |C|^2|Y|,
$$
and therefore 
$$
 |X| = |K'|+|Y'| \le d^4e |C|^2|Y| + |Y| \le d^4e |C|^2|Y| + |C|^2|Y| =
    (d^4e+1)|C|^2|Y|
$$
(here we use that $C\neq 0$). Thus, let $\lambda = d^4e+1.$
\end{proof} 

There is the following converse:

\begin{Prop}\label{all-present}
Let $Q = (D\Lambda)^b$. Then any module of length at most $b$ is present in 
${}^\Lambda[\to Q\rangle.$
\end{Prop}

\begin{proof}  Let $M$ be a module of length at most $b$, thus the socle $\soc M$ of $M$
is a semisimple module of length at most $b$ and therefore a submodule of $Q = (D\Lambda)^b$.
It follows that $M$ itself can be embedded into $Q$. 
Such an embedding $f: M \to Q$ is right minimal and 
right $\Lambda$-determined, thus $M$ is present in $^\Lambda[\to Q\rangle.$
\end{proof}

\begin{Cor} Let $\mathcal  M$ be a set of modules. 
The modules in $\mathcal  M$ are of bounded length if and only if 
there exist
modules $C, Y$ such that any module $M$ in $\mathcal  M$ is present in ${}^C[\to Y\rangle$.
\end{Cor} 
	
\begin{proof} If all the modules in $\mathcal  M$ are present in ${}^C[\to Y\rangle$, then they have
to be of length at most $\lambda(\Lambda)|C|^2|Y|$, thus of bounded length, see Proposition 
\ref{bound}.
Conversely, if the modules in $\mathcal  M$ are of length at most $b$, then they are
present in ${}^\Lambda[\to Q\rangle$ where $Q = (D\Lambda)^b$, according to Proposition 
\ref{all-present}.
\end{proof}
	
\noindent 
{\bf Remark: Modules versus Morphisms.} 
As we have seen, given any infinite set $\mathcal  M$ of {\bf modules} of bounded
length, there are modules $C,Y$ such that all the modules in $\mathcal  M$ are present
in ${}^C[\to Y\rangle$. On the other hand, given modules $X,Y$,
one cannot expect that the right equivalence classes $[f\rangle$ of 
all (or at least infinitely many)  non-zero {\bf morphisms}
$f: X \to Y$ belong to some ${}^C[\to Y\rangle$, since the kernels of these maps $f$
may belong to infinitely many isomorphism classes.

\section{Minimal infinite families.}
\label{sec:12}
 
Recall that a Krull-Remak-Schmidt category $\mathcal  C$ is said 
to be {\it finite} provided the number of isomorphism classes of indecomposable objects in $\mathcal  C$ is
finite, otherwise $\mathcal  C$ is said to be {\it infinite.}

Let $\mathcal  M$ be a family of modules. We say that $\mathcal  M$ is {\it minimal
infinite} provided $\add\mathcal  M$ is infinite whereas $\add\mathcal  M'$ is finite,
where $\mathcal  M'$ is the set of modules $M'$ which are proper submodules or proper
factor modules of modules in $\mathcal  M.$ 
	
\begin{Lem}\label{minimal-infinite}
If $\mathcal  M$ is a minimal infinite family of modules (not necessarily
of the same length), then
there is an infinite  subset $\mathcal  N \subseteq \mathcal  M$ which 
consists of pairwise non-isomorphic indecomposable modules of fixed length.
\end{Lem} 

Of course, $\mathcal  N$ is again minimal infinite.

\begin{proof}
Since $\add\mathcal  M$ is infinite, there is a sequence of modules $M_i\in 
\mathcal  M$ with $i\in \mathbb N$ such that $M_{i}$ does not belong to
$\add\{M_1,\dots,M_{i-1}\}$ for all $i$. Write $M_{i} = N_i \oplus N'_i$
such that $N_{i}$ is indecomposable and does not belong to 
$\add\{M_1,\dots,M_{i-1}\}$. It follows that the modules $N_1,N_2,\dots$
are indecomposable and pairwise non-isomorphic. 
Let $I$ be the set of natural numbers such that
$N'_i\neq 0.$ If $i\in I$, then $N_i$ is a proper submodule of $M_i$, thus
belongs to $\add\mathcal  M'$. Since $\add\mathcal  M'$ is finite and the modules
$N_i$ with $i\in I$ are indecomposable and pairwise non-isomorphic, it
follows that $I$ is finite. Let $\mathcal  N$ be the set of modules $N_i$
with $i\notin I$, then $\mathcal  N$ is an infinite subset of $\mathcal  M$ and consists
of pairwise non-isomorphic indecomposable modules. 

It is easy to see that {\it the modules in a
minimal infinite family of indecomposable modules 
are of bounded length.} Namely, according to \cite{[R7]}
(as well as \cite{[R8]}), any indecomposable module $M$ of length at least 2 has
an indecomposable proper submodule of length at least $\frac 1{pq}|M|$, where
$p$ is the maximal length of an indecomposable projective module, $q$ the maximal
length of an indecomposable injective module. Thus, if the modules in $\mathcal  N$
are indecomposable, but not of bounded length, then we find indecomposable
submodules of modules in $\mathcal  N$ which are of arbitrarily large length. 

Thus, assume that the modules in $\mathcal  N$ are of length at most $b$. Then there
are infinitely many modules in $\mathcal  N$ of fixed length $b'$, for some $1\le b' \le b.$ 
\end{proof} 

\begin{Prop} Let $\mathcal  M$ be a minimal infinite family. Then there are indecomposable non-projective 
modules $C,Y$ and infinitely many indecomposable modules $M_i\in \mathcal  M$
with short exact sequences
$$
 0 \to \tau C \longrightarrow  M_i \overset{f_i}\longrightarrow Y \to 0
$$
such that $f_i$ is a co-Gabriel-Roiter projection. All these equivalence 
classes $[f_i\rangle$ belong to ${}^C[\to Y\rangle^1$.
\end{Prop}
	
\begin{proof} According to Lemma \ref{minimal-infinite}, 
we can assume that all the modules in $\mathcal  M$
are indecomposable and of fixed length $b$.
Let $\mathcal  M'$ be the set of modules $M'$ 
which are proper submodules or proper factor modules of modules in $\mathcal  M.$
By assumption, there are only finitely many
isomorphism classes of modules in $\mathcal  M'$.
Since there are only finitely many simple modules,
we must have $b\ge 2$, thus any $M\in \mathcal  M$ has a 
co-Gabriel-Roiter factor-module
$Q_M = M/K_M$; here, $K_M$ is a submodule of $M$.
Note that $K_M$ is a proper non-zero submodule, that $Y_M$ is a proper factor 
module of $M$, and that both modules $K_M$ and $Y_M = M/K_M$ are indecomposable. 
Of course, both $K_M$ and $Y_M$ 
belong to $\mathcal  M'.$
Since there are (up to isomorphism) only finitely many pairs $(K,Y)$ in $(\mathcal  M')^2$, 
it follows that there is a pair $(K,M)$
such that there are 
infinitely many pairwise non-isomorphic indecomposable modules $M$
with $K = K_M$ and $Y = Y_M.$ The exact sequences $0 \to K \to M \to Y \to 0$ with
$M$ indecomposable show that $K$ cannot be injective, 
thus $C = \tau^{-}K$ is again indecomposable. This completes the proof of the
proposition.

It remains to observe that $[f_i\rangle$
belongs to ${}^C[\to Y\rangle^1$, but this has been shown in \ref{co-GR}.
\end{proof} 
		\medskip 

The lattices ${}^C[\to Y\rangle$ which arise in this way can have arbitrarily
large height, as the following examples show.
	\medskip

\noindent 
{\bf Example 10.} Let $\Lambda$ be the $n$-Kronecker algebra,
this is the path
algebra of the quiver
$$
{\beginpicture
\setcoordinatesystem units <1.5cm,1cm>
\put{$a$} at 0 0
\put{$b$} at 1 0
\arr{0.8 0.2}{0.2 0.2}
\arr{0.8 -.2}{0.2 -.2}
\setdots <0.5 mm>
\plot 0.55 -.1 0.55 .1 /
\endpicture}
$$
with $n$ arrows. Let $Y = S(b)$ and $C = \tau^{-}S(a)$. Then $C$ has
dimension vector $(n^2-1,n)$, thus $\dim\Hom(C,Y) = n$.
Since $\End(C) = k$, we see that $\mathcal  S\Hom(C,Y)$
is of the form $\mathbb G(n)$ as exhibited in section \ref{sec:19}. 
The family of modules which we are interested in are the indecomposable
modules $M$ of length $2$. Given such a module $M$, there is an exact sequence 
$$
 0 \to  S(a) \longrightarrow  M \overset{f_M}\longrightarrow S(b) \to  0,
$$
with $f_M$ a co-Gabriel-Roiter projection. As we will see in the next
section, this family of maps $f_M$ is a cofork.

\section{Forks and coforks.}
\label{sec:13}
We call a family of maps $(g_i: X \to M_i)_{i\in I}$ a {\it fork} provided
for any finite subset $J\subseteq I$, the map $g_J = (g_i)_{i\in J}: X \to \bigoplus_{i\in J} M_i$
is left minimal. The dual notion will be that of a {\it cofork,} this is a set
of maps $(f_i:  M_i\to Y)_{i\in I}$ such that 
the direct sum map $f^J=(f_i)_{i\in J}:  \bigoplus_{i\in J} M_i \to Y$
with $J$ any finite subset of $I$ is right minimal.
	
\begin{Lem}\label{fork}
Let $g_i: X \to M_i$ with $i\in I$ be a family of non-zero
maps with indecomposable modules $M_i$. Then $(g_i: X \to M_i)_i$ is a fork
if and only if $g_i \notin \sum_{j\neq i} \Hom(M_j,M_i)g_j$ for all $i\in I$.
\end{Lem}
	
\begin{proof} 
First, assume that there is some $i$ with
$g_i \in \sum_{j\neq i} \Hom(M_j,M_i)g_j,$ say there is the subset
$\{1,2,\dots,t\} \subseteq I$ such that 
$g_1 \in \sum_{j=2}^t\Hom(M_j,M_1)g_j.$ Thus, for $2\le j \le t$ 
there are maps $p_j: M_j\to M_1$ such that $g_1 = \sum_{j=2}^t p_jg_j.$
We want to show that the map 
$g_J: X \to M = \bigoplus_{i=1}^t M_i$
is not left minimal. Let $N$ be the kernel of the map 
$$
 (-1,p_2,\dots,p_t):  \bigoplus_{i=1}^t M_i \longrightarrow M_1.
$$
It is easy to check that $M$ is the direct sum of $M_1$ and $N$.
But the equality $-g_1+ \sum_{j=2}^t p_jg_j$ shows that the image
of $g_J$ is contained in $N$. Since the image of $g_J$ is contained in a
proper direct summand of $M$, we see that $g_J$ is not left minimal.

Conversely, assume that $g_i \notin \sum_{j\neq i} \Hom(M_j,M_i)g_j$ for all $i\in I$. We want to show that for any
finite subset $J\subseteq I$, the map $g_J: X \to \bigoplus_{i\in J} M_i$ 
is left minimal. If $J$ is empty, then we deal with
the zero map $X \to 0$ which of course is left minimal. If $J$
consists of the single element $i$, then we deal with $g_i: X \to M_i$.
Since $g_i\neq 0$ and $M_i$ is indecomposable, the map $g_i$ is
left minimal. Thus we can assume that $J$ contains $t\ge 2$ elements,
say $J = \{1,2,\dots,t\}.$ Assume that the map
$g_J : X \to M$ with $M = \bigoplus_{i=1}^t M_i$ is not
left minimal.  
Then there is a proper direct decomposition $M = N\oplus N'$ 
such that the image of $g_J$ is contained in $N$. We can assume that $N'$ is indecomposable,
thus isomorphic to some $M_i$, say to $M_1$. 
Let $p: M \to N' = M_1$ be the
projection with kernel $N$, write $p = (p_i)_i$ with $p_i: M_i \to M_1.$ Note that 
$p_1$ has to be an isomorphism. Replacing any $p_i$ by $(p_1)^{-1}p_i$, we can assume that
$p_1 = 1.$ Since the image of $g_J$ is contained in the kernel $N$ of $p$, 
we have
$\sum_{i=1}^t p_ig_i = 0$, thus 
$
 g_1 = -\sum_{i=2}^t p_ig_i. 
$ 
This completes the proof.
\end{proof}

We may use forks and coforks in order to construct inductively
families of modules:

\begin{Prop}\label{inductive}
 Let $(f_i: M_i\to Y)_{i\in I}$ be an
infinite cofork,
with indecomposable modules $M_i$ which have the same Gabriel-Roiter
measure $\gamma_0$. Let 
$$
 \mathcal  K = \add\{\Ker(f^J)\mid J \subseteq I,\
  |J| < \infty \}.
$$ 
Then $\mathcal  K$ is infinite, all indecomposable modules $K$
in $\mathcal  K$ are
cogenerated by $\{M_i\mid i\in I\}$ and have Gabriel-Roiter measure $\gamma(K) < \gamma_0.$
\end{Prop} 
	
Before we give the proof, let us note the following consequence: either
$\mathcal  K$ will contain indecomposable modules of arbitrarily large
length, or else the indecomposable modules in $\mathcal  K$ are of bounded length.
In the latter case, there is an infinite subset of 
indecomposable modules $K_i$ in $\mathcal  K$ all of which have the same 
Gabriel-Roiter measure $\gamma_1 < \gamma_0.$
	
\begin{proof}  
First, assume that $\mathcal  K = \add K$ for some module $K$, let $C = \tau^{-}K.$
For any natural number $t$, we may choose a subset $J$ of $I$ of cardinality $t$. 
Then the module $M_J = \bigoplus_{i\in J}M_i$
is present in ${}^C[\to Y\rangle$. But the modules which are present in 
${}^C[\to Y\rangle$ are of bounded length, whereas $|M_J| \ge t.$ This contradiction
shows that $\mathcal  K$ is infinite. 

Let $K$ be an indecomposable module in $\mathcal  K$, say a direct summand of 
the kernel $K_J$ of $f^J: M_J \to Y$.
Assume that $\gamma(K) \ge \gamma_0$. This implies that the inclusion map
$K \subseteq K_J \subset M_J$ splits. But this contradicts the fact that the map 
$f^J: M_J \to  Y$ is right minimal.
\end{proof} 
	
Let us present two different ways for obtaining forks:

\begin{Prop}[Gabriel-Roiter forks]\label{GR-fork} 
Let $M_i$ with $i\in I$ be pairwise non-isomorphic
indecomposable modules of fixed length with isomorphic Gabriel-Roiter submodule
$U$, say with embeddings $u_i: U \to M_i$. Then the family $(u_i: U \to M_i)_{i\in I}$
is a fork.
\end{Prop} 

\begin{proof} 
The maps $u_i: U \to M_i$ are non-zero maps and the modules $M_i$
are indecomposable. Thus, if the family $(u_i: U \to M_i)_i$ is not a fork,
then Lemma \ref{fork} asserts that there is some $i$ with
$u_i \in \sum_{j\neq 1} \Hom(M_j,M_1)f_j$, thus we can assume that there
is the subset $\{1,2,\dots,t\} \subseteq I$ and maps $p_j: M_j \to M_1$
for $2\le j \le t$ such that $u_1 = \sum_{j=2}^t p_ju_j.$

Let $M'' = \bigoplus_{j=2}^tM_j$.
For $2\le j \le t$, consider the submodule $p_ju_j(U)$ of $M_1$. It is a proper submodule of $M_1$,
thus $\gamma(p_ju_j(U)) \le \gamma(U)$, since $U$ is a Gabriel-Roiter submodule of $M_1$. 
On the other hand, the image of $(p_ju_j)_{j=2}^t: U \to M''$ is contained in 
$\bigoplus_{j=2}^t p_ju_j(U).$ Since $u_1 = \sum_{j=2}^t p_ju_j,$
this map 
$(p_ju_j)_{j=2}^t$ is injective, thus $U$ is a submodule of 
$\bigoplus_{j=2}^t p_ju_j(U)$.
It follows that $\gamma(U) \le \max\gamma(p_ju_j(U))$. 
Thus there is some $s$ with $2\le s \le t$ such that
$\gamma(U) = \gamma(p_su_s(U)).$ Since $U$ is indecomposable, $U$ has to be
a direct summand of $p_su_s(U)$. But $p_su_s(U)$ is a factor module of $U$, thus
$p_su_s: U \to p_su_s(U)$ is an isomorphism. As a consequence, $u_s$ is a split monomorphism.
But this is impossible, since $u_s$ is the inclusion of a Gabriel-Roiter submodule. 
\end{proof} 

Let us add the dual assertion.

\begin{Prop}[co-Gabriel-Roiter coforks]\label{GR-cofork}
 Let $M_i$ with $i\in I$ be pairwise non-isomorphic
indecomposable modules of fixed length with isomorphic 
co-Gabriel-Roiter factor modules
$Y$, say with projections  $v_i: M_i \to Y$. Then the family
$(v_i: M_i \to Y)_{i\in I}$ is a cofork.
\end{Prop}
	
\noindent 
{\bf Remark.} Let us stress the
following:
Let $M_i$ with $i\in I$ be pairwise non-isomorphic indecomposable modules
with Gabriel-Roiter submodules $U_i \subset M_i$ and assume that the
modules $M_i/U_i$ are isomorphic, say to $Y$, with projection maps
$f_i: M_i \to Y$. Then the family $(f_i: M_i \to Y)_{i\in I}$ is not necessarily a cofork.
Namely, consider again example 8, and look at the projections $f: P(b) \to S(b)$
and $f': \tau^{-}S(a) \to S(b)$. As we mentioned already, the kernels of both
maps are Gabriel-Roiter submodules. Since there is a factorization
$f = f'h$, the map $[f,f']: P(b)\oplus \tau^{-}S(a) \to S(b)$ is not
right minimal. 
	\medskip 

Here is a consequence of \ref{GR-cofork} and \ref{co-GR}.
	
\begin{Prop} Let $M_i$ with $i\in I$ be pairwise non-isomorphic
indecomposable modules with isomorphic 
co-Gabriel-Roiter factor modules
$Y$, say with projections  $v_i: M_i \to Y$ and assume that also the kernels
are isomorphic, say isomorphic to $K$. Let $C = \tau^{-}K.$ Then
$(v_i: M_i \to Y)_{i\in I}$ is a cofork which belongs to ${}^C[\to Y\rangle^1.$
\end{Prop}

A proof similar to \ref{GR-fork}
shows that starting with any infinite family $\mathcal  M$ of indecomposable modules
with fixed length, there are infinite forks consisting 
of maps $S \to M$, where $S$ is a simple module and $M\in \mathcal  M$.
A fork $(g_i: S \to M_i)_i$ with $S$ simple will be called a 
{\it simple fork.} Similarly, a cofork $(f_i: M_i \to S)_i$ with $S$
simple is called a {\it simple cofork}.

\begin{Prop}\label{simple-module}
Let $M_i$ with $i\in I$ be pairwise non-isomorphic
indecomposable modules with fixed Gabriel-Roiter measure.
Then there exists for every index $i\in I$ a simple submodule $S_i$ of $M_i$,
say with inclusion map $u_i: S_i \to M_i,$ such that for every simple
module $S$ and $I(S) = \{i\in I \mid S_i = S\}$ 
the family $(u_i: S\to M_i)_{i\in I(S)}$ is a fork.
\end{Prop} 
	
\begin{proof} Let us assume that the modules in $\mathcal  M$ have length $b$. Of course, $b > 1.$
Consider a module $M_i$ with $i\in I$.
It is not cogenerated by the remaining modules, thus the intersection of the
kernels of all maps $\phi: M_i \to M_j$ with $j\neq i$ is non-zero.
Let $S_i$ be a simple submodule of $M_i$ which is contained in this intersection (thus
$\phi(S_i) = 0$ for all maps $\phi: M_i \to M_j$ with $j\neq i$. 
Denote by $u_i: S_i \to M$ the inclusion map, thus $\Hom(M_i,M_j)u_i) = 0$ for all
$j\neq i.$ For any simple module $S$, let $I(S)
 = \{i\in I \mid S_i = S\}$. 

In order to see that  $(u_i: S\to M_i)_{i\in I(S)}$ is a fork,
we have to show the following: For any finite subset $J$ of $I$,
say $J = \{1,2,\dots, t\}$, the map
$u^J: (u_i)_i: S \to \bigoplus_{i=1}^t M_i$ is left minimal. 
Again, this is clear for $t=1$, since the maps $u_i$ are non-zero and the
modules $M_i$ are indecomposable. 
Thus we can assume that $t\ge 2$. If  
the map $u^J$ is not left minimal, then, up to permutation of the
indices, there are maps $p_i: M_i \to M_1$ with $2\le i \le t$ such that 
$u_1 = \sum_{i=2}^t p_iu_i. $ 
However, by construction, $\Hom(M_i,M_1)u_i = 0$ for $i\neq 1$, thus all the summands $p_iu_i$
with $2\le i \le t$ are zero. Since $u_1 \neq 0$, we obtain a contradiction.
\end{proof} 

\begin{Cor}[Simple forks] Let $\mathcal  M$ be an infinite set of pairwise non-isomorphic 
indecomposable modules of fixed length. Then there exists an infinite 
simple fork $(u_i: S\to M_i)_{i}$ such that all $M_i$ belong to $\mathcal  M$.
\end{Cor}
	
\begin{proof} Since the modules in $\mathcal  M$ are of bounded length, only finitely many
Gabriel-Roiter measures occur, thus there is an infinite subset $\mathcal  M' \subseteq \mathcal  M$
which consists of indecomposable modules with fixed Gabriel-Roiter measure. 
Now we apply \ref{simple-module}. Since there are only finitely many simple modules, one of the forks
$(u_i: S\to M_i)_{i\in I(S)}$ has to be infinite. 
\end{proof} 
	
Again we add the dual assertion.
	
\begin{Cor}[Simple coforks]\label{simple-cofork}
Let $\mathcal  M$ be an infinite set of pairwise non-isomorphic 
indecomposable modules of fixed length. Then there exists an infinite 
simple cofork $(v_i: M_i\to S)_{i}$ such that all $M_i$ belong to $\mathcal  M$.
\end{Cor}

The setting developed here allows to provide a proof of the following result which first was
established in \cite{[R6]}. Note that this result strengthens the assertion of the first
Brauer-Thrall conjecture \cite{[Ro]}.
	
\begin{Cor}[First Brauer-Thrall conjecture]
 Let $\mathcal  M$ be an infinite set of indecomposable modules of a fixed
length. Then there are indecomposable modules of arbitrarily large length which are
cogenerated by modules in $\mathcal  M$.
\end{Cor}

\begin{proof} 
Let $\mathcal  M_0 = \mathcal  M$ and apply \ref{simple-cofork}. 
Thus, there is a simple cofork
$(v_i: M_i\to S)_{i}$ such that all $M_i$ belong to $\mathcal  M_0$.
Let $\mathcal  K = \add\{\Ker(v^J)\mid J \subseteq I,\
|J| < \infty \}$. According to \ref{inductive} we know that $\mathcal  K$ is infinite, that 
all indecomposable modules $K$ in $\mathcal  K$ are
cogenerated by $\{M_i\mid i\in I\}$ and that they 
have Gabriel-Roiter measure $\gamma(K) < \gamma_0.$
Now either the indecomposable modules in $\mathcal  K$ are of unbounded length, then we are done.
Or else they are of bounded length: then we find in $\mathcal  K$ an infinite set 
$\mathcal  M_1$ of indecomposable modules having the same Gabriel-Roiter measure, 
say $\gamma_1$ and $\gamma_1 < \gamma_0.$
Inductively, we construct a sequence of sets of indecomposable modules
$$
 \mathcal  M_0, \mathcal  M_1, \dots, \mathcal  M_i
$$
such that the modules in $\mathcal  M_i$ are cogenerated by $\mathcal  M_{i-1}$ and
have fixed Gabriel-Roiter measure $\gamma_i < \gamma_{i-1}$,
for $i\ge 1$. The procedure stops in case there are indecomposable modules of 
unbounded length which are cogenerated by $\mathcal  M_i$, and then these modules are
cogenerated by $\mathcal  M$. Otherwise the procedure can be continued indefinitely.
But then we have constructed infinitely many sets $\mathcal  M_i$ of indecomposable
modules. Since the modules in $\mathcal  M_i$ 
have Gabriel-Roiter measure $\gamma_i$ and
the measures $\gamma_i$ are pairwise different,
the modules in $\bigcup_i\mathcal  M_i$ cannot be of bounded length. Also, the modules
in any $\mathcal  M_i$ are cogenerated by $\mathcal  M.$ This completes the proof.
\end{proof}

\section{The Kronecker algebra.}
\label{sec:14}

Throughout this section, $\Lambda$ will be the Kronecker algebra as introduced
already in Example 2. 
It is a very important artin algebra and a clear understanding of its
module category $\mo \Lambda$ seems to be of interest.

For all pairs $C,Y$ of indecomposable $\Lambda$-modules, we
are going to describe the lattice ${}^C[\to Y\rangle$
as well as all the modules present in ${}^C[\to Y\rangle^1$, this
is the subset of
${}^C[\to Y\rangle$ of elements of right $C$-length $1$.
	\medskip

Let us recall the structure of the category $\mo \Lambda$ (see for example \cite{[R4]},
or \cite{[ARS]}, section VIII.7). 
There are the preprojective and the preinjective modules, 
modules without an indecomposable
direct summand which is preprojective or preinjective are said to be regular.
For any $\Lambda$-module $M$, its defect is defined by $\delta(M) = 
\dim\Hom(M,Q_0) - \dim\Hom(P_0,M)$. Any indecomposable preprojective $\Lambda$-module has
defect -1, the indecomposable preinjective modules have defect 1, all the regular
modules have defect 0. There are countably many indecomposable preprojective modules,
they are labeled $P_i$, and also countably many indecomposable preinjective modules,
they are labeled $Q_i$; both $P_i$ and $Q_i$ have length $2i+1$. 
The indecomposable regular modules are those modules which belong to stable
Auslander-Reiten components, and all these components are stable tubes
of rank $1$. The full subcategory $\mathcal  R$ of all regular modules is abelian. 
By definition,
the simple regular modules are the regular modules which are simple objects in this
subcategory. Given any indecomposable regular module $R$, its 
endomorphism ring $\End(R)$
is a commutative ring (namely a ring of the form $k[T]/\langle f\rangle$, where
$k[T]$ is the polynomial ring in one variable $T$ with coefficients in $k$ and
$f$ is a power of an irreducible polynomial)
and $\dim R = 2\dim \End(R)$. 

As we have mentioned, we are interested in pairs $C,Y$ of indecomposable 
$\Lambda$-modules such that a family of modules is present in ${}^C[\to Y\rangle$.
It turns out that only the case of $C$ being preprojective, $Y$ being preinjective
is relevant, as the following proposition shows.

\begin{Prop}\label{Kronecker}
Let $\Lambda$ be the Kronecker algebra and $C,Y$
indecomposable $\Lambda$-modules. If $C$ is preprojective or preinjective, then
${}^C[\to Y\rangle$ is a projective geometry. If $C$ is regular, then
${}^C[\to Y\rangle$ is a chain.
\end{Prop} 
	
By definition, a projective geometry $\mathbb G(d)$ over the field $k$
is the lattice of subspaces of the $k$-space of dimension $d.$ The chain 
$\mathbb I(d)$
is the set of integers $i$ with $0\le i \le d$ with the usual ordering. 
The labels have been chosen in such a way that 
the height of $\mathbb G(d)$ as well as of $\mathbb I(d)$ is just $d$. 
	\medskip 

The following table provides the precise data: Here, an indecomposable regular module
of regular length $t$ and with regular socle $R$ is denoted by $R[t]$.
$$
{\beginpicture
\setcoordinatesystem units <1cm,.6cm>
\put{row} at -1 0.2
\put{1)} at -1 -1
\put{2)} at -1 -2
\put{3)} at -1 -3
\put{4)} at -1 -4
\put{5)} at -1 -5
\put{6)} at -1 -6
\put{$C$} at 1 .2
\put{$Y$} at 3 .2
\put{${}^C[\to Y\rangle$} at 6.5 .2
\plot -2 -.4  7.7 -.4 /
\put{$P_i$} at 1 -1
\put{$P_j$} at 3 -1
\put{$\mathbb G(j-i+1)$} at 6.5 -1
\put{$P_i$} at 1 -2
\put{$R[t]$} at 3 -2
\put{$\mathbb G(\tfrac12 \dim R[t])$} at 6.5 -2
\put{$P_i$} at 1 -3
\put{$Q_j$} at 3 -3
\put{$\mathbb G(i+j)$} at 6.5 -3
\put{$R[s]$} at 1 -4
\put{$R[t]$} at 3 -4 
\put{$\mathbb I(\min(s,t))$} at 6.5 -4
\put{$R[s]$} at 1 -5
\put{$Q_j$} at 3 -5
\put{$\mathbb I(\tfrac12 \dim R[s])$} at 6.5 -5
\put{$Q_i$} at 1 -6
\put{$Q_j$} at 3 -6 
\put{$\mathbb G(i-j+1)$} at 6.5 -6
\put{$(i\le j)$} at 4.3 -1   
\put{$(i\ge j)$} at    4.3 -6 
\endpicture}
$$
Let us stress that in row 4), the regular modules $C,Y$ are supposed to belong 
to the same tube, namely to the tube containing a fixed simple regular module $R$.
For all pairs $C,Y$ of indecomposable Kronecker modules which are not contained in the
table, one has $\Hom(C,Y) = 0$, thus ${}^C[\to Y\rangle$ consists of a single element. 
	
\begin{proof}[Proof of proposition \ref{Kronecker}] 
First, let us calculate $\dim\Hom(C,Y)$. The reflection functors of \cite{[BGP]} yield
$$
{\beginpicture
\setcoordinatesystem units <1.35cm,.6cm>
\put{$\Hom(P_i,P_j)$} [l] at 0 4 
\put{$\Hom(P_i,R[t])$} [l] at 0 3
\put{$\Hom(P_i,Q_j)$} [l] at 0 2 
\put{$\Hom(R[s],Q_j)$} [l] at 0 1 
\put{$\Hom(Q_i,Q_j)$} [l] at 0 0 

\put{$= \dim \Hom(P_0,P_{j-i}) = j-i+1$} [l] at 2 4 
\put{$= \dim \Hom(P_0,R[t]) = \tfrac12 \dim R[t]$} [l] at 2 3 
\put{$= \dim \Hom(P_0,Q_{j+i}) = i+j$} [l] at 2 2 
\put{$= \dim \Hom(R[s],Q_0) = \tfrac12 \dim R[s]$} [l] at 2 1 
\put{$= \dim \Hom(Q_{i-j},Q_0) = i-j+1$} [l] at 2 0 
\endpicture}
$$
As we have mentioned, we have
$\dim R = 2\dim\End(R)$. It follows that
$$
 \dim \Hom(R[s],Q_j) = \tfrac12 \dim R[s] =  s\dim\End(R).
$$
Finally,  
$$
 \dim\Hom(R[s],R[t]) = \min(s,t)\dim\End(R).
$$

In case $C = P_i$ or $Q_j$ one has $\End(C) = k$, thus 
Auslander's Second Theorem asserts that 
${}^C[\to Y\rangle$ is of the form $\mathbb G(\dim\Hom(C,Y)).$ This yields the rows 1), 2), 3)
and 6) of the table. 

It remains to look at the rows 4) and 5), thus we assume now that $C = R[s]$
and $Y = R[t]$ or $Y = Q_j$.
We show that $\Hom(C,Y)$ is a cyclic $\Gamma(C)$-module, thus we have to find
an element $g\in\Hom(C,Y)$ such that $g\End(C) = \Hom(C,Y)$, or,
equivalently, such that
$g(\rad\End(C))^{r-1} \neq 0$, where $r$ is the length of $\Hom(C,Y)$ (as
a $\Gamma(C)$-module). Note that $\Gamma(C)$ is a local ring, thus there is a unique
simple $\Gamma(C)$-module $S$. Since $S$ has $k$-dimension $\dim \End(R)$, the 
calculations above show that $r = s$ in case $Y = Q_j$ and $r = \min(s,t)$ in
case $Y = R[t].$ On the other hand, $(\rad\End(C))^{r-1}$ is generated by
any endomorphism of $R[s]$ with image $R[s-r+1]$. Thus let $p: R[s] \to R[s-r+1]$
be the canonical projection, $u: R[s-r+1] \to R[s]$ the canonical inclusion,
then $up$ generates $(\rad\End(C))^{r-1}$. 
Our aim is to exhibit $g: R[s] \to Y$ such that $gup \neq 0.$ 

First, let $Y = R[t]$ and $s\le t.$ Then $r = \min(s,t) = s$ and $R[s-r+1] = R[1] = R.$
Let $g: R[s] \to R[t]$ be the canonical inclusion, thus $gu: R \to R[t]$ is
an inclusion, in particular non-zero, and therefore also $gup \neq 0.$ 

Second, let $Y = R[t]$ and $s > t$. Then $r = \min(s,t) = t$ and $R[s-r+1] = R[s-t+1]$.
Let $g: R[s] \to R[t]$ be the canonical projection. Then $gu: R[s-t+1] \to R[t]$
has image $R$ (since the kernel of $g$ is $R[s-t]$). This shows that $gu$ is non-zero,
and therefore also $gup\neq 0.$ 

Finally, we have to deal with the case $Y = Q_j$. Take a non-zero map
$g': R \to Q_j$. Since $\Ext^1(R[s]/R,Q_j) = 0$, there exists $g: R[s] \to Q_j$
such that $gu = g'$. Since $gu \neq 0,$ it follows that also $gup \neq 0$. 

Thus, always we have found $g: R[s] \to Y$ such that $gup \neq 0.$ As a consequence,
$\Hom(C,Y)$ is a cyclic $\Gamma(C)$-module. Since $\Gamma(C)$ is a local uniserial
ring, it follows that $\mathcal  S\Hom(C,Y)$ is of the form $\mathbb I(r)$, where $r$
is the length of $\Hom(C,Y)$. According to our calculations, $r = \min(s,t)$
in case $Y = R[t]$ and $r = \frac12\dim R[s]$ in case $Y = Q_j$. 

Thus we have verified the assertions presented in the table. If the pair
$C,Y$ does not occur in the table, then it is well-known that $\Hom(C,Y) = 0$,
thus $\mathcal  S\Hom(C,Y)$ consists of a single element and therefore is of the
form $\mathbb I(0) = \mathbb G(0).$ This completes the proof. 
\end{proof} 
	
The following assertion which has been shown in the proof will be of further interest:
	
\begin{Lem}\label{extension}
Any non-zero map $g': R \to Q_j$ can be extended to a map
$g: R[t]\to Q_j$, and any such $g: R[t] \to Q_j$ 
generates the $\Gamma(C)$-module $\Hom(R[t],Q_j).$
\end{Lem} 
	
\noindent 
{\bf Remark.} We have mentioned in section \ref{sec:4}
 that both sides of the Auslander bijection concern maps with target $Y$, 
but that they invoke these maps in quite different
ways. A nice illustration seems to be Proposition \ref{Kronecker}. 
The map $g: R[t] \to Q_j$ constructed there is 
a generator of the maximal submodule of $\Hom(R[t],Q_j)$ 
and is used in the proof of Lemma \ref{extension} in order to show that
$\mathcal  S\Hom(R[t],Q_j)$ is of the form $\mathbb I(d)$ for some $d$.
On the other hand, in proposition \ref{maximal}, we will consider the
right equivalence class $[g\rangle$ as an element of
$[\to Q_j\rangle.$
	
\begin{Prop}\label{tube}
Let $\Lambda$ be the Kronecker algebra, let $C, Y$
be indecomposable $\Lambda$-modules and $M$ an indecomposable direct summand
of a module present in ${}^C[\to Y\rangle$. 

If $C,Y$ both are preprojective, also $M$ is preprojective. If $C,Y$ both
are preinjective, also $M$ is preinjective. 
If both $C,Y$ belong to the tube $\mathcal  T$, also $M$ belongs to $\mathcal  T$.

If $C$ is preprojective and $Y$ belongs to the tube $\mathcal  T$, then $M$ is
preprojective or belongs to $\mathcal  T$. If $C$ belongs to the tube $\mathcal  T$ and
$Y$ is preinjective, then $M$ is preinjective or belongs to the tube $\mathcal  T$. 
\end{Prop} 
	
\begin{proof} First, assume that $Y$ is preprojective. If $f: X \to Y$ is right minimal,
then $X$ has to be preprojecive. Thus, for any module $C$, all the modules
present in ${}^C[\to Y\rangle$ are preprojective.

Second, assume that $Y$ is regular, say belonging to the tube $\mathcal  T$.
Again, assume that $f: X \to Y$ is right minimal. Then $X$ is the direct sum
of a preprojective module and a module in $\mathcal  T.$ Thus, 
for any module $C$, all the modules
present in ${}^C[\to Y\rangle$ are direct sums of preprojective modules and modules in
$\mathcal  T.$

Finally, assume that $Y$ is preinjective and $C$ is regular or preinjective.
Let $f: X \to Y$ be right minimal and right $C$-determined.
Since $C$ has no indecomposable projective direct summand and $\Lambda$ is
hereditary, we see that $f$ is surjective and its kernel is in $\add K$ with $K 
= \tau C.$ Now, if $C$ is preinjective, then also $K$ is preinjective and $X$
as an extension of a preinjective module by a preinjective module is preinjective
again. On the other hand, if $C$ is regular, say belonging to the tube $\mathcal  T$,
then also $K$ belongs to $\mathcal  T$. Since $X$ is an extension of a module in $\mathcal  T$
by a preinjective module, it is the direct sum of a module in $\mathcal  T$ and
a preinjective module. This completes the proof.
\end{proof} 
	
\noindent 
{\bf Remark.} Let us stress that the cases $C = P_i, Y = R[t]$ and
$C = R[s], Y = Q_j$ are not at all dual (as the consideration of
$\Hom(P_i,R[t])$ and $\Hom(R[t],Q_j)$ could suggest), but 
are of completely different nature. The reason is that we always consider
$\Hom(C,Y)$ as a $\Gamma(C)$-module! 

Of course, we have already seen that
${}^C[\to Y\rangle$ is in the first case a projective geometry, in the
second case a chain. But also if we look at the different layers, we
encounter clear differences. In the chain case, all the elements of
${}^C[\to Y\rangle$ are given by short exact sequences
of the form $(R[s] \to M \to Q_j)$, thus by elements of $\Ext^1(Q_j,R[s])$.
In this case, we may interpret ${}^C[\to Y\rangle$ as a display of the
various orbits in $\Ext^1(Q_j,R[s])$ with respect to the action of
the automorphism group of $R[s]$. In contrast, in the projective geometry case,
only the elements of ${}^C[\to Y\rangle$ of right $C$-length at most $1$ are
given by short exact sequences of the form $(P_i\to M \to R[t])$, whereas for
the elements of right $C$-length at least $2$, we need short exact sequences
of the form $(P_i{}^a\to M \to R[t])$ with $a \ge 2.$ 

In the proof of proposition \ref{tube}, we have seen that for $C = P_i, Y = R[t]$,
the modules $M$ present in ${}^C[\to Y\rangle$ are direct sums of preprojective
modules and modules in the tube which contains $\mathcal  T$, and that for
$C = R[s], Y = Q_j$ the modules present in ${}^C[\to Y\rangle$ are direct sums of preinjective
modules and modules in $\mathcal  T$. However, in the first case the number of
indecomposable preprojective direct summands of $M$ may be large, whereas in the
second case there is just one direct summand of $M$ which is indecomposable 
preinjective. 
		\bigskip 

\noindent 
{\bf Preprojective $C$, preinjective $Y$.} 
Let us focus now the attention on ${}^C[\to Y\rangle$,
where $C$ is indecomposable preprojective and $Y$ is
indecomposable preinjective.  Here is the general behavior as seen in Proposition 
\ref{Kronecker}:
$$
{\beginpicture
\setcoordinatesystem units <1cm,0.8cm>
\setcoordinatesystem units <1cm,0.7cm>
\multiput{} at 0 0  12  12 /
\plot 0 10  12 10 /
\plot 2 0  2 12 /
\plot 0 12  2 10 /
\setdots <1mm>
\plot 0 2  12 2  /
\plot 0 4  12 4  /
\plot 0 6  12 6  /
\plot 0 8  12 8  /

\plot 4 0  4 12 /
\plot 6 0  6 12 /
\plot 8 0  8 12 /
\plot 10 0  10 12 /
\put{$K$} at 0.2 10.3
\put{$C$} at 1 10.3 
\put{$C$} at 1 10.3 
\put{$P_0$} at 1 9
\put{$P_1$} at 1 7 
\put{$P_2$} at 1 5 
\put{$P_3$} at 1 3
\put{$P_4$} at 1 1
\put{$P_0$} at 0.2 5 
\put{$P_1$} at 0.2 3
\put{$P_2$} at 0.2 1
 
\put{$Y$} at 1.5 11.5
\put{$Q_0$} at 3 11.5
\put{$Q_1$} at 5 11.5
\put{$Q_2$} at 7 11.5
\put{$Q_3$} at 9 11.5
\put{$Q_4$} at 11 11.5
\multiput{$\ssize 0$} at 2.2 9.8 /
\multiput{$\ssize 1$} at 2.2 7.8  4.2 9.8 /
\multiput{$\ssize 2$} at 2.2 5.8  4.2 7.8  6.2 9.8  /
\multiput{$\ssize 3$} at 2.2 3.8  4.2 5.8  6.2 7.8  8.2 9.8 /
\multiput{$\ssize 4$} at 2.2 1.8  4.2 3.8  6.2 5.8  8.2 7.8  10.2 9.8 /

\put{$\bullet$} at 3 9 

\multiput{\beginpicture
\setcoordinatesystem units <.5cm,.4cm>
\multiput{$\bullet$} at  0 0  0 1 /
\setsolid
\plot 0 0  0 1  /
\endpicture
} at 3 7  5 9   /

\multiput{\beginpicture
\setcoordinatesystem units <.5cm,.39cm>
\multiput{$\bullet$} at  1 1  2 0  3.5 1  2 2   1.5 1   0.5 1 /
\put{$\cdots$} at 2.5 1 
\setsolid
\plot 1 1  2 0  3.5 1  2 2  /
\plot 1 1  2 2 /
\plot 2 0  1.5 1  2 2 /
\plot 2 0  0.5 1  2 2 /
\plot 2 0  2.25 .5   /
\plot 2.25 1.5  2 2 /
\endpicture
} at 3 5  5 7  7 9   /

\multiput
{\beginpicture
\setcoordinatesystem units <.37cm,.33cm>
\multiput{$\ssize \bullet$} at  1 1   1 2  1.5 1   1.5 2  3 0  2 1  2 2  3 3 
    5 1  5 2    /
\setsolid
\plot 3 0  1 1  1 2  3 3   5 2  5 1  3 0 /
\plot 2 1  3 0  1.5 1 /
\plot 2 2  3 3  1.5 2 /
\plot 1 1  1.5 2  2 1  2 2  1.5 1  1 2 /
\multiput{$\cdots$} at 3.5 1  3.5 2   /
\endpicture} at 3 3  5 5  7 7  9 9  /

\multiput
{\beginpicture
\setcoordinatesystem units <.3cm,.27cm>
\multiput{$\ssize \bullet$} at 0.5 2  1 1  1 2  1 3  1.5 1  1.5 2  1.5 3  3 0  2 1  2 3  3 4     5 1  5 3  5.5 2  /
\setsolid
\plot 3 0  1 1  0.5 2  1 3  3 4   5 3  5.5 2  5 1  3 0 /
\plot 1 1  1 3 /
\plot 3 0  1.5 1  0.5 2  /
\plot 1.5 3   3 4 /
\plot 1.5 3  0.5 2 /
\plot  3 4  2 3  1 2  2 1  3 0 /
\plot  1 3  1.5 2  1.5 1  /
\plot 2 1  1.5 2 /
\multiput{$\cdots$} at 3.5 1  3.5 3  3 2  4.5 2 /
\endpicture} at 3 1  5 3  7 5  9 7  11 9 /

\multiput
{\beginpicture
\setcoordinatesystem units <.3cm,.3cm>
\multiput{$\ssize \bullet$} at 0.5 2  1 1  1 2  1 3  1.5 1  1.5 2  1.5 3  3 0  2 1  2 3  3 4     5 1  5 3  5.5 2  /
\setsolid
\plot 3 0  1 1  0.5 2  1 3  3 4   5 3  5.5 2  5 1  3 0 /
\plot 1 1  1 3 /
\plot 3 0  1.5 1  0.5 2  /
\plot 1.5 3   3 4 /
\plot 1.5 3  0.5 2 /
\plot  3 4  2 3  1 2  2 1  3 0 /
\plot  1 3  1.5 2  1.5 1  /
\plot 2 1  1.5 2 /
\multiput{$\cdots$} at 3.5 1  3.5 3  3 2  4.5 2 /
\endpicture} at 3 1  5 3  7 5  9 7  11 9 /

\endpicture}
$$
We want to know which modules are present in ${}^C[\to Y\rangle^1$.
As we will show, these are certain regular modules.
	
\begin{Lem}\label{strongly-regular}
Let $\Lambda$ be the Kronecker algebra and $M$ a regular module.
The following conditions are equivalent.
\item{\rm (i)} If $M = M'\oplus M''$, then $\Ext^1(M',M'') = 0.$
\item{\rm (ii)} The regular socle of $M$ is multiplicity-free.
\item{\rm (ii$^*$)} The regular top of $M$ is multiplicity-free.
\item{\rm (iii)} $M = \bigoplus_{i=1}^n M_i$, with indecomposable modules $M_i$ belonging
    to pairwise different tubes. 
\item{\rm (iv)} $\End(M)$ is commutative.
\end{Lem}

\begin{proof} 
The equivalence of (ii) and (iii), and dually of (ii$^*$) and (iii) is 
straight forward. The implication (iii) $\implies$ (i) follows from the fact that
$\Ext^1(M_i,M_j) = 0$ in case $M_i, M_j$ are regular and belong to different 
Auslander-Reiten components. The converse implication 
(i) $\implies$ (iii) follows from the fact that
$\Ext^1(M_i,M_j) \neq 0$ in case $M_i, M_j$ are nonzero regular and belong to 
the same Auslander-Reiten component.
In order to see the implication (iii) $\implies$ (iv), one should be aware that
for $M = \bigoplus_{i=1}^n M_i$, with indecomposable modules $M_i$ belonging
 to pairwise different tubes, one has $\End(M) = \prod_{i} \End(M_i)$
and that $\End(R)$ is commutative for any indecomposable regular module $R$.
Conversely, in order to show the implication (iv) $\implies$ (iii),
assume that $M$ is regular and assume that $M = M'\oplus M'' \oplus M''$
with $M', M''$ non-zero modules belonging to some tube. Then $\Hom(M',M'') \neq 0$,
thus there is a non-zero homomorphism $\phi: M' \to M''$  and we may consider this
as an endomorphism of $M$ by setting $\phi$ to be zero on $M''\oplus M'''$.
Let $e': M \to M$ be the projection of $M$ to $M'$ with kernel $M''\oplus M'''$.
Then $\phi e' = \phi \neq 0$, whereas $e'\phi = 0.$ This shows that $\End(M)$
is not commutative. 
\end{proof}

A regular Kronecker-module $M$ will be said to be 
{\it strongly regular} provided the equivalent conditions
of the Lemma are satisfied. 
Let $\mathcal  R(i)$ be the set of isomorphism classes of 
strongly regular modules of length $i$.
Note that $\mathcal  R(i)$ is empty in case $i$ is odd or negative, and $\mathcal  R(0)$
has just one element, namely the isomorphism class of the zero module.
 	\bigskip

As we have mentioned, we want to 
see in which way families of modules
may be present in ${}^C[\to Y\rangle^1,$ with $C$ indecomposable preprojective, and $Y$ indecomposable
preinjective. Here is the description of these sets.
	
\begin{Prop}\label{identification} 
Let $\Lambda$ be the Kronecker algebra, $C$ indecomposable preprojective, $Y$ indecomposable preinjective. 

If $C = P(S)$ for some simple module
$S$, then ${}^C[\to Y\rangle^1$ may be identified with the set of inclusion 
maps $X \to Y$ such that the socle of $Y/X$ is equal to $S$.

If $C = \tau^{-}K$ for some indecomposable module $K$, then 
${}^C[\to Y\rangle^1$ consists of the right equivalence classes of surjective
maps $X \to Y$ with kernel $K$.
\end{Prop} 

\begin{proof} See \ref{projective} and \ref{indecomposable-kernel}.
\end{proof} 
	
Given a morphism $f: X \to Y$, let $\sigma(f) = [X]$, the isomorphism class of the
source $X$ of $f$. 
We study the function $\sigma$ defined on ${}^C[\to Y\rangle^1$ with values in
the set of isomorphism classes of modules. The main result of this section 
is the following description of ${}^C[\to Y\rangle^1$:

\begin{Prop}\label{main}
 Let $\Lambda$ be the Kronecker algebra, $C$ indecomposable preprojective,
$Y$ indecomposable preinjective. Then $\sigma$ is a bijection
\Rahmen{\sigma: {}^C[\to Y\rangle^1 \longrightarrow\mathcal  R(i)
 \quad\text{\it with}\quad i=|C|+|Y|-4.}
\end{Prop} 
	
For the proof, we need some preliminary considerations.
	
\begin{Lem}  Let $\Lambda$ be the Kronecker algebra, let
$0 \to U \to X \to Y \to 0$ be a non-split exact sequence 
with $Y$ indecomposable preinjective,
$U$ preprojective. Then no indecomposable direct summand
of $X$ is preinjective.
\end{Lem} 
	
\begin{proof} 
 Let $X = X'\oplus X''$ with $X'$ indecomposable. Assume that $X'$ is 
preinjective. Denote the map $X \to Y$ by $f$ and let $f'$ be its restriction to
$X'.$ Then $f' \neq 0$, since otherwise $X'$ is a direct summand of the kernel
of $f$, thus equal to $K$, so that the sequence splits. Non-zero maps between
indecomposable preinjective modules are surjective, thus $f'$ is surjective. Of course,
$f'$ is not an isomorphism, since otherwise the sequence would split.
Let $U'$ be the kernel of $f'$. As we see, $U'\neq 0.$
We have the following commutative diagram with exact rows:
$$
{\beginpicture
\setcoordinatesystem units <1.7cm,1.3cm>
\arr{-.75 1}{-.25 1}
\arr{0.25 1}{0.75 1}
\arr{1.25 1}{1.75 1}
\arr{2.25 1}{2.75 1}

\arr{-.75 0}{-.25 0}
\arr{0.25 0}{0.75 0}
\arr{1.25 0}{1.75 0}
\arr{2.25 0}{2.75 0}

\arr{2 .7}{2 .3}
\arr{1 .7}{1 .3}
\plot 0.03 0.7  0.03 0.3 /
\plot -.03 0.7  -.03 0.3 /

\multiput{$0$} at -1 0  -1 1  3 0  3 1 /

\put{$U'$} at 0 1
\put{$U$} at 1 1
\put{$X''$} at 2 1
\put{$U'$} at 0 0
\put{$X'$} at 1 0
\put{$Y$} at 2 0

\put{$f'$} at 1.5 .2
\endpicture}
$$
Now $U'$ is a submodule of $U$, thus it is preprojective.
Since $U' \neq 0$, we must have $\delta(U') \le -1$,
where $\delta$ denotes the defect.
It follows that $\delta(X') = \delta(U') + \delta(Y) \le 0,$ since $\delta(Y) = 1.$ This contradicts our assumption that $X'$ is preinjective.
\end{proof} 
	
\begin{Lem}\label{characterization}
Let $X$ be regular, $Y$ indecomposable preinjective. Then 
the following conditions are equivalent:
\item{\rm (i)} $X$ is strongly regular, 
\item{\rm (ii)} There exists $f: X \to Y$ such that the kernel of $f$ does not contain
a simple regular submodule.

Any map $f: X\to Y$ with no simple regular submodule in its kernel 
is a monomorphism or an epimorphism. 
If $f: X \to Y$ and $f': X' \to Y$ are maps with no simple regular submodule
in the kernel, and $X, X'$ are isomorphic, then $f,f'$ are right equivalent.
\end{Lem} 
	
\begin{proof} 
First, we show that (ii)implies (i). Assume that there is $f: X \to Y$ such 
that the kernel of $f$ does not contain a simple regular submodule of $X$.
In order to show that $X$ is strongly regular, we show that its regular socle
is multiplicity-free. Assume, for the contrary, that $X$ has a submodule $U$
such that $U = R\oplus R$ with $R$ simple regular. Let $f_1,f_2$ be the restrictions
of $f$ to $R\oplus 0$ or $0\oplus R$, respectively. Since no simple regular submodule of 
$X$ is contained in the kernel of $f$, we see that both $f_1,f_2$ are non-zero maps.
According to Lemma \ref{extension}, there is a map $h: R \to R$ such that $f_1 = f_2h.$
But then $R' = \{(-h(x),x)\mid x\in R\}$ belongs to the kernel of $f$. Of course,
$R'$ is isomorphic to $R$, thus $R'$ is a simple regular submodule of $X$ which
belongs to the kernel of $f$, a contradiction.

Conversely, assume that $X = \bigoplus_i R_i[t_i]$ is a strongly regular
module, with pairwise different simple regular modules $R_i$.
According to Lemma \ref{extension}, there is a map $f_i: R_[t_i] \to Y$
such that the restriction of $f_i$ to $R_i$ is non-zero. Since the simple
regular submodules $R_i$ are pairwise non-isomorphic, these are the only
simple regular submodules of $X$, thus no simple regular submodule of $X$ lies
in the kernel of $f = (f_i)_i: X \to Y$, 

Now assume that $f: X \to Y$ is a map such that 
the kernel of $f$ does not contain a simple regular submodule.
Assume that $f$ is not an epimorphism. 
The image of $f$ is a factor module of $X$, thus the direct sum of a regular and a preinjective module. But $Y$ has no non-zero proper submodule 
which is preinjective. Thus, the image of $f$ is regular. 
The kernel of a map between regular modules is regular, thus
either a monomorphism, or it contains a simple regular submodule. Since the latter
is not possible, we see that $f$ has to be a monomorphism.

Finally, assume that 
$f: X \to Y$ and $f': X' \to Y$ are maps with no simple regular submodule
in the kernel. Let $g: X \to X'$ be an isomorphism. 
Write $X = \bigoplus R_i[t_i]$ with pairwise non-isomorphic
simple regular modules $R_i$. Let $f_i$ and $f'_i$ 
be the restriction of $f$ and $f'g$, respectively, to $R_i[t_i]$. Since $R_i$
is not in the kernel of $f$, the restriction of $f_i$ to $R_i$ is non-zero.
According to
Lemma \ref{extension}, 
there is a map $h_i: R_i\to R_i$ such that $f_i = f'_ih_i.$ 
But also the restriction of $f'_i$ to  $R_i$ is non-zero, thus $R_i$ is not in the kernel
of $h_i$ and therefore $h_i$ is an automorphism. Let $h = (h_i)_i: X \to X$. This
is an automorphism of $X$ with $f = f'gh$.
Since $f = f'gh$ and $gh$ is an isomorphism, we see that $f, f'$ are right equivalent.
\end{proof} 

\begin{Prop}\label{maximal}
Let $Y$ be indecomposable preinjective. 
The maximal submodules
of $Y$ are pairwise non-isomorphic and these are, up to isomorphism, 
all the strongly regular modules of length $|Y|-1$.

The kernels of the non-zero maps $Y \to Q_1$ 
are pairwise non-isomorphic and these are, up to isomorphism, 
all the strongly regular modules
of length $|Y|-3$.
\end{Prop} 

\begin{proof} 
Let $X$ be a maximal submodule of $Y$. Then $Y/X$ is simple injective, thus
$\delta(X) = \delta(Y)-\delta(Y/X) = 0.$ Since $Y$ has no proper non-zero 
preinjective submodule, we see that $X$ has to be regular. According to 
Lemma \ref{characterization},
the inclusion map $X \to Y$ shows that $X$ is even strongly regular, and of course
of length $|Y|-1.$ Conversely, if $X$ is strongly regular and of length $|Y|-1$,
Lemma \ref{characterization} 
yields a monomorphism $X \to Y$. Now assume that two maximal submodules $X, X'$ 
are isomorphic, let $f: X \to Y$ and $f': X' \to Y$ be the inclusion maps. Then,
according to Lemma \ref{characterization}, 
$f,f'$ are right equivalent, thus $X = X'.$

In the same way, we consider the kernels $X$ of the non-zero maps $Y \to Q_1.$
Clearly, all non-zero maps $Y \to Q_1$ are surjective, thus $\delta(X) = 0$ and
again, $X$ has to be regular, and according to Lemma \ref{characterization} 
even strongly regular,
of course of length $|Y|-3$. Conversely, if $X$ is strongly regular and of length
$|Y|-1$, Lemma \ref{characterization} yields a monomorphism $X \to Y$. The factor module $Y/X$ is
of length 3 and has the same composition factors as $Q_1$. But the only factor
module of $Y$ of length 3 with the same composition factors as $Q_1$ is $Q_1$
itself, thus $X$ is the kernel of a non-zero map $Y \to Q_1.$ 
Finally, we use again Lemma \ref{characterization} 
in order to see that isomorphic kernels of
non-zero maps $Y \to Q_1$ are actually identical. 
\end{proof} 

\noindent 
{\bf Remark.} It should be stressed that for $C = P_0$ and $C= P_1$, the inverse $\eta_{CY}^{-1}$
of the Auslander bijection can be seen very well: 

If $M$ is a Kronecker module,
we may write $M$ in the form 
$$
  M = (M_a,M_b;\ \alpha: M_b\to M_a,\beta: M_b\to M_a)
$$ 
where $M_a,M_b$ are vector spaces and $M_\alpha, M_\beta$ are linear maps.
We can identify $M_a$ with $\Hom(P_0,M)$ and $M_b$ with $\Hom(P_1,M)$, since
$P_0 = P(a), P_1 = P(b).$ 

Consider the case $C = P_1.$ Given a maximal submodule $U$ of
$\Hom(C,Y)$, we may interpret it as a maximal submodule of $Y_b= \Hom(C,Y)$, and we may
consider the submodule $\Lambda U$ of $Y$ generated by $U$, let 
$u: \Lambda U \to Y$ be the inclusion map. Then $[u: \Lambda U \to Y\rangle$
belongs to ${}^C[\to Y\rangle_1$ and $\eta_{CV}(u) = U.$ This shows that
$[u\rangle = \eta_{CY}^{-1}(U).$

Similarly, for $C = P_0$, starting with a maximal subspace $V_a$ of 
the vector space $Y_a = \Hom(C,Y)$, 
we may construct $V_b = Y_\alpha^{-1}(V_a) \cap Y_\beta^{-1}(V_a)$,
then $V = (V_a,V_b;Y_\alpha|V_b, Y_\beta|V_b)$ is a submodule of $Y$,
say with inclusion map $v: V \to Y$. 
Then $[v: V \to Y\rangle$
belongs to ${}^C[\to Y\rangle_1$ and $\eta_{CV}(v) = V_a.$ This shows that
$[v\rangle = \eta_{CY}^{-1}(V_a).$
 	
\begin{Prop}\label{main0}
Let $\Lambda$ be the Kronecker algebra 
and $Y$ indecomposable preinjective, $C$ indecomposable preprojective. Let 
$d = |C|+ |Y|-4.$ Let $X$ be any module. 

There is $f: X \to Y$ right minimal,
right $C$-determined with  $|f|_C = 1$, if any only if the isomorphism class of 
$X$ belongs to $\mathcal  R(d)$.

If $f: X \to Y$ and $f': X' \to Y$ are 
right minimal, right $C$-determined maps with  $|f|_C = |f'|_C = 1$. Then $f,f'$ 
are right equivalent if and only if $X,X'$ are isomorphic.
\end{Prop} 
	
\begin{proof} First, consider the case when $C$ is not projective, thus $|C| \ge 5$,
therefore $d > |Y|.$ 

Since $\Lambda$ is hereditary and $C$ has no indecomposable projective direct
summand, any right $C$-determined map $f: X \to Y$ is surjective 
and its kernel belongs to $\add K$, 
where $K = \tau C$. If we assume that $|f|_C = 1$, then the kernel of $f$
has to be equal to $K$. It follows that $\delta(X) = \delta(Y)+\delta(K) = 0.$
Since all indecomposable submodules of $X$ are preprojective or regular, we
see that $X$ has to be regular. Of course, its length is just $d$.
Note that no simple regular submodule $R$ of $X$ can be contained in the kernel
of $f$, since by assumption the kernel of $f$ is $K$, thus preprojective. 
The Lemma \ref{characterization} shows that $X$ is strongly regular, thus the isomorphism class
of $X$ belongs to $\mathcal  R(d)$. 

Conversely, assume that $X$ is a strongly regular
module of length $|X| = d.$ 
According to the Lemma \ref{characterization} there exists a morphism $f: X \to Y$
such that its kernel contains no simple regular module.
Since $|X| = d > |Y|$, the map $f$ cannot be a monomorphism, thus it is an
epimorphism. Since the kernel of $f$ does not contain a simple regular module,
it is the direct sum of indecomposable preprojective modules. Since the defect
of the kernel is $\delta(X)-\delta(Y) = -1$, we see that the kernel of
$f$ is indecomposable preprojective. The length of the kernel is $|X|-|Y| =
|C|-4 = |\tau C|$,
thus the kernel of $f$ is isomorphic to $\tau C$. Altogether, we have shown:
if the isomorphism class of $X$ belongs to $\mathcal  R(d)$, then there is 
$f: X \to Y$ right minimal, right $C$-determined with  $|f|_C = 1$.
 
Finally, assume that $f: X \to Y$ and $f': X' \to Y$ are 
right minimal, right $C$-determined maps with  $|f|_C = |f'|_C = 1$. 
Then $X, X'$ are strongly regular. 
If these maps are right equivalent, then clearly $X,X'$ are isomorphic. 
Conversely, assume that
$X,X'$ are isomorphic. In order to show that $f,f'$ are right equivalent, we may
assume that $X' = X.$ Lemma \ref{characterization} asserts that $f,f'$ are right equivalent.
	
Now assume that $C$ is projective, thus $|C| \le 3$,
therefore $d < |Y|.$ In case $C = P_1$, we have to consider the submodules $X$
of $Y$ with $Y/X = Q_0$, in case $C = P_0$, we have to consider the submodules
$X$ of $Y$ with $Y/X = Q_1$. This has been done in the previous proposition. 
\end{proof} 

\begin{proof}[Proof of proposition \ref{main}]
The proposition is an immediate consequence of \ref{identification} and \ref{main0}.
\end{proof} 
	
\begin{Cor} Let $\Lambda$ be the Kronecker algebra, 
$C$ indecomposable preprojective,
$Y$ indecomposable preinjective. Then $\sigma\eta_{CY}^{-1}$ yields a bijection
\Rahmen{\sigma\eta_{CY}^{-1}:  \mathcal  S_m\Hom(C,Y) \longrightarrow \mathcal  R(i)
 \quad\text{\it with}\quad i=|C|+|Y|-4.}
\end{Cor}
	
Note that $\Gamma(C) = k$, thus $\Hom(C,Y)$, considered as a $\Gamma(C)$-module
is just a vector space, thus $\mathcal  S_m\Hom(C,Y)$ is the set of the maximal
subspaces of a vector space, and therefore a projective space. We see that
{\it we obtain a parameterization of the set $\mathcal  R(2i)$ of all
strongly regular Kronecker modules of length $2i$ by the projective space
$\mathbb P_i$.}
	\bigskip

\noindent 
{\bf Remark.} This parameterization of $\mathcal  R(d)$ extends the well-known 
description of the geometric quotient of the ``open sheet'', when dealing
with conjugacy classes of $(n\times n)$-matrices with coefficients in a
field, see for example Kraft \cite{[Kt]}. 

Let us mention some details: Let $R = k[T]$
be the polynomial ring in one variable $T$ with coefficients in the field $k$,
and let us assume that $k$ is algebraically closed. 
We consider $R$ as the path algebra of the quiver with one vertex and one loop; in this way, the 
$R$-modules of dimension $n$ are just pairs $(V,\phi)$, where $V$ is a $k$-space
of dimension $n$ and $\phi: V \to V$ an endomorphism of $V$, or, after 
choosing a basis of $V$, we just deal with $(n\times n)$-matrices with
coefficients in $k$. Isomorphism of $R$-modules translates to equivalence
(or conjugacy) of matrices. The assertions of Lemma \ref{strongly-regular} can be reformulated
in this context, {\it the properties mentioned there characterize just the
cyclic $R$-modules of finite length.} Note that the category of 
$k[T]$-modules of finite length is equivalent to the full subcategory $\mathcal  R'$ 
of all regular Kronecker modules without eigenvalue $\infty$
(to be precise: let us  denote the two arrows 
of the Kronecker quiver by $\alpha,\beta$
and let $\mathcal  T_\infty$ be the Auslander-Reiten component which 
contains the indecomposable regular 
representation $M$ with $M_\alpha = 0$. The subcategory $\mathcal  R'$ consists
of all regular representations with no indecomposable direct summand 
in $\mathcal  T_\infty$, or equivalently, it is the full subcategory of
all representations $M$ such that $M_\alpha$ is bijective. 
Under the equivalence
of the category of finite-dimensional $k[T]$-modules and $\mathcal  R'$
$k[T]$-module $N$ corresponds to the representation $M$ such that $M_\alpha: N \to N$
is the identity map, and $M_\beta: N \to N$ the multiplication by $T$). 
A representation $M$ in $\mathcal  R'$ is strongly regular if and only if $M$
corresponds under this equivalence to a cyclic $k[T]$-module. 
Thus, one may be tempted to call
the strongly regular Kronecker modules ``cyclic'' modules, but this 
would be in conflict with standard terminology.
	\bigskip 

\noindent 
{\bf Remark. ``Modules determined by morphisms''.}
A bijection of two sets can always be read in two different directions.
This survey is concerned with the Auslander bijection
$$
 \eta_{CY}: {}^C[\to Y\rangle \longrightarrow \mathcal  S\Hom(C,Y),
$$
In the previous sections, the focus was going from left to right: Any right minimal
morphism $f\in {}^C[\to Y\rangle$ yields under $\eta_{CY}$ a submodule of $\Hom(C,Y)$ 
and is uniquely determined by this submodule,
this is the philosophy of saying that morphisms as elements of $[\to Y\rangle$)
are determined by modules.

The considerations in the present section point into the reverse direction: 
we use sets of morphisms
as convenient indices for parameterizing isomorphism classes of modules. 
We recall from section \ref{sec:7} that the restriction of $\eta_{CY}$ furnishes a bijection
$$
\eta_{CY}: {}^C[\to Y\rangle^1 \longrightarrow \mathcal  S_m\Hom(C,Y),
$$
thus we can use the right hand set
in order to parametrize the left hand set.

Our interest lies in the
special case where maps $f: M \to Y$ and $f': M' \to Y$ of right $C$-length $1$
are right equivalent only in case $M$ and $M'$ are isomorphic. In this case,
the maximal submodule $\eta_{CY}(f: M \to Y)$ (thus a set of morphisms)
uniquely determines the module $M$. In addition, we will assume that $C$ is a brick,
or at least that $C$ is indecomposable and $\rad\End(C)$ annihilates $\Hom(C,Y)$.
In this case, $\mathcal  S_m\Hom(C,Y)$ is a projective space (namely the projective
$d$-space over the division ring $\End(C)/\rad\End(C)$, provided $\Hom(C,Y)$
is a module of length $d+1$). 
	\bigskip
	
Let us determine $\eta_{CY}^{-1}(0)$ for $C = P_i,\ Y = Q_j$
with $i,j\in \mathbb N_0$. Note that
$$
 \dim\Hom(P_{i-1},Q_j) = i+j-1,
$$ 
where $P_{-1} = 0$. Thus, the universal map from $\add P_{i-1}$ 
to $Q_j$ is of the form $P_{i-1}^{i+j-1} \to Q_j$ (for $j=0$, it is a map
of the form $P_{i-1}^{i-1} \to Q_0$). 
	
\begin{Prop}\label{universal}
Let $\Lambda$ be the Kronecker algebra.
Let $C = P_i,\ Y = Q_j$ with $i,j\in \mathbb N_0$. Let 
$f: P_{i-1}^{i+j-1} \to Q_j$ be the universal map from $\add P_{i-1}$ 
to $Q_j$. Then 
\Rahmen{\eta_{CY}^{-1}(0) = [f: P_{i-1}^{i+j-1} \to Q_j\rangle.}
\end{Prop} 

\begin{proof} 
For $i = 0,$ the intersection of the kernels of the maps $Q_j \to Q_1$ is zero. 
For $i = 1,$ the module $P_0 = P_{i-1}$ is simple projective, thus, for $j \ge 1$,
the universal map $f: P_0^{j} \to Q_j$ is just the embedding of the socle of $Q_j$
into $Q_j$ and the socle is the intersection of the kernels of the maps
$Q_j \to Q_0$. For $i=1$ and $j= 0$, the intersection of the kernels of the maps
$Q_j \to Q_0$ is zero, but also $P_{i-1}^{i+j-1}$ is zero. 

Assume now that $i\ge 2.$ We have 
$$
 \dim\Ext^1(Q_j,P_{i-2}) = \dim \Hom(P_{i},Q_j)= \dim\Hom(P_{i+j},Q_0) = i+j,
$$
thus the universal extension of $Q_j$ from below using copies of $P_{i-2}$
looks as follows
$$
 0 \to  P_{i-2}^{i+j} \to  X \overset f\to  Q_j \to   0
$$
with a module $X$ such that $\Ext^1(X,X) = 0.$ The dimension vector of $X$
is
$$
 \bdim X = (i+j)\bdim P_{i-2} + \bdim Q_j = (i+j-1)\bdim P_{i-1}.
$$
Since also $\Ext^1(P_{i-1},P_{i-1}) = 0,$ 
it follows that $X = P_{i-1}^{i+j-1}.$ Since $f$ is right minimal and
$\dim\Hom(P_{i-1},Q_j) = i+j-1$, we see that $f$ has to be the 
universal map from $\add P_{i-1}$ to $Q_j$.
\end{proof} 

\noindent 
{\bf Example 11:} We consider the special case of
the Kronecker modules $C = P_4,\ Y = Q_0.$ Note that $\bdim Y = (0,1)$ and 
$\bdim P_4 = (5,4)$, thus $\tau C = P_2$ has dimension vector $(3,2)$
and $\Hom(\tau C,Y) = 2$. Here is a sketch of $^C[\to Y\rangle$.
In any layer $^C[\to Y\rangle_t$ with $1\le t \le 3$,
we indicate the elements $[f:X \to Y\rangle$ 
in the form $P_2^t \to  \bdim X \to  Y$. The map in the layer $t=0$
has been described in Proposition \ref{universal}.

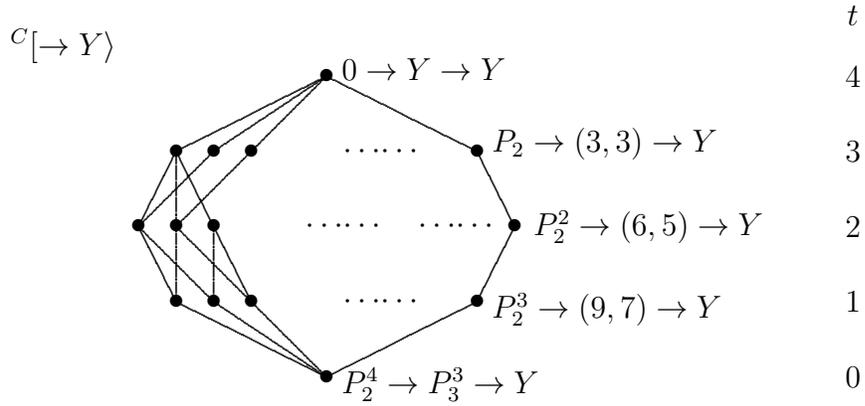
$$
{\beginpicture
\setcoordinatesystem units <1cm,1cm>
\multiput{$\bullet$} at 0.5 2  1 1  1 2  1 3  1.5 1  1.5 2  1.5 3  3 0  2 1  2 3  3 4  
   5 1  5 3  5.5 2  /
\plot 3 0  1 1  0.5 2  1 3  3 4   5 3  5.5 2  5 1  3 0 /
\plot 1 1  1 3 /
\plot 3 0  1.5 1  0.5 2  /
\plot 1.5 3   3 4 /
\plot 1.5 3  0.5 2 /
\plot  3 4  2 3  1 2  2 1  3 0 /
\plot  1 3  1.5 2  1.5 1  /
\plot 2 1  1.5 2 /
\multiput{$\cdots$} at 4 1  4 3  3.5 2  5 2 /
\multiput{} at -2 0  10 0 /
\put{$0$} at 10 0
\put{$1$} at 10 1
\put{$2$} at 10 2
\put{$3$} at 10 3
\put{$4$} at 10 4
\put{$t$} at 10 4.8
\put{$^C[\to Y\rangle$} at -.5 4.4
\put{$0\to Y\to Y$} [l] at 3.2  4.1
\put{$P_2 \to (3,3) \to Y$} [l] at 5.2  3.1
\put{$P_2^2 \to (6,5) \to Y$} [l] at 5.75  2
\put{$P_2^3 \to (9,7) \to Y$} [l] at 5.2  0.9
\put{$P_2^4 \to P_3^3 \to Y$} [l] at 3.2  -0.1
\multiput{$\cdots$} at 3.5 1  3.5 3  3 2  4.5 2 /
\multiput{} at -2 0  10 0 /
\endpicture}
$$
Let us exhibit some of the short exact sequences $P_1^t \to X \to Y$.

For $(3,3)$, the module $X$ must be a strongly regular module, thus, 
there are three different kinds: 
direct sums of three pairwise 
non-isomorphic indecomposable modules of length 2,
direct sums of an indecomposable module $M$ of length 4 and an 
indecomposable module $M'$ of length 2 such that $\Hom(M,M') = 0$, 
and finally indecomposable modules of length 6.

For $(6,5)$ and $(9,7)$, we deal with direct sums of preprojective and 
regular modules. We consider the case $\bdim X = (6,5)$, thus $t=2$,
in detail. Let $X = X_a, X_b; X_\alpha, X_\beta)$ 
be a Kronecker module with a submodule $U$ isomorphic to
$P_2^2$ such that $X/U$ is isomorphic to $Q_0$.

We start with a basis of $U_b$ and add an element $x$ 
in order to obtain a basis of $X_b$, thus
$X_a = U_a,$ and $X_b = U_b\oplus \langle x\rangle.$
What we have to describe are the elements $\alpha(x)$
and $\beta(x)$ in $U_a$, and we have to provide a 
decomposition of $X$ into indecomposables. In this way, we
also will see that the map $X \to Q_0$ with kernel
$U$ is right minimal. 

Since $U$ is isomorphic to $P_2^2$, we can exhibit $U$ as follows:
$U_a$ has a basis $v_0,\dots,v_5$ and $U_b$ has a basis
$u_1,u_2,u_4,u_5$ and  $\alpha(u_i) = v_{-1}$ and
$\beta(u_i) = v_i.$
In order to describe $X$, we have to discuss possible values for $\alpha(x)$
and $\beta(x)$. 
Four different cases will be of interest. 
In the first three cases, let $\alpha(x) = v_2$.

(1) If we define $\beta(x) = v_3$, then clearly $X = P_5.$

(2) If we define $\beta(x) = v_4$, then we get a decomposition of $X$
as follows: The elements $u_1,u_2,x,u_5,v_0,v_1, v_2,v_4,v_5$
yield a submodule of the form $P_4$, the elements $x-u_4, v_2-v_3$
also yield a submodule (since $\beta(x-u_4) = 0$). These two
submodules provide a direct decomposition.

(3) If we define $\beta(x) = v_5,$ then we see that the elements
$u_1,u_2,x;v_0,v_1, v_2, v_5$ yield a submodule of the form $P_3$.
The elements $u_2-u_4, x-u_5, v_1-v_3, v_2-v_5$ yield a 4-dimensional
indecomposable submodule, and we obtain in this way a direct decomposition.

(4) Finally, let $\alpha(x) = v_1, \beta(x) = v_4$. Then we get a submodule
of $X$ 
with basis $u_1,x,u_5, v_0,v_1,v_4,v_5$ which is of the form $P_3$
as well as two indecomposable submodules of length 2, namely with basis
$x-u_1, v_4-v_2$ and with basis $x-u_4, u_1-u_3$.

In this way, we obtain short exact sequences
$$
 0 \to P_2^2 \to X \to  Q_0 \to 0
$$
such that $X$ is of the form $P_5,$ of the form $P_4\oplus R$,
of the form $P_3\oplus R''$, or finally of the form $P_3\oplus R\oplus R'$,
where $R, R'$ are both regular of length 2, and $R''$ is regular of length 4.

\section{Lattices of height at most 2.}
\label{sec:15}

The height of $\mathcal  S\Hom(C,Y)$ is the length of the
$\Gamma(C)$-module $\Hom(C,Y)$.
If $\Lambda$ is a $k$-algebra, $k$ is algebraically closed,
and $C$ is multiplicity-free (what we can assume),
then the height of $\mathcal  S\Hom(C,Y)$ is the $k$-dimension of $\Hom(C,Y)$.

Our main interest will concern the height 2 lattices which are not distributive,
since this is the first time that one may encounter infinite families. 
Here is a discussion of the lattices of height at most 2, in
general. 
	\medskip 

\noindent 
{\bf Height 0.} The lattice $\mathcal  S M$ has height 0 if and only if $M = 0.$ Thus, to 
say that $\mathcal  S\Hom(C,Y)$ has height zero means that $\Hom(C,Y) = 0$.
		\medskip 

\noindent 
{\bf Height 1.} The lattice $\mathcal  S M$ has height 1 if and only if $M$
is a simple module. Thus, in our case $M = \Hom(C,Y)$,
we deal with a simple $\Gamma(C)$-module. Note that $\Hom(C,Y)$ is a simple
$\Gamma(C)$-module if any only if there is a right minimal, right $C$-determined
map $f: X \to Y$ which is not an isomorphism, such that for any 
right minimal, right $C$-determined
map $f': X' \to Y$ which is not an isomorphism, there is an
isomorphism $h: X \to X'$ such that $f = f'h.$
	
Three cases should be noted: 

(1) $f$ may be an epimorphism. For example, take the path algebra
 of the quiver of type $\mathbb A_2$ as exhibited in example 1.
 Let $C = Y = S(b)$ (thus $\tau C = S(a)$). 
 Then the epimorphism $f: P(b) \to S(b) = Y$ is, up to isomorphism,
the only right minimal, right $C$-determined
 map ending in $Y$ which is not an isomorphism.

(2) $f$ may be a monomorphism. To obtain an example, take again 
the path algebra  of the quiver of type $\mathbb A_2$.
 Let $C = S(a)$ and $Y = P(b)$, thus now $C$ is projective.
 The monomorphism $f: S(a) \to P(b)= Y$ is, up to isomorphism,
the only right minimal, right $C$-determined
 map ending in $Y$ which is not an isomorphism.

(3) $f$ is neither epi nor mono. As an example, take the radical square zero
algebra with the linearly oriented quiver of type $\mathbb A_3$, see example 6.
Let $Y= P(c)$ and $C = S(b)$ (thus $\tau C = S(a)$).
The non-zero map $f: P(b) \to P(c)$ is, up to right equivalence,
the only right minimal, right $C$-determined
map ending in $Y$ which is not an isomorphism.
	\medskip

\noindent 
{\bf Height 2.} A lattice of height 2 may be either a chain (thus of the form 
$\mathbb I(2)$) or not. If the lattice is not a chain, the lattice still may be distributive
(case III) or not (case IV). The submodule lattices $\mathcal  SM$ of type III and IV 
occur for a semisimple module $M$ of length 2; in case $M$ is the
direct sum of two non-isomorphic simple modules, we deal with case III, 
otherwise $M$ is the direct sum of two isomorphic simple modules and then we deal
with case IV. Since the lattices we are interested are submodule lattices (here
of a module of length 2), we distinguish also in the case of a serial module $M$
of length 2, whether the two composition factors are isomorphic (case I) or 
not (case II). 

Altogether, we see that for a lattice $\mathcal  S\Hom(C,Y)$ of height 2,
there are the following four cases:
$$
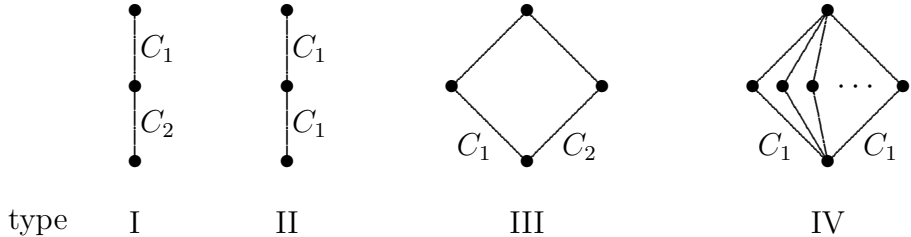
{\beginpicture
\setcoordinatesystem units <1cm,1cm>
\put{\beginpicture
\plot 0 0  0 2 /
\multiput{$\bullet$} at 0 0  0 1  0 2 /
\put{$C_1$} at 0.3 1.5 
\put{$C_2$} at 0.3 0.5 
\put{I} at 0 -.8
\endpicture} at 0 0
\put{\beginpicture
\plot 0 0  0 2 /
\multiput{$\bullet$} at 0 0  0 1  0 2 /
\put{$C_1$} at 0.3 1.5 
\put{$C_1$} at 0.3 0.5 
\put{II} at 0 -.8
\endpicture} at 2 0
\put{\beginpicture
\plot 0 0  -1 1   0 2  1 1  0 0  /
\multiput{$\bullet$} at 0 0  -1 1  1 1  0 2 /
\put{$C_1$} at -.7 0.2 
\put{$C_2$} at 0.7 0.2 
\put{III} at 0 -.8
\endpicture} at 5 0
\put{\beginpicture
\plot 0 0  -1 1   0 2  1 1  0 0  /
\multiput{$\bullet$} at 0 0  -1 1  1 1  0 2 
 -0.6 1  -0.2 1  /
\plot 0 0  -.6 1  0 2  -.2 1  0 0 /
\put{$\cdots$} at .4 1
\put{$C_1$} at -.7 0.2 
\put{$C_1$} at 0.7 0.2 
\put{IV} at 0 -.8

\endpicture} at 9 0

\put{\beginpicture
\multiput{} at 0 2.15 /
\put{type} at 0 -.8
\endpicture} at -1.5 0

\endpicture}
$$
here, the labels $C_1,C_2$ concern the type of the corresponding pair of neighbors: 
$C_1,C_2$ are non-isomorphic indecomposable modules.
	\medskip 

\noindent 
{\bf Type I: Example 12.} Take the linearly oriented quiver of type $\mathbb A_3$ as
discussed in Examples 5. 
Let $Y = S(3)$ and $C = Q(2)\oplus S(3)$. Then
$\Gamma(C)$ is the path algebra of the quiver of type $\mathbb A_2$ and
$\Hom(C,Y)$ is the indecomposable $\Gamma(C)$ module of length $2$.
The lattices ${}^C[\to Y\rangle$ and $\mathcal  S\Hom(C,Y)$ look as follows:
$$
{\beginpicture
\setcoordinatesystem units <1cm,1.3cm>
\put{\beginpicture
\put{$Y = S(3)$} at 0 2
\put{$Q(2)$} at 0 1
\put{$Q(1)$} at 0 0
\arr{0 0.3}{0 0.7}
\arr{0 1.3}{0 1.7}
\put{$f'$} at -.3 1.5
\endpicture} at 0 0
\put{\beginpicture

\multiput{$\bullet$} at 0 0  0 1  0 2 /
\plot 0 0  0 2 /
\put{$\Hom(C,Y)$} [l] at 0.2 2
\put{$f'\Hom(C,Q(2))$} [l] at 0.3 1
\put{$0$} [l] at 0.2 0
\put{$\bigcirc$} at 0 1

\endpicture} at 4 0
\endpicture}
$$

{\bf Type II.} In the last section, we have seen such examples 
for the Kronecker algebra, namely:
if $C,Y$ are modules with $C$ indecomposable, 
such that ${}^C[\to Y\rangle$ is of the form
$\mathbb I(2)$, then we must be in type II.
	\medskip 

Another example has been presented already in section 7
when dealing
with Riedtmann-Zwara degenerations, namely example 9. There, 
we have chosen (non-projective) 
indecomposable modules
$C$ and $Y$ such that $\Hom(C,Y)$ was a cyclic module
of length $2$. There, an additional module $Y$ was considered with an
epimorphism $M' \to Y$ such that the composition $M \to M' \to Y$ is
non-zero. Note that this procedure fits into the consideration of families 
$\mathcal  M = \{M_i\mid i\in I\}$
of modules which is based on dealing with fixed morphisms $f_i: M_i \to Y$
for some module $Y$.
	\bigskip

\noindent 
{\bf Type III. Example 13.} 
Take the quiver of type $\mathbb A_3$ with two sinks:
$$
{\beginpicture
\setcoordinatesystem units <.8cm,.6cm>
\put{$b$} at 1 0
\put{$a_1$} at 0 1
\put{$a_2$} at 0 -1
\arr{0.75 0.2}{0.25 0.85} 
\arr{0.75 -.2}{0.25 -.85} 
\endpicture}
$$
and let $\Lambda$ be its path algebra. Here is the Auslander-Reiten quiver of $\Lambda$:
$$
{\beginpicture
\setcoordinatesystem units <1cm,.8cm>
\put{\beginpicture
\put{$S(a_1)$} at 0 1
\put{$P(b)$} at 1 0
\put{$S(a_2)$} at 0 -1
\arr{0.3 0.7}{0.7 0.3}
\arr{0.3 -.7}{0.7 -.3}
\arr{1.3 0.3}{1.7 .7} 
\arr{1.3 -.3}{1.7 -.7} 

\put{$Q(a_2)$} at 2 1
\put{$S(b)$} at 3 0
\put{$Q(a_1)$} at 2 -1
\arr{2.3 0.7}{2.7 0.3}
\arr{2.3 -.7}{2.7 -.3}

\endpicture} at 6 0
\endpicture}
$$

Let $Y = S(b)$ and $C = Q(a_1)\oplus Q(a_2)$. As usual, let us show
both ${}^C[\to Y\rangle$ as well as $\mathcal  S\Hom(C,Y):$
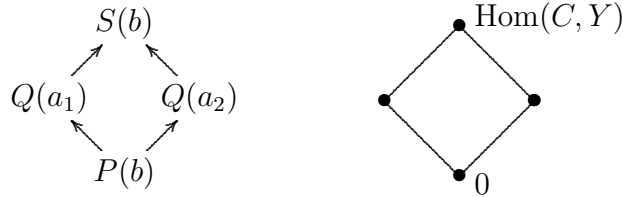
$$
{\beginpicture
\setcoordinatesystem units <1cm,1cm>
\put{\beginpicture
\put{$S(b)$} at 0 2
\put{$Q(a_1)$} at -1 1
\put{$Q(a_2)$} at  1 1
\put{$P(b)$} at 0 0
\arr{-.3 0.3}{-.7 0.7}
\arr{-.7 1.3}{-.3 1.7}
\arr{.3 0.3}{.7 0.7}
\arr{.7 1.3}{.3 1.7}
\endpicture} at 0 0
\put{\beginpicture
\put{$\Hom(C,Y)$} [l] at 0.2 2.1
\multiput{$\bullet$} at 1 1  0 0  -1 1  1 1  0 2  /
\put{$0$} [l] at 0.2 -.1
\plot 0 0  -1 1  0 2  1 1  0 0 /
\endpicture} at 5 0
\endpicture}
$$

{\bf Example 14.} Here is a second example of type III.
In contrast to the previous
example, here the two incomparable right minimal maps $f: X \to Y$ and
$f': X' \to Y$ 
have the property that the modules $X$ and $X'$ are isomorphic. We consider
the Kronecker algebra (see section 11).
Let $Y = S(b)$ and $C = R\oplus R'$, where $R, R'$ are non-isomorphic regular
modules of length 2. We may consider $R, R'$ as submodules of $Q(a)$
and we denote by $f: Q(a) \to S(b)$ the projection with kernel $R$, by
$f': Q(a) \to S(b)$ the projection with kernel $R'$. The pullback of $f$ and $f'$
is the preinjective module $Q_2$ (with dimension vector $(2,3)$).
$$
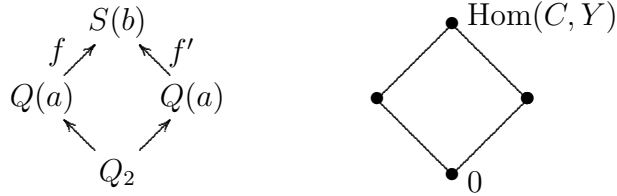
{\beginpicture
\setcoordinatesystem units <1cm,1cm>
\put{\beginpicture
\put{$S(b)$} at 0 2
\put{$Q(a)$} at -1 1
\put{$Q(a)$} at  1 1
\put{$Q_2$} at 0 0
\put{$f$} at -.8 1.6
\put{$f'$} at .85 1.6
\arr{-.3 0.3}{-.7 0.7}
\arr{-.7 1.3}{-.3 1.7}
\arr{.3 0.3}{.7 0.7}
\arr{.7 1.3}{.3 1.7}
\endpicture} at 0 0
\put{\beginpicture
\put{$\Hom(C,Y)$} [l] at 0.2 2.1
\multiput{$\bullet$} at 1 1  0 0  -1 1  1 1  0 2  /
\put{$0$} [l] at 0.2 -.1
\plot 0 0  -1 1  0 2  1 1  0 0 /
\endpicture} at 5 0
\endpicture}
$$
	\medskip

\noindent 
{\bf Type IV.} Let $C,Y$ be a pair of modules such that ${}^C[\to Y\rangle$
is of type IV. As we will see, and this is our main concern, the behaviour of the
modules present in ${}^C[\to Y\rangle ^1$ may be quite different.
		\medskip 

\noindent 
{\bf Examples 15, where the modules present in ${}^C[\to Y\rangle ^1$ are all isomorphic.}
As we have seen in section 11, there are many
such examples for the Kronecker algebra, thus let $\Lambda$ be the Kronecker algebra,
let $P_0,P_1,P_2,\dots$ be the indecomposable preprojective modules, and
$Q_0, Q_1, Q_2,\dots$ the indecomposable preinjective modules,
with both $P_i$ and $Q_i$ being of length $2i+1.$ 
	\medskip 

First, let $C = P_i$ and $Y = P_{i+1}$, for some $i \ge 0$, 
thus $\dim\Hom(C,Y) = 2$ and therefore $\mathcal  S\Hom(C,Y)$ is of the form IV.

If $i=0$, then we deal with the lattice of submodules $U$ of $P_1$ such that
the socle of $P_1/U$ is generated by $P_0$. Such a submodule is either $0$ or simple
(thus of the form $P_0$) or equal to $P_1$:
$$
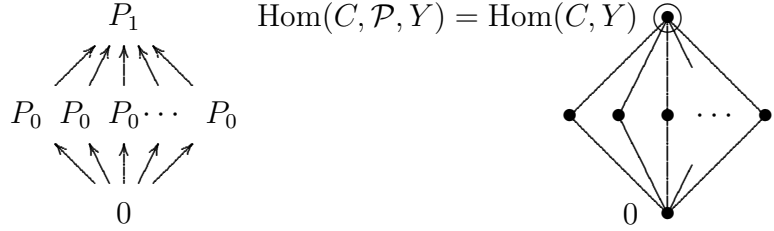
{\beginpicture
\setcoordinatesystem units <1.3cm,1.3cm>
\put{\beginpicture

\put{$P_1$} at 2 2
\multiput{$P_0$} at 1 1  3 1  1.5 1   2 1   /
\put{$0$} at 2 0
\put{$\cdots$} at 2.4 1 
\arr{1.7 0.3}{1.3 0.7}
\arr{2.3 0.3}{2.7 0.7}
\arr{2.7 1.3}{2.3 1.7}
\arr{1.3 1.3}{1.7 1.7}
\arr{1.85 0.3}{1.65 0.7}  
\arr{1.65 1.3}{1.85 1.7}
\arr{2 0.3}{2 0.7}
\arr{2 1.3}{2 1.7}
\arr{2.15 0.3}{2.35 0.7}
\arr{2.35 1.3}{2.15 1.7}
\endpicture} at 0 0
\put{\beginpicture
\multiput{$\bullet$} at  1 1  2 0  3 1  2 2   1.5 1  2 1   /
\put{$\Hom(C,\mathcal  P,Y) = \Hom(C,Y)$} [r] at 1.7 2 
\put{$0$} [r] at 1.7 0
\put{$\cdots$} at 2.5 1 
\plot 1 1  2 0  3 1  2 2  /
\plot 1 1  2 2 /
\plot 2 0  1.5 1  2 2 /
\plot 2 0  2 1  2 2 /
\plot 2 0  2.25 .5   /
\plot 2.25 1.5  2 2 /
\put{$\bigcirc$} at 2 2 
\endpicture} at 4 0
\endpicture}
$$

If $i=1$, then we deal with the lattice of all submodules $U$ of $P_2$ which contain
the socle of $P_2$ (these are just the submodules $U$ of $P_2$ such that the socle
of $P_2/U$ is generated by $P_1$). Note that such a submodule is either the socle itself,
thus isomorphic to $P_0{}^3$, or of the form $N = P_0\oplus P_1$, or $P_2$ itself.
$$
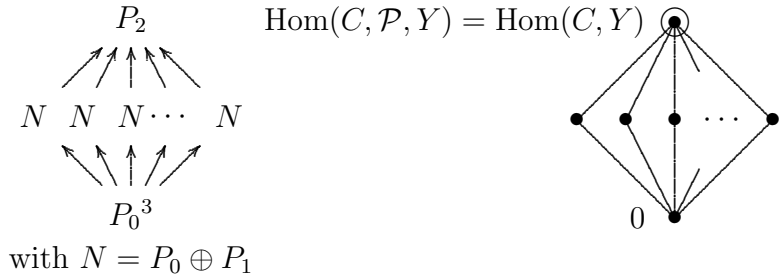
{\beginpicture
\setcoordinatesystem units <1.3cm,1.3cm>
\put{\beginpicture

\put{$P_2$} at 2 2
\multiput{$N$} at 1 1  3 1  1.5 1   2 1   /
\put{$P_0{}^3$} at 2 0
\put{$\cdots$} at 2.4 1 
\arr{1.7 0.3}{1.3 0.7}
\arr{2.3 0.3}{2.7 0.7}
\arr{2.7 1.3}{2.3 1.7}
\arr{1.3 1.3}{1.7 1.7}
\arr{1.85 0.3}{1.65 0.7}  
\arr{1.65 1.3}{1.85 1.7}
\arr{2 0.3}{2 0.7}
\arr{2 1.3}{2 1.7}
\arr{2.15 0.3}{2.35 0.7}
\arr{2.35 1.3}{2.15 1.7}
\endpicture} at 0 0
\put{\beginpicture
\multiput{$\bullet$} at  1 1  2 0  3 1  2 2   1.5 1  2 1   /
\put{$\Hom(C,\mathcal  P,Y) = \Hom(C,Y)$} [r] at 1.7 2 
\put{$0$} [r] at 1.7 0
\put{$\cdots$} at 2.5 1 
\plot 1 1  2 0  3 1  2 2  /
\plot 1 1  2 2 /
\plot 2 0  1.5 1  2 2 /
\plot 2 0  2 1  2 2 /
\plot 2 0  2.25 .5   /
\plot 2.25 1.5  2 2 /
\put{$\bigcirc$} at 2 2 
\endpicture} at 4 0
\put{with $N = P_0\oplus P_1$} at 0 -1.45
\endpicture}
$$

Whereas for the cases $i=0,1$ 
all the maps shown in the lattice ${}^C[\to Y\rangle$
are inclusion maps, 
the maps exhibited in ${}^C[\to Y\rangle$ for $i\ge 2$ are all epimorphisms,
see Proposition 6.6 and Lemma 6.8.

So let us assume that $i\ge 2$. 
A map $f: X \to Y$
is right $P_i$-determined if and only if its kernel is a direct sum of copies
of $P_{i-2}$. Since $\Ext^1(P_{i+1},P_{i-2}) =D\Hom(P_i,P_{i+1})$ is a
two-dimensional vector space, we see that we deal with short exact sequences of the form
$$
{\beginpicture
\setcoordinatesystem units <1.3cm,.7cm>
\put{$0 \to P_{i-2} \to N \to P_{i+1} \to 0$} at 0 1
\put{$0 \to P_{i-2}{}^2 \to M \to P_{i+1} \to 0$} at 0 0 
\endpicture}
$$
such that the maps $N\to P_{i+1}$ and $M \to P_{i+1}$ are right minimal.
It follows easily that all these modules $N$ have to be of the form $P_{i-1}\oplus P_i$,
and $M$ has to be of the form $P_{i-1}{}^3$. Thus, the 
lattice ${}^C[\to Y\rangle$ looks as 
follows: 
$$
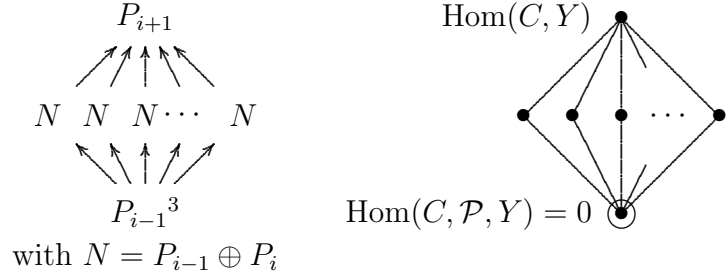
{\beginpicture
\setcoordinatesystem units <1.3cm,1.3cm>
\put{\beginpicture

\put{$P_{i+1}$} at 2 2
\multiput{$N$} at 1 1  3 1  1.5 1   2 1   /
\put{$P_{i-1}{}^3$} at 2 0
\put{$\cdots$} at 2.4 1 
\arr{1.7 0.3}{1.3 0.7}
\arr{2.3 0.3}{2.7 0.7}
\arr{2.7 1.3}{2.3 1.7}
\arr{1.3 1.3}{1.7 1.7}
\arr{1.85 0.3}{1.65 0.7}  
\arr{1.65 1.3}{1.85 1.7}
\arr{2 0.3}{2 0.7}
\arr{2 1.3}{2 1.7}
\arr{2.15 0.3}{2.35 0.7}
\arr{2.35 1.3}{2.15 1.7}
\endpicture} at 0 0
\put{with $N = P_{i-1}\oplus P_i$} at 0 -1.45
\put{\beginpicture
\multiput{$\bullet$} at  1 1  2 0  3 1  2 2   1.5 1  2 1   /
\put{$\Hom(C,Y)$} [r] at 1.7 2 
\put{$\Hom(C,\mathcal  P,Y) = 0$} [r] at 1.7 0
\put{$\cdots$} at 2.5 1 
\plot 1 1  2 0  3 1  2 2  /
\plot 1 1  2 2 /
\plot 2 0  1.5 1  2 2 /
\plot 2 0  2 1  2 2 /
\plot 2 0  2.25 .5   /
\plot 2.25 1.5  2 2 /
\put{$\bigcirc$} at 2 0 
\endpicture} at 4 0
\endpicture}
$$

Let us describe the right minimal maps $N = P_1\oplus P_2 \to P_3$ in detail. We have
$\dim\Hom(P_1\oplus P_2,P_3) = 3,$ but actually only the homomorphisms $P_2 \to P_3$
matter and $\Hom(P_2,P_3) = 2.$ Why only the homomorphisms $P_2 \to P_3$
matter? We need an epimorphism $f: P_1\oplus P_2 \to P_3,$ 
write it as $f = [f_1,f_2]$ with $f_1: P_1 \to P_3$ and
$f_2: P_2 \to P_3.$ In order that $f$ is an epimorphism, the following two conditions
have to be satisfied:

(1) The restriction $f_2$ of $f$ to $P_2$ has to be non-zero.

(2) The image $f_1(P_1)$ is not contained in $f_2(P_2)$, or, equivalently
(since $P_1$ is projective) $f_1$ does not factor through $f_2.$
 
If two such maps $[f_1,f_2]$ and $[f'_1,f'_2]$ are given with both $f_2$ and $f'_2$
non-zero, and  $f_1(P_1) \not\subset f_2(P_2)$ as well as
$f'_1(P_1) \not\subset f'_2(P_2)$, then $[f_1,f_2]$ is right equivalent to 
$[f'_1,f'_2]$ if and only if $f_2$ is right equivalent to $f'_2$, if and only if
there is a scalar $c\in k^*$ such that $f'_2 = cf_2.$ 

It follows that the existence of the
one-parameter family of right minimal maps $N \to P_3$ 
comes from the fact that $\dim\Hom(P_2,P_3) = 2.$ 
	\medskip 	

Observe that the lattices ${}^C[\to Y\rangle$ for $C = P_i, Y = P_{i+1}$ and
all $i\ge 0$ have the same form, provided we set $P_{-1} = 0$, 
	\medskip 

\noindent 
{\bf Examples 16, where the modules in ${}^C[\to Y\rangle ^1$ are pairwise
non-isomorphic.}
Again, we deal with the Kronecker algebra.
Let $Y = Q_0$ and $C = P_2$, thus
$\dim\Hom(C,Y) = 2$ and $K = \tau C = P_0$.
The right minimal right $C$-determined morphisms ending in $Y$ are 
epimorphisms with kernel in $\add K$. Here are the lattices in question:
$$
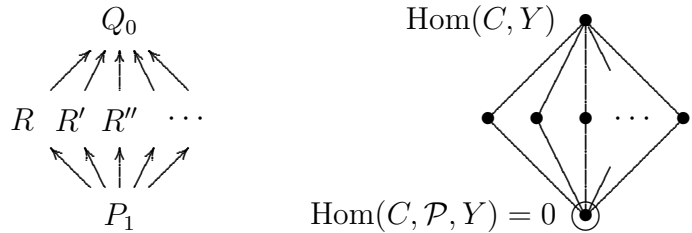
{\beginpicture
\setcoordinatesystem units <1.3cm,1.3cm>
\put{\beginpicture

\put{$Q_0$} at 2 2
\put{$R$} at 1 1 
\put{$R'$} at 1.5 1  
\put{$R''$} at 2 1  
\put{$P_1$} at 2 0
\put{$\cdots$} at 2.7 1 
\arr{1.7 0.3}{1.3 0.7}
\arr{2.3 0.3}{2.7 0.7}
\arr{2.7 1.3}{2.3 1.7}
\arr{1.3 1.3}{1.7 1.7}
\arr{1.85 0.3}{1.65 0.7}  
\arr{1.65 1.3}{1.85 1.7}
\arr{2 0.3}{2 0.7}
\arr{2 1.3}{2 1.7}
\arr{2.15 0.3}{2.35 0.7}
\arr{2.35 1.3}{2.15 1.7}
\endpicture} at 0 0
\put{\beginpicture
\multiput{$\bullet$} at  1 1  2 0  3 1  2 2   1.5 1  2 1   /
\put{$\Hom(C,Y)$} [r] at 1.7 2 
\put{$\Hom(C,\mathcal  P,Y) = 0$} [r] at 1.7 0
\put{$\cdots$} at 2.5 1 
\plot 1 1  2 0  3 1  2 2  /
\plot 1 1  2 2 /
\plot 2 0  1.5 1  2 2 /
\plot 2 0  2 1  2 2 /
\plot 2 0  2.25 .5   /
\plot 2.25 1.5  2 2 /
\put{$\bigcirc$} at 2 0 
\endpicture} at 4 0
\endpicture}
$$
with  pairwise non-isomorphic indecomposable
representations $R,R',R'',\dots$ of length 2.
	\bigskip

\noindent 
{\bf Examples 17, where the modules in ${}^C[\to Y\rangle^1$ belong to 
a finite number of isomorphism classes with a fixed dimension vector.}

Take the $3$-subspace quiver as considered in example 9. 
Let $Y = Q(a)$, and $C$  the maximal indecomposable module, thus $C = \tau Q(a)$
and $\dim\Hom(C,Y) = 2.$ Then $\mathcal  S\Hom(C,Y)$ is 
the non-distributive lattice of height $2$, and the elements of height 1
form a $\mathbb P_1$-family.

This $\mathbb P_1$-family in $\mathcal  S\Hom(C,Y)$
contains three special elements, namely
the three subspaces which are 
generated by the composition of irreducible maps $C \to \tau^{-}P(b_i)$ and
$\tau^{-}P(b_i) \to Y,$ for $1\le i \le 3.$
The remaining elements of the $\mathbb P_1$-family are generated by the
surjective maps $C \to Y$. 

Correspondingly, in ${}^C[\to Y\rangle_1$, there are the right equivalence classes
$[f_i\rangle$  of surjective maps $f_i: M(i) \to Y$, where 
$M(i) = P(b_i)\oplus\tau^{-}P(b_i)$ for $1\le i \le 3$; 
the remaining elements of the $\mathbb P_1$-family are the right equivalence classes of
the surjective maps $C \to Y$. Note that all the modules $M(1), M(2), M(3)$
as well as $C$ have the same dimension vector, namely  
$\smallmatrix  &1\cr
               2 &1 \cr
               &1  \endsmallmatrix$.

The zero element of the lattice ${}^C[\to Y\rangle$ is the projective cover $P(Y) \to Y$.
In the submodule lattice $\mathcal  S\Hom(C,Y)$, the zero element is $\Hom(C,\mathcal  P,Y) = 0.$
$$
{\beginpicture
\setcoordinatesystem units <2.2cm,1.1cm>
\put{\beginpicture

\put{$Y$} at 2 2
\put{$M(1)$} at 1.1 1 
\put{$M(2)$} at 1.55 1 
\put{$M(3)$} at 2 1 
\multiput{$C$} at 2.4 1  2.85 1 /
\put{$P(Y)$} at 2 0
\put{$\cdots$} at 2.65 1 
\arr{1.7 0.3}{1.3 0.7}
\arr{2.3 0.3}{2.7 0.7}
\arr{2.7 1.3}{2.3 1.7}
\arr{1.3 1.3}{1.7 1.7}
\arr{1.85 0.3}{1.65 0.7}  
\arr{1.65 1.3}{1.85 1.7}
\arr{2 0.3}{2 0.7}
\arr{2 1.3}{2 1.7}
\arr{2.15 0.3}{2.35 0.7}
\arr{2.35 1.3}{2.15 1.7}
\endpicture} at 0 0
\put{\beginpicture
\multiput{$\bullet$} at  1 1  2 0  3 1  2 2   1.5 1  2 1  2.4 1   /
\put{$\Hom(C,Y)$} [r] at 1.7 2 
\put{$\Hom(C,\mathcal  P,Y) = 0$} [r] at 1.7 0
\put{$\cdots$} at 2.65 1 
\plot 1 1  2 0  3 1  2 2  /
\plot 1 1  2 2 /
\plot 2 0  1.5 1  2 2 /
\plot 2 0  2 1  2 2 /
\plot 2 0  2.25 .5   /
\plot 2.25 1.5  2 2 /
\put{$\bigcirc$} at 2 0 
\endpicture} at 2.7 0

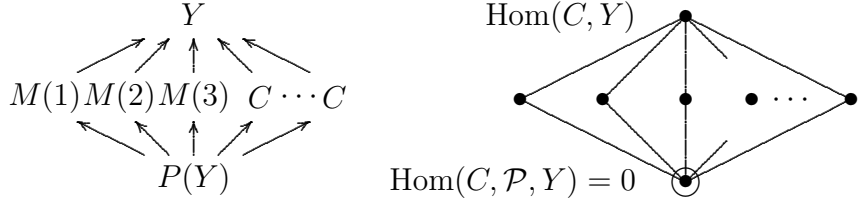
\endpicture}
$$
	\medskip 

\noindent 
{\bf Examples 18, where the modules in ${}^C[\to Y\rangle^1$ 
form a family of modules with varying dimension vectors.}
Again, we take the $3$-subspace quiver as considered in example 9.
Now let $C = P(a)$ and $Y = \tau Q(a)$ the maximal indecomposable module.

Since $C$ is the projective module $P(a)$, the right minimal right $C$-determined
maps $f: X \to Y$ are the inclusions maps of submodules $X$ of $Y$ such that the
socle of $Y/X$ is  a direct sum of copies of $S(a)$. The indecomposable submodules
of $Y$ are $P(b_1), P(b_2), P(b_3)$ as well as submodules isomorphic to $S(a)$.
In order that the socle of $Y/X$ is a direct sum of copies of $S(a)$, either
$X = Y$ (and $Y/X = 0$), or $X = 0$ (and $Y/X = Y$), or $X$ has to be indecomposable.
If $X$ is an indecomposable proper submodule of $Y$, then the socle of $Y/X$
is either isomorphic to $S(0)$ or else $X$ is contained in one of the modules 
$P(b_1), P(b_2), P(b_3)$.
It follows that the $1$-parameter family in the middle of ${}^C[\to Y\rangle$
consists of the indecomposable proper submodules $X$ of $Y$ such that the socle of
$Y/X$ is isomorphic to $S(a)$, thus either $X$ is one of $P(b_1), P(b_2), P(b_3)$,
or else $X$ is a simple submodule of $Y$ not contained in $P(b_1), P(b_2),P(b_3).$
$$
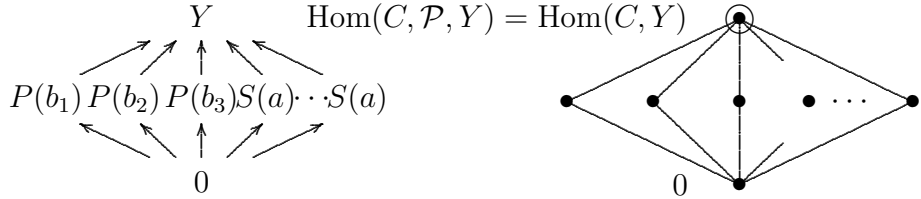
{\beginpicture
\setcoordinatesystem units <2.3cm,1.1cm>
\put{\beginpicture

\put{$Y$} at 2 2
\put{$P(b_1)$} at 1.1 1 
\put{$P(b_2)$} at 1.55 1 
\put{$P(b_3)$} at 2 1 
\multiput{$S(a)$} at 2.37 1  2.9 1 /
\put{$0$} at 2 0
\put{$\cdots$} at 2.65 1 
\arr{1.7 0.3}{1.3 0.7}
\arr{2.3 0.3}{2.7 0.7}
\arr{2.7 1.3}{2.3 1.7}
\arr{1.3 1.3}{1.7 1.7}
\arr{1.85 0.3}{1.65 0.7}  
\arr{1.65 1.3}{1.85 1.7}
\arr{2 0.3}{2 0.7}
\arr{2 1.3}{2 1.7}
\arr{2.15 0.3}{2.35 0.7}
\arr{2.35 1.3}{2.15 1.7}
\endpicture} at 0 0
\put{\beginpicture
\multiput{$\bullet$} at  1 1  2 0  3 1  2 2   1.5 1  2 1  2.4 1 /
\put{$\Hom(C,\mathcal  P,Y) =\Hom(C,Y)$} [r] at 1.7 2 
\put{$ 0$} [r] at 1.7 0
\put{$\cdots$} at 2.65 1 
\plot 1 1  2 0  3 1  2 2  /
\plot 1 1  2 2 /
\plot 2 0  1.5 1  2 2 /
\plot 2 0  2 1  2 2 /
\plot 2 0  2.25 .5   /
\plot 2.25 1.5  2 2 /
\put{$\bigcirc$} at 2 2 
\endpicture} at 2.4 0

\endpicture}
$$
	\medskip 
	
\noindent 
{\bf Example 19. Another example where the modules in ${}^C[\to Y\rangle^1$ 
form a family of modules with varying dimension vectors.}
Take the one-point extension $\Lambda$ of the Kronecker algebra 
using a regular module $R$ of length 2, say $R = R_\infty$
$$
{\beginpicture
\setcoordinatesystem units <2cm,1cm>
\put{$a$} at 0 0 
\put{$b$} at 1 0 
\put{$c$} at 2 0 
\arr{0.7 0.1}{0.3 0.1}
\arr{0.7 -.1}{0.3 -.1}
\arr{1.7 0}{1.3 0}
\setdots <1mm>
\setquadratic
\plot 0.5 -.2  1 -.5  1.5 -.2 /
\put{} at 0 -.3
\endpicture}
$$
The vertex $c$ is the extension vertex and $\rad P(c) = R_\infty$.
The regular Kronecker modules of length 2 
different from $R_\infty$ will be denoted by
$R_\lambda$ with $\lambda\in k.$

Let $K = S(a)$ and $C = \tau^{-}K$, this is the indecomposable Kronecker-module $P_2$,
its dimension vector (as a $\Lambda$-module) is $(3,2,0)$. 
Let $Y = P(c)/S(a)$,  its dimension vector
is $(0,1,1).$ We have $\dim\Hom(C,Y) = 2$. Since $\End(C) = k$, the
submodule lattice $\mathcal  S\Hom(C,Y)$ is just the projective line.
Under $\eta_{SY}$ we obtain a $\mathbb P_1$-family of right minimal maps ending in
$Y$, namely those with the following short exact sequences:
$$
{\beginpicture
\setcoordinatesystem units <1.3cm,.7cm>
\put{$K \longrightarrow P(c) \overset p \longrightarrow Y$} [l] at 0 1
\put{$K \longrightarrow R_\lambda \overset{f'_\lambda} \longrightarrow S(b)$} [l] at 0 0
\put{for $\lambda\in k$} [l] at 3 -.16
\endpicture}
$$
Here is, on the left,
 ${}^C[\to Y\rangle$, and, on the right, $\mathcal  S\Hom(C,Y)$ 
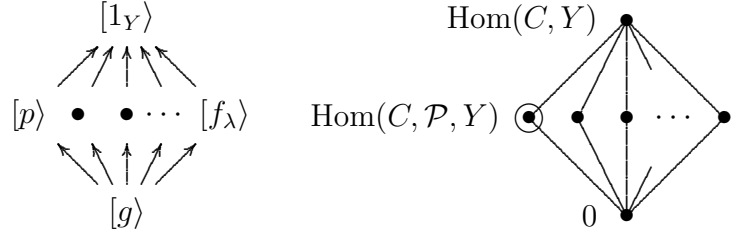
$$
{\beginpicture
\setcoordinatesystem units <1.3cm,1.3cm>
\put{\beginpicture

\multiput{$\bullet$} at    1.5 1   2 1   /

\put{$[1_Y\rangle$} at 2 2
\put{$[p\rangle$} at 1 1 
\put{$[f_\lambda\rangle$} at 3 1 
\put{$[g\rangle$} at 2 0
\put{$\cdots$} at 2.4 1 
\arr{1.7 0.3}{1.3 0.7}
\arr{2.3 0.3}{2.7 0.7}
\arr{2.7 1.3}{2.3 1.7}
\arr{1.3 1.3}{1.7 1.7}
\arr{1.85 0.3}{1.65 0.7}  
\arr{1.65 1.3}{1.85 1.7}
\arr{2 0.3}{2 0.7}
\arr{2 1.3}{2 1.7}
\arr{2.15 0.3}{2.35 0.7}
\arr{2.35 1.3}{2.15 1.7}
\endpicture} at 0 0
\put{\beginpicture
\multiput{$\bullet$} at  1 1  2 0  3 1  2 2   1.5 1  2 1   /
\put{$\Hom(C,Y)$} [r] at 1.7 2 
\put{$\Hom(C,\mathcal  P,Y)$} [r] at 0.7 1
\put{$0$} [r] at 1.7 0
\put{$\cdots$} at 2.5 1 
\plot 1 1  2 0  3 1  2 2  /
\plot 1 1  2 2 /
\plot 2 0  1.5 1  2 2 /
\plot 2 0  2 1  2 2 /
\plot 2 0  2.25 .5   /
\plot 2.25 1.5  2 2 /
\put{$\bigcirc$} at 1 1 
\endpicture} at 4 0

\endpicture}
$$
The $0$-subspace of $\Hom(C,Y)$ corresponds to a map
$g: P(b) \to Y$ with image $S(b)$, namely to the 
projective presentation of $S(b)$
$$
 K^2 \longrightarrow  P(b) \overset{g'}\longrightarrow S(b). \qquad\qquad
$$

It should be noted that the short exact sequence
$$
  K \longrightarrow  R_\infty \overset{f'_\infty}\longrightarrow S(b) \qquad\qquad
$$
yields the map $R_\infty \overset{f'\infty}\longrightarrow S(b)\ \subset\ Y$ 
which we denote by $f_\infty$ and which is not $C$-determined.
Namely, there is the following commutative diagram
$$
{\beginpicture
\setcoordinatesystem units <3cm,1.2cm>
\arr{0.3 1}{0.75 1}
\arr{0.25 0}{0.8 0}
\arr{0 .7}{0 .3}
\arr{1 .7}{1 .3}

\put{$\rad P(c)$} at 0 1
\put{$P(c)$} at 1 1
\put{$R_\infty$} at 0 0
\put{$Y$} at 1 0

\put{$\ssize \iota$} at 0.5 1.2
\put{$\ssize f_\infty$} at 0.5 .2
\put{$$} at -.15 .5
\put{$\ssize p$} at 1.1 .5
\endpicture}
$$
with inclusion map $\iota$. Since $p$ does not factor through $f_\infty$, 
we see that $p$ almost factors
through $f_\infty$, thus the theory asserts that $P(c)$ has to belong to any
determiner of $f_\infty$  and therefore $f_\infty$ cannot belong to 
${}^C[\to Y\rangle.$
	\bigskip

As we have noted, $P(c)$ has to belong to any
determiner of $f_\infty$.
Let us add $P(c)$ to $C$ and consider the Auslander bijection for
$C\oplus P(c)$ and $Y$. We have $\dim\Hom(C\oplus P(c),Y) = 3$. Note that the
endomorphism ring of $C\oplus P(c)$ is hereditary of type $\mathbb A_2$
and the submodule structure of $\Hom(C\oplus P(c),Y)$ looks as follows:
$$
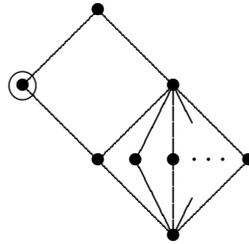
{\beginpicture
\setcoordinatesystem units <1cm,1cm>
\multiput{$\bullet$} at 0 2  1 1  2 0  3 1  2 2  1 3  1.5 1  2 1   /
\put{$\cdots$} at 2.5 1 
\plot 0 2  2 0  3 1  1 3  0 2 /
\plot 1 1  2 2 /
\plot 2 0  1.5 1  2 2 /
\plot 2 0  2 1  2 2 /
\plot 2 0  2.25 .5   /
\plot 2.25 1.5  2 2 /
\put{$\bigcirc$} at 0 2 
\endpicture}
$$
For the proof, we only have to verify that a non-trivial map $C \to P(c)$
does not annihilate the module $\Hom(C\oplus P(c),Y)$.
The encircled vertex is the submodule $\Hom(C\oplus P(c),\mathcal  P,Y).$

The corresponding diagram in ${}^{C\oplus P(c)}[\to Y\rangle$ looks as follows;
here, we write its elements as short exact sequences ending in a submodule
of $Y$:

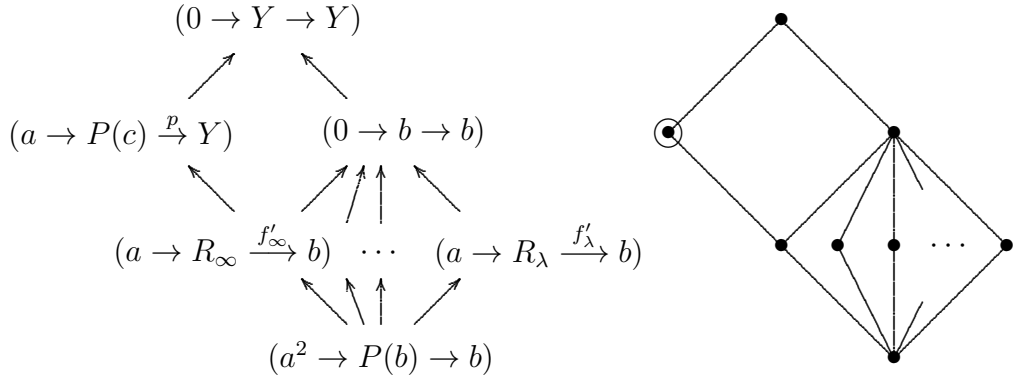
$$
{\beginpicture
\setcoordinatesystem units <1.5cm,1.5cm>
\put{\beginpicture
\put{$\cdots$} at 2 .93
\arr{1.7 0.3}{1.3 0.7}
\arr{0.7 1.3}{0.3 1.7}
\arr{0.3 2.3}{0.7 2.7}
\arr{1.3 1.3}{1.7 1.7}
\arr{2.3 0.3}{2.7 0.7}
\arr{2.7 1.3}{2.3 1.7}
\arr{1.7 2.3}{1.3 2.7}
\arr{2 0.3}{2 0.7}
\arr{2 1.2}{2 1.7}
\arr{1.85 0.3}{1.7 0.7}
\arr{1.7 1.2}{1.85 1.7}

\put{$(a^2\to P(b)\to b)$} at 2 0
\put{$(a\to  R_\infty \overset{f'_\infty}\longrightarrow  b)$} at .6 1
\put{$(a\to   R_\lambda \overset{f'_\lambda}\longrightarrow  b)$} at 3.4 1
\put{$(a\to   P(c)\overset{p}\to  Y)$} at -.3 2
\put{$(0\to Y\to Y)$} at 1 3
\put{$(0\to b\to b)$} at 2.2 2
\endpicture} at 0 0 

\put{\beginpicture
\multiput{$\bullet$} at 0 2  1 1  2 0  3 1  2 2  1 3  1.5 1  2 1   /
\put{$\cdots$} at 2.5 1 
\plot 0 2  2 0  3 1  1 3  0 2 /
\plot 1 1  2 2 /
\plot 2 0  1.5 1  2 2 /
\plot 2 0  2 1  2 2 /
\plot 2 0  2.25 .5   /
\plot 2.25 1.5  2 2 /
\put{$\bigcirc$} at 0 2 
\endpicture} at 4.5 0
\endpicture}
$$
We may label the lines of the Hasse diagram of $\mathcal  S\Hom(C\oplus P(c),Y)$
by the corresponding type, thus either by
$C$ or by $P(c)$ 
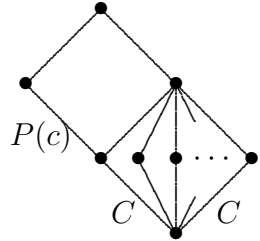
$$
{\beginpicture
\setcoordinatesystem units <1cm,1cm>
\multiput{$\bullet$} at 0 2  1 1  2 0  3 1  2 2  1 3  1.5 1  2 1   /
\put{$\cdots$} at 2.5 1 
\plot 0 2  2 0  3 1  1 3  0 2 /
\plot 1 1  2 2 /
\plot 2 0  1.5 1  2 2 /
\plot 2 0  2 1  2 2 /
\plot 2 0  2.25 .5   /
\plot 2.25 1.5  2 2 /
\multiput{$C$} at  1.3 0.3  2.7  0.3 / 
\multiput{$P(c)$} at 0.2 1.3 /
\endpicture}
$$

We should add that conversely, 
we can recover $\mathcal  S\Hom(C,Y)$ from $\mathcal  S\Hom(C\oplus P(c),Y)$
by deleting the shaded part:
$$
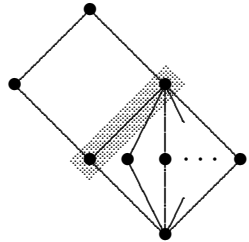
{\beginpicture
\setcoordinatesystem units <1cm,1cm>
\multiput{$\bullet$} at 0 2  1 1  2 0  3 1  2 2  1 3  1.5 1  2 1   /
\put{$\cdots$} at 2.5 1 
\plot 0 2  2 0  3 1  1 3  0 2 /
\plot 1 1  2 2 /
\plot 2 0  1.5 1  2 2 /
\plot 2 0  2 1  2 2 /
\plot 2 0  2.25 .5   /
\plot 2.25 1.5  2 2 /

\setshadegrid span <.3mm>
\vshade 0.7  1  1  <,z,,>  1 0.7  1.3  <z,z,,> 2 1.7 2.3  <z,,,> 2.3 2 2 / 
\endpicture}
$$
	\bigskip\bigskip

\centerline{\huge \bf III. Special cases.}

\section{The module $C$ being a generator.}
\label{sec:16}

The special case $C = \Lambda$ has been discussed already at the end of section 4.
In this case, $\mathcal  S \Hom(C,Y)$ is just the lattice of all submodules of $Y$.
	
\begin{Prop}  
Let $C$ be a generator. Then $f: X \to Y$ is right $C$-determined
if and only if the intrinsic kernel of $f$ is in $\add\tau C$.
\end{Prop}
	
\begin{proof} According to \ref{determiner}, the map $f$ is right $C$-determined if and only if 
$C(f) \in \add C$. Thus, assume that $C(f)\in \add C$. Let $K$ be an indecomposable direct
summand of the intrinsic kernel of $f$. By definition of $C(f)$, we know that
$\tau^{-}K$ belongs to $\add C(f)$, and $K$ is not injective.
Thus $K = \tau\tau^{-}K$ belongs to $\add \tau C(f) \subseteq \add C.$ This shows that
the intrinsic kernel of $f$ belongs to $\add \tau C.$

Conversely, assume that the intrinsic kernel belongs to $\add \tau C,$ 
in particular any indecomposable direct summand of the intrinsic kernel $K$ belongs
to $\add\tau C$, thus $\tau^{-}K$ belongs to $\add \tau^{-}\tau C \subseteq \add C$.
Also, since $C$ is a generator, any indecomposable projective module belongs to
$\add C$. As a consequence, $C(f)$ belongs to $\add C.$ This completes the proof. 
\end{proof}

Everyone admits that the concept of being determined is not very intuitive,
however in the special case when $C$ is a generator (and this is often the only
important case), one knows: The maps in $^C[\to Y\rangle$ can be described by the exact 
sequences
$$
 0 \to K \to X \overset e\to  Y' \to 0
$$
where $K$ is in $\add \tau C$ and $Y'$ is a submodule of $Y$ say with
inclusion map $m: Y' \to Y$, the map in $^C[\to Y\rangle$ in question 
is the composition $me$. To repeat: {\it If $C$ is a generator, a right minimal
is right
$C$-determined if and only if its kernel belongs to $\add \tau C$.}

This means: for $C$ a generator, the set $^C[\to Y\rangle$ 
can be visualized very well.
Unfortunately, this seems to be difficult in general, but the notion of 
right determination just allows to have
the prolific bijection $  {}^C[\to Y\rangle \leftrightarrow \mathcal  S\Hom(C,Y)$.
	
\begin{Prop}  Let $C$ be a generator and $u: Y'\to Y$ a
monomorphism. Then $f\mapsto uf$ defines an embedding $u_*: {}^C[\to Y'\rangle \to 
{}^C[\to Y\rangle$. The image of this embedding is an ideal of the lattice 
${}^C[\to Y\rangle$ and the following diagram commutes:
\Rahmen{\beginpicture
\setcoordinatesystem units <3.7cm,1.2cm>
\put{$^C[\to Y\rangle$} at 0 1 
\put{$\mathcal  S\Hom(C,Y)$} at 1 1 
\put{$^C[\to Y'\rangle$} at 0 0 
\put{$\mathcal  S\Hom(C,Y')$} at 1 0
\arr{0.3 0}{0.6 0}
\arr{0.3 1}{0.6 1}
\arr{0 0.3}{0 0.65}
\arr{1 0.3}{1 0.65}
\put{$\ssize \eta_{CY}$} at 0.45 1.15
\put{$\ssize \eta_{CY'}$} at 0.45 0.15
\put{$\ssize \mathcal  S\Hom(C,u)$} [l] at 1.07 0.47
\put{$u_*$} at -.08 0.47
\endpicture}
\end{Prop}

\begin{proof} Let $f: X\to Y'$ be right $C$-determined. Since the intrinsic kernel of
$f$ and of $uf$ are the same, also $uf$ is right $C$-determined by 14.1.
If $f\preceq f'$ are maps ending in $Y'$, then there is a map $h$ such that 
$f = f'h$, therefore $uf = uf'h$, thus $uf \preceq uf'$. This shows that $f\mapsto uf$
yields a map $u_*: {}^C[\to Y'\rangle \to {}^C[\to Y\rangle$. In order to see that
$u_*$ is an embedding, let us assume that $f,f'$ are maps ending in $Y$ such that
$uf \preceq uf'$. Then there is $h$ with $uf = uf'h$, but this implies that
$f = f'h$ (since $u$ is a monomorphism), and therefore $f \preceq f'$.
In order to see that the image of $u_*$ is an ideal, let $f$ be right $C$-determined
and ending in $Y'$ and $g$ right $C$-determined ending in $Y$ such that $g \preceq 
u_*(f).$ We want to show that $g$ is in the image of $u_*$. But $g \preceq uf$
means that there is $h$ with $g = ufh$, thus $g = u_*(fh).$ 

It remains to show that the diagram commutes. Let $f: X\to Y'$ be right 
$C$-determined. Then we have $\eta_{CY'}(f) = f\Hom(C,X)$. The map $\Hom(C,u)$
send any $\phi\in \Hom(C,Y')$ to $\Hom(C,u)(\phi) = u\phi$, thus $\mathcal  S\Hom(C,u)$
sends $f\Hom(C,X)$ to $uf\Hom(C,X).$ On the other hand, we also have $\eta_{CY}(u_*(f)) =
\eta_{CY}(uf) = uf\Hom(C,X).$
\end{proof}

If  $u: Y'\to Y$ is a
monomorphism, but $C$ is not a generator, then given $[f\rangle \in {}^C[\to Y'\rangle$,
the right equivalence class of the composition $uf$ usually will not belong 
to ${}^C[\to Y\rangle.$ 
	\medskip 

\noindent 
{\bf Example 20.} Take the quiver of type $\mathbb A_2$ as considered in example 1
and take $C = Y = S(2)$ and $Y' = 0$ with inclusion map $u: 0 \to S(2)$. 
The zero map $0\to 0$ yields an element of 
${}^C[\to Y'\rangle$, but $u$ is not right $C$-determined, since $P(2)$ almost factors through
$u$, so that $P(2)$ has to be a direct summand of $C(u)$ and therefore of any right
determiner of $u$.  

\section{The case of $[\to Y\rangle$ being of finite height.}
\label{sec:17}

Here we consider modules $Y$ such that there are only finitely many indecomposable 
modules $M_i$ with
$\Hom(M_i,Y) \neq 0$. Let $M$ be the direct sum of these modules and consider
the Auslander bijection
$$
 {}^M[\to Y\rangle \quad \leftrightarrow \quad \mathcal  S\Hom(M,Y),
$$
it maps $\bigoplus_{i,j} M_i^{n_i} \overset{(f_{ij})}\longrightarrow  Y$ to the $\Gamma(M)$-submodule
of $\Hom(M,Y)$ generated by the maps $f_{ij}$ (considered as maps
$M \to M_i \to Y$ where the map $M \to M_i$ is the canonical projection onto the
direct summand).

\begin{Prop}\label{finite-height}
Let $Y$ be a module.

{\rm(a)} If $[\to Y\rangle$ is of height $h$, then the number of 
isomorphism classes of indecomposable modules $M_i$
with $\Hom(M_i,Y) \neq 0$ is at most $h$.

{\rm (b)} If $M_1,\dots, M_{t}$ are all the indecomposable modules $M_i$
with $\Hom(M_i,Y) \neq 0$, one from each isomorphism class, and 
$M = \bigoplus_i^{t} M_i$, then $[\to Y\rangle = {}^M[\to Y\rangle$
and any module $C$ with  $[\to Y\rangle = {}^C[\to Y\rangle$ satisfies 
$M \in \add C.$
\end{Prop} 

\begin{proof} 
Let $M_1,\dots, M_t$ be pairwise non-isomorphic 
indecomposable modules with $\Hom(M_i,Y) \neq 0$, and let $M = \bigoplus_i^t M_i.$
Let $\Gamma(M) = \End(M)^{\text{op}}$.  Then the indecomposable projective 
$\Gamma(M)$-modules $P(i) = \Hom(M,M_i)$ are pairwise non-isomorphic, and
$\Hom_{\Gamma(M)}(P(i),\Hom(M,Y)) \neq 0$, thus the $\Gamma(M)$-module 
$\Hom(M,Y)$ has length 
at least $t$. Now, according to the Auslander bijection, the poset
$\mathcal  S M$ is isomorphic to the subposet ${}^M[\to Y\rangle$ of $[\to Y\rangle$. 
The length of $\mathcal  SM$
is at least $t$, the length of $[\to Y\rangle$ is $h$. This shows that $t \le h.$
This shows (a). 

In order to show (b), we recall from \ref{non-zero-map} that given a map $f$ ending in $Y$,
any indecomposable direct summand $C_0$ of $C(f)$  satisfies $\Hom(C_0,Y) \neq 0$,
thus $C_0$ is in $\add M$ and therefore $C(f)\in \add M$.
Thus any map $f: X \to Y$ is right $M$-determined and therefore 
$[\to Y\rangle = {}^M[\to Y\rangle$. 

On the other hand, let us assume that
$[\to Y\rangle = {}^C[\to Y\rangle$ for some module $C$.
Using again 3.9, we can assume that
any indecomposable direct summand $C_i$ of $C$ satisfies $\Hom(C_i,Y) \neq 0.$
Thus we can assume that $C$ is a direct summand of $M$. 
But if $C$ is a proper direct summand of $M$, say $C = \bigoplus_{i=1}^{t'} M_i$
with $t' < t$, then 
$$
{\beginpicture
\setcoordinatesystem units <1.5cm,.7cm>
\put{$|[\to Y\rangle| =  |{}^C[\to Y\rangle| = |\Hom(C,Y)|$} [l] at 0 2
\put{$= \sum_{i=1}^{t'} |\Hom(M_i,M)| 
 < \sum_{i=1}^{t} |\Hom(M_i,M)|$} [l] at 1 1
\put{$= |\Hom(M,Y)| = |{}^M[\to Y\rangle| \le 
 |[\to Y\rangle|$} [l] at 1 0 
\endpicture}
$$
a contradiction. This shows that $C = M$. 
\end{proof} 

\noindent 
{\bf Remark.} As we see, the indecomposable modules $M_i$ which occur as
direct summands of a minimal module $C$ with $[\to Y\rangle = {}^C[\to Y\rangle$
are modules with $\Hom(M_i,Y)\neq 0.$ But this is not surprising, since
the minimal right determiner $C(f)$ of any morphism $f$ has as indecomposable direct
summands only modules $C_i$ with $\Hom(C_i,Y) \neq 0$, see 3.9. Also
the converse should be stressed, namely the following part of the
assertion (b): any 
indecomposable module $M_i$ with $\Hom(M_i,Y) \neq 0$ is needed as a direct
summand of $C$. 
	
\begin{Cor} 
 Let $Y$ be a module.
The following conditions are equivalent:
\item{\rm (i)} There are only finitely many isomorphism classes of
 indecomposable modules $X$ with  $\Hom(X,Y)\neq 0.$
\item{\rm (ii)} The lattice $[\to Y\rangle$ has finite height.
\item{\rm (iii)} There is a module $C$ with $[\to Y\rangle = {}^C[\to Y\rangle$.
\end{Cor} 
	
\begin{proof} 
 The implication (ii) $\implies$ (i) has been shown in \ref{finite-height}(a), the
implication (i) $\implies$ (iii) in 15.1(b). Any lattice of the form
${}^C[\to Y\rangle$ is of finite height, thus obviously (iii) implies (ii).
\end{proof} 
	
If $S$ is simple and $[\to Q(S)\rangle$ is of finite height, then 
we deal just with the hammock corresponding to $S$, as considered in \cite{[RV]}.
	\medskip

\noindent 
{\bf Example 21.} Let $\Lambda$ be the 3-subspace quiver as considered in
example 9 and let $Y = Q(a).$  As before, we denote the indecomposable modules
which are neither projective nor injective by $N(i) = \tau Q(b_i)$ with $1\le i \le 3$
and $M = \tau Q(a)$. 
Here is the lattice $[\to Y\rangle$
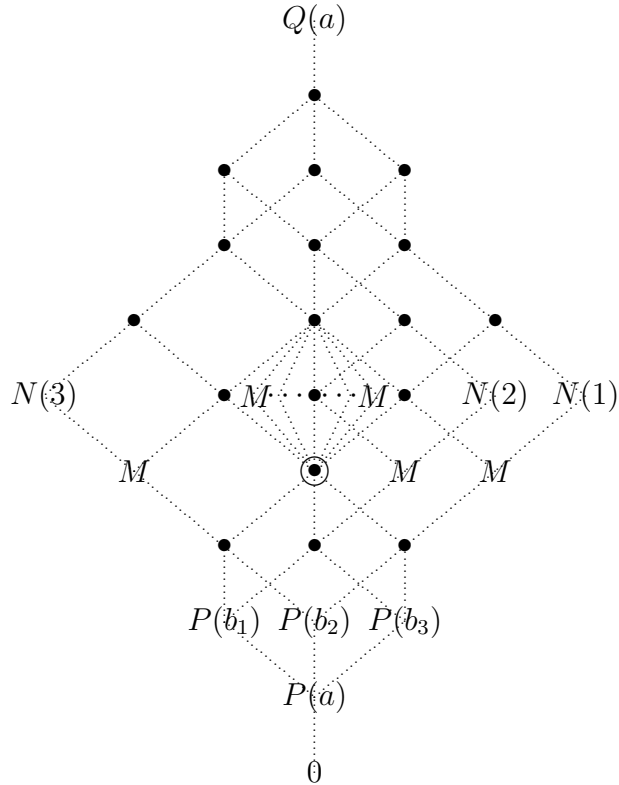
$$
{\beginpicture
\setcoordinatesystem units <1.2cm,1cm>
\put{$0$} at 0 0
\put{$P(a)$} at 0 1
\put{$P(b_1)$} at -1 2
\put{$P(b_2)$} at 0 2
\put{$P(b_3)$} at 1 2
\multiput{$M$} at -2 4  1 4  2 4 /
\put{$N(3)$} at -3 5
\put{$N(2)$} at 2 5
\put{$N(1)$} at 3 5
\put{$Q(a)$} at 0 10
\setdots <.8mm>
\plot 0 0  0 1  -1 2  -1 3  -3  5  -1 7  -1 8  0 9  0 10 /
\plot  0 1  1 2  1 3  3  5  1 7  1 8  0 9  /
\plot -1 2  0 3  1 2 /
\plot -1 3  0 2  1 3 /

\plot -1 7  0 8  1 7 /
\plot -1 8  0 7  1 8 /
\plot -2 4  1 7 /
\plot 2 4  -1 7 /

\plot -1 3  2 6 /
\plot 1 3  -2 6 /

\plot 0 1  0 2 /
\plot 0 3  0 4 /

\plot 0 6  0 7 /
\plot 0 8  0 9 /

\plot 0 4  0 6 /

\plot 0 3  2 5  0 7 /

\plot 1 4  0 5  1 6 /

\multiput{$\bullet$} at -1 3  0 3  1 3  0 4  -1 5  0 5  1 5 
       -2 6  0 6  2 6   -1 7  0 7  1 7  -1 8  0 8  1 8  0 9   1 6 /
\multiput{$M$} at -.65 5  .65 5 /
\multiput{$\cdots$} at -.3 5  .3 5 /
\plot 0 4  -.7 5  0 6 /
\plot 0 4  .7 5  0 6 /
\plot 0 4  -.4 5  0 6 /
\plot 0 4  .4 5  0 6 /
\put{$\bigcirc$} at 0 4 
\endpicture}
$$
Let $A$ be the direct sum of all indecomposable $\Lambda$-modules,
one from each isomorphism class. Then
the $\Gamma(A)$-module $\Hom(A,Y)$ is of length $10$ with $9$ different 
composition factors, they correspond to the indecomposables 
$$
 P(a); P(b_1), P(b_2), P(b_3); M; N(1), N(2), N(3); Q(a),
$$
and the composition factor corresponding to $M$
occurs with multiplicity 2 in $\Hom(A,Y)$.

In the picture above, only join irreducible elements (as well as the zero element) 
have been labeled. Note that in the middle layer
of the non-distributive interval of length 2, all but three elements are
join irreducible, and have the label $M$; here we deal with the various
epimorphisms $M \to Q(0)$.
Altogether there are 18 non-zero elements which are not join-irreducible. 

Note that the picture is obtained from the free modular lattice in 3 generators
as presented by Dedekind \cite{[D]} in 1900 by inserting in the non-distributive interval
of length $2$ further diagonals (one may call it the free $k$-modular lattice
in 3 generators).
	\medskip

There is an obvious action of the symmetric group $S_3$ of order 6 on the
3-subspace quiver, and thus also on the lattice $[\to Y\rangle$. Six vertices of $[\to Y\rangle$ are
invariant under this action, five of them correspond to important maps 
 ending in $Y$:
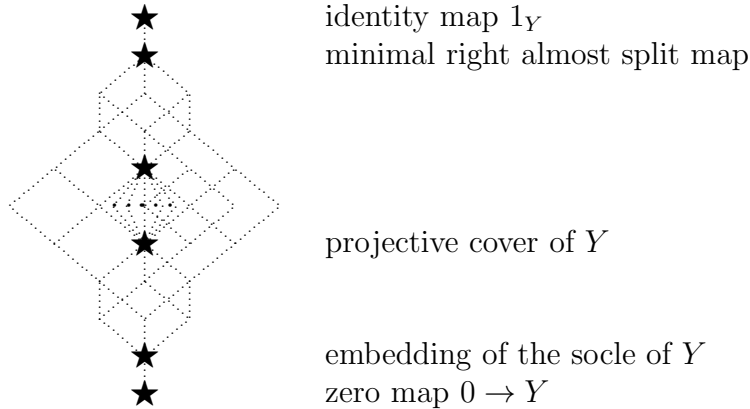
$$
{\beginpicture
\setcoordinatesystem units <.6cm,0.5cm>
\multiput{} at 0 0  0 10 /
\setdots <.8mm>
\plot 0 0  0 1  -1 2  -1 3  -3  5  -1 7  -1 8  0 9  0 10 /
\plot  0 1  1 2  1 3  3  5  1 7  1 8  0 9  /
\plot -1 2  0 3  1 2 /
\plot -1 3  0 2  1 3 /

\plot -1 7  0 8  1 7 /
\plot -1 8  0 7  1 8 /
\plot -2 4  1 7 /
\plot 2 4  -1 7 /

\plot -1 3  2 6 /
\plot 1 3  -2 6 /

\plot 0 1  0 2 /
\plot 0 3  0 4 /

\plot 0 6  0 7 /
\plot 0 8  0 9 /

\plot 0 4  0 6 /

\plot 0 3  2 5  0 7 /

\plot 1 4  0 5  1 6 /

\multiput{$\cdots$} at -.3 5  .3 5 /
\plot 0 4  -.7 5  0 6 /
\plot 0 4  .7 5  0 6 /
\plot 0 4  -.4 5  0 6 /
\plot 0 4  .4 5  0 6 /
\multiput{$\bigstar$} at 0 0  0 1  0 4  0 6  0 9  0 10 /
\put{identity map $1_Y$} [l] at 4 10
\put{minimal right almost split map} [l] at 4 9
\put{} [l] at 4 6
\put{projective cover of $Y$} [l] at 4 4
\put{embedding of the socle of $Y$} [l] at 4 1
\put{zero map $0 \to Y$} [l] at 4 0
\endpicture}
$$
the remaining one is the universal map from $\add M$ to $Y$. 

Instead of looking at representatives $f$ in right equivalence classes,
we may also draw the attention on the right equivalence classes $[f\rangle$ themselves.
For example, $[1_Y\rangle$ is the class
of all split epimorphisms ending in $Y$. If $g: X\to Y$ is minimal right almost split,
then $[g\rangle$ is the class of all right almost split maps ending in $Y$ and finally, 
$[0\to Y\rangle$ is the class 
all zero maps ending in $Y$.
	\bigskip

\noindent 
{\bf Remark.} An attempt to deal in a similar way with the $n$-subspace quivers 
for $n\ge 4$ was given by Gelfand and Ponomarev in a series of papers 
 (\cite{[GP2],[GP3],[GP4],[GP5]}), see also
\cite{[DR]} and \cite{[R3]}. In order to get information about the free modular lattice in
$n$ generators, they described part of the lattice $[\to Q(a)\rangle$ where 
$\Delta$ is the $n$-subspace quiver and $Q(a)$ is the indecomposable injective 
$k\Delta$-module with $a$ the sink of $\Delta$. Of course, for $n\ge 4$, this lattice 
$[\to Q(a)\rangle$ is of infinite height! It may be worthwhile to look for 
an interpretation of the results of Gelfand-Ponomarev in terms of the
Auslander bijections. 

\section{Some serial modules $\Hom(C,Y)$.}
\label{sec:18}

The Auslander bijections are defined for any pair of $\Lambda$-modules $C,Y$
and one of the posets involved is $\mathcal  S\Hom(C,Y)$. Assume that $J$ is an ideal of
$\Lambda$ which annihilates both modules $C,Y$ so that we may consider
$C$ and $Y$ as $\Lambda'$-modules, with $\Lambda' = \Lambda/J$. On the one hand, we
have $\Hom_\Lambda(C,Y) = \Hom_{\Lambda'}(C,Y)$. On the other hand, we have to distinguish
the set ${}^C[\to Y\rangle_{\Lambda'}$ of right equivalence
classes of right $C$-determined $\Lambda'$-modules ending in $Y$ from 
${}^C[\to Y\rangle_{\Lambda} = {}^C[\to Y\rangle$. Using the Auslander bijections
for $\Lambda$ as well as for $\Lambda'$, it is clear that the posets
${}^C[\to Y\rangle_{\Lambda'}$ and ${}^C[\to Y\rangle_{\Lambda}$ are isomorphic,
however the modules present in ${}^C[\to Y\rangle_{\Lambda'}$ usually will be
completely different from those present in ${}^C[\to Y\rangle_{\Lambda}$.
The following examples will show such deviations. 
	\medskip 

Let $\Lambda$ be a local uniserial ring.  When dealing with a local uniserial ring,
the indecomposable module of length $n$ will be denoted just be $n$. Also, if $C$
is a module, we will write $\rad^t(C,C)$ instead of
$(\rad\End(C))^t.$
	\medskip

\noindent 
{\bf Example 22. Let $C = Y = 4$.} First, let $\Lambda$ be of length at least 8,
so that $\Hom(4,\mathcal  P,4) = 0.$
$$
{\beginpicture
\setcoordinatesystem units <2.5cm,1.5cm>
\put{\beginpicture
\put{$0$} at 1 4
\multiput{$4$} at 1 3  1 2  1 1  1 0  2 4  3 4  3 3 3 2  2 4   3 1  3 0 /
\put{$0$} at 1 4
\put{$5\oplus 3$} at 2 3
\put{$6\oplus 2$} at 2 2
\put{$7\oplus 1$} at 2 1
\put{$8$} at 2 0
\arr{1.2 4}{1.7 4}
\arr{1.2 3}{1.7 3}
\arr{1.2 2}{1.7 2}
\arr{1.2 1}{1.7 1}
\arr{1.2 0}{1.7 0}

\plot 2.3 4.03  2.8 4.03 /
\plot 2.3 3.97  2.8 3.97 /
\arr{2.3 3}{2.8 3}
\arr{2.3 2}{2.8 2}
\arr{2.3 1}{2.8 1}
\arr{2.3 0}{2.8 0}

\arr{1 0.3}{1 0.7}
\arr{1 1.3}{1 1.7}
\arr{1 2.3}{1 2.7}
\arr{1 3.3}{1 3.7}

\arr{2 0.3}{2 0.7}
\arr{2 1.3}{2 1.7}
\arr{2 2.3}{2 2.7}
\arr{2 3.3}{2 3.7}

\plot 3.02 2.3  3.02 2.7 /
\plot 2.98 2.3  2.98 2.7 /

\plot 3.02 3.3  3.02 3.7 /
\plot 2.98 3.3  2.98 3.7 /

\plot 3.02 1.3  3.02 1.7 /
\plot 2.98 1.3  2.98 1.7 /

\plot 3.02 0.3  3.02 0.7 /
\plot 2.98 0.3  2.98 0.7 /

\multiput{$\left[\smallmatrix m \cr -e \endsmallmatrix \right]$} at 1.45 1.2  
    1.45 2.2  1.45 3.2 /
\put{$\left[\smallmatrix e \cr 0 \endsmallmatrix \right]$} at 2.15 0.5
\multiput{$\left[\smallmatrix e & m \endsmallmatrix \right]$} at 2.2 3.5 
    2.6 3.2  2.6 2.2  2.6 1.2    /
\put{$\left[\smallmatrix e \cr 
                                 & m \endsmallmatrix \right]$} at 2.2 1.5   
\put{$\left[\smallmatrix e \cr 
                                 & m \endsmallmatrix \right]$} at  2.2 2.5 
\put{$\ssize m$} at 1.45 0.15
\put{$\ssize e$} at 2.6 0.15
\multiput{$\ssize \phi$} at 1.08 0.45  1.08 1.45  1.08 2.45 /

\put{(AR)} [l] at 3.2 3

\endpicture} at 0 0
\put{\beginpicture
\multiput{$\bullet$} at 0 0  0 1  0 2  0 3  0 4  /
\plot 0 0  0 4 /
\put{$\Hom(C,C)$} [l]  at 0.1 4
\put{$\rad(C,C)$} [l]  at 0.1 3
\put{$\rad^2(C,C)$}  [l] at 0.1 2
\put{$\rad^3(C,C)$}  [l] at 0.1 1
\put{$0 = \Hom(C,\mathcal  P,C)$} [l]  at 0.1 0
\put{$\bigcirc$} at 0 0
\endpicture} at  2.5 0 

\endpicture}
$$
here, we have denoted by $m$ the canonical inclusion maps, by $e$ the canonical
projections and $\phi$ is a radical generator $\End(C)$. 
The Auslander-Reiten sequence is marked as (AR). 
	\bigskip

Second, let $\Lambda$ be of length 6, so that $\Hom(4,\mathcal  P,4) = \rad^2(C,C).$
$$
{\beginpicture
\setcoordinatesystem units <2.5cm,1.5cm>
\put{\beginpicture
\put{$0$} at 1 4
\multiput{$4$} at 1 3  1 2  1 1  1 0  2 4  3 4  3 3 3 2  2 4 /
\put{$0$} at 1 4
\put{$3$} at 3 1 
\put{$2$} at 3 0
\put{$5\oplus 3$} at 2 3
\put{$6\oplus 2$} at 2 2
\put{$6\oplus 1$} at 2 1
\put{$6$} at 2 0
\arr{1.2 4}{1.7 4}
\arr{1.2 3}{1.7 3}
\arr{1.2 2}{1.7 2}
\arr{1.2 1}{1.7 1}
\arr{1.2 0}{1.7 0}

\plot 2.3 4.03  2.8 4.03 /
\plot 2.3 3.97  2.8 3.97 /
\arr{2.3 3}{2.8 3}
\arr{2.3 2}{2.8 2}
\arr{2.3 1}{2.8 1}
\arr{2.3 0}{2.8 0}

\arr{1 0.3}{1 0.7}
\arr{1 1.3}{1 1.7}
\arr{1 2.3}{1 2.7}
\arr{1 3.3}{1 3.7}

\arr{2 0.3}{2 0.7}
\arr{2 1.3}{2 1.7}
\arr{2 2.3}{2 2.7}
\arr{2 3.3}{2 3.7}

\arr{3 0.3}{3 0.7}
\arr{3 1.3}{3 1.7}
\plot 3.02 2.3  3.02 2.7 /
\plot 2.98 2.3  2.98 2.7 /

\plot 3.02 3.3  3.02 3.7 /
\plot 2.98 3.3  2.98 3.7 /

\multiput{$\left[\smallmatrix m \cr -e \endsmallmatrix \right]$} at 1.45 1.2  1.45 2.2  1.45 3.2 /
\put{$\left[\smallmatrix 1 \cr 0 \endsmallmatrix \right]$} at 2.15 0.5
\multiput{$\left[\smallmatrix e & m \endsmallmatrix \right]$} at 2.2 3.5 
    2.6 3.2  2.6 2.2  2.6 1.2    /
\put{$\left[\smallmatrix \phi \cr 
                                 & m \endsmallmatrix \right]$} at 2.2 1.5   
\put{$\left[\smallmatrix e \cr 
                                 & m \endsmallmatrix \right]$} at  2.2 2.5 
\multiput{$\ssize m$} at 1.45 0.15  3.08 0.45  3.08 1.45  /
\put{$\ssize e$} at 2.6 0.15

\put{(AR)} [l] at 3.2 3
\multiput{$\ssize \phi$} at 1.08 0.45  1.08 1.45  1.08 2.45 /
\endpicture} at 0 0
\put{\beginpicture
\multiput{$\bullet$} at 0 0  0 1  0 2  0 3  0 4  /
\plot 0 0  0 4 /
\put{$\Hom(C,C)$} [l]  at 0.1 4
\put{$\rad(C,C)$} [l]  at 0.1 3
\put{$\rad^2(C,C) = \Hom(C,\mathcal  P,C)$}  [l] at 0.1 2
\put{$\rad^3(C,C)$}  [l] at 0.1 1
\put{$0$} [l]  at 0.1 0
\put{$\bigcirc$} at 0 2

\endpicture} at  2.5 0 

\endpicture}
$$
again, $m$ stands for a canonical inclusion map, $e$ for a canonical
projection map and $\phi$ for a radical generator of the endomorphism ring of
a uniserial module. 

Finally, let $\Lambda$ be of length 4, thus $C$ is projective and
therefore $\Hom(C,\mathcal  P,C) = \Hom(C,C).$ Since $C$ is projective,
all the right minimal, right $C$-determined maps are inclusion maps: 
$$
{\beginpicture
\setcoordinatesystem units <2.5cm,1.5cm>
\put{\beginpicture
\multiput{$0$} at 1 4  1 3  1 2  1 1  1 0  2 0  3 0 /
\multiput{$4$} at 2 4  3 4 /
\multiput{$3$} at 2 3  3 3 /
\multiput{$2$} at 2 2  3 2 /
\multiput{$1$} at 2 1  3 1 /
\arr{1.2 4}{1.7 4}
\arr{1.2 3}{1.7 3}
\arr{1.2 2}{1.7 2}
\arr{1.2 1}{1.7 1}
\arr{1.2 0}{1.7 0}

\plot 2.3 4.03  2.8 4.03 /
\plot 2.3 3.97  2.8 3.97 /

\plot 2.3 3.03  2.8 3.03 /
\plot 2.3 2.97  2.8 2.97 /

\plot 2.3 2.03  2.8 2.03 /
\plot 2.3 1.97  2.8 1.97 /

\plot 2.3 1.03  2.8 1.03 /
\plot 2.3 0.97  2.8 0.97 /

\plot 2.3 0.03  2.8 0.03 /
\plot 2.3 -.03  2.8 -.03 /


\arr{1 0.3}{1 0.7}
\arr{1 1.3}{1 1.7}
\arr{1 2.3}{1 2.7}
\arr{1 3.3}{1 3.7}

\arr{2 0.3}{2 0.7}
\arr{2 1.3}{2 1.7}
\arr{2 2.3}{2 2.7}
\arr{2 3.3}{2 3.7}

\arr{3 0.3}{3 0.7}
\arr{3 1.3}{3 1.7}
\arr{3 2.3}{3 2.7}
\arr{3 3.3}{3 3.7}

\put{} [l] at 3.4 3
\multiput{$\ssize m$} at  2.08 1.45   2.08 2.45   2.08 3.45  
                          3.08 1.45   3.08 2.45   3.08 3.45  /

\endpicture} at 0 0
\put{\beginpicture
\multiput{$\bullet$} at 0 0  0 1  0 2  0 3  0 4  /
\plot 0 0  0 4 /
\put{$\Hom(C,C) = \Hom(C,\mathcal  P,C)$} [l]  at 0.1 4
\put{$\rad(C,C)$} [l]  at 0.1 3
\put{$\rad^2(C,C)$}  [l] at 0.1 2
\put{$\rad^3(C,C)$}  [l] at 0.1 1
\put{$0$} [l]  at 0.1 0
\put{$\bigcirc$} at 0 4

\endpicture} at  2.5 0 

\endpicture}
$$
	\bigskip

\noindent 
{\bf Example 23. Let $C = 1\oplus 2$ and $Y = 3$.}
The ring $\Gamma(C) = \End(C)^{\text{op}}$ is the Nakayama algebra with Kupisch series $2,3$.
The $\Gamma(C)$-module $\Hom(C,Y)$ is the indecomposable 
projective module of length 3.
	\bigskip 
If we work over a uniserial ring $\Lambda$ of length at least 5,
so that $\Hom(C,\mathcal  P,Y) = 0$, then the situation is as follows:
$$
{\beginpicture
\setcoordinatesystem units <2.5cm,1.5cm>
\put{\beginpicture
\put{$0$} at 1 4
\multiput{$3$} at   2 4  3 4  3 3  3 2  3 1 /
\put{$1$} at 1 2
\multiput{$2$} at 1 1  1 3  / 
\multiput{$4$} at  2 2 /

\put{$4\oplus 1$} at 2 3
\put{$5\oplus 1$} at 2 1
\arr{1.2 4}{1.7 4}
\arr{1.2 3}{1.7 3}
\arr{1.2 2}{1.7 2}
\arr{1.2 1}{1.7 1}

\plot 2.3 4.03  2.8 4.03 /
\plot 2.3 3.97  2.8 3.97 /
\arr{2.3 3}{2.8 3}
\arr{2.3 2}{2.8 2}
\arr{2.3 1}{2.8 1}

\arr{1 1.3}{1 1.7}
\arr{1 2.3}{1 2.7}
\arr{1 3.3}{1 3.7}

\arr{2 1.3}{2 1.7}
\arr{2 2.3}{2 2.7}
\arr{2 3.3}{2 3.7}

\plot 3.02 2.3  3.02 2.7 /
\plot 2.98 2.3  2.98 2.7 /

\plot 3.02 1.3  3.02 1.7 /
\plot 2.98 1.3  2.98 1.7 /

\plot 3.02 3.3  3.02 3.7 /
\plot 2.98 3.3  2.98 3.7 /

\multiput{$\left[\smallmatrix m \cr -e \endsmallmatrix \right]$} at   1.45 3.2 /
\put{$\left[\smallmatrix 1 \cr 0 \endsmallmatrix \right]$} at 2.15 2.5
\multiput{$\left[\smallmatrix e & m \endsmallmatrix \right]$} at 2.2 3.5  2.6 3.2 
     2.2 1.5  2.6 1.2 /
\multiput{$\ssize m$} at 1.45 1.15  1.45 2.15  /
\multiput{$\ssize e$} at 2.6 2.15 /
\endpicture} at 0 0
\put{\beginpicture
\multiput{$\bullet$} at 0 1  0 2  0 3  0 4  /
\plot 0 1  0 4 /
\put{$\Hom(C,Y)$} [l] at 0.1 4
\put{$0 = \Hom(C,\mathcal  P,Y)$} [l] at 0.1 1
\put{} at -.2 1
\put{$\bigcirc$} at 0 1
\endpicture} at  2.5 0 

\put{} at 3.2 0 

\endpicture}
$$

Next, assume that $\Lambda$ is of length 4. 
$$
{\beginpicture
\setcoordinatesystem units <2.5cm,1.5cm>
\put{\beginpicture
\put{$0$} at 1 4
\multiput{$3$} at   2 4  3 4  3 3  3 2  /
\put{$1$} at 1 2
\multiput{$2$} at 1 1  1 3  3 1 / 
\multiput{$4$} at 2 1  2 2 /

\put{$4\oplus 1$} at 2 3
\arr{1.2 4}{1.7 4}
\arr{1.2 3}{1.7 3}
\arr{1.2 2}{1.7 2}
\arr{1.2 1}{1.7 1}

\plot 2.3 4.03  2.8 4.03 /
\plot 2.3 3.97  2.8 3.97 /
\arr{2.3 3}{2.8 3}
\arr{2.3 2}{2.8 2}
\arr{2.3 1}{2.8 1}

\arr{1 1.3}{1 1.7}
\arr{1 2.3}{1 2.7}
\arr{1 3.3}{1 3.7}

\arr{2 1.3}{2 1.7}
\arr{2 2.3}{2 2.7}
\arr{2 3.3}{2 3.7}

\arr{3 1.3}{3 1.7}
\plot 3.02 2.3  3.02 2.7 /
\plot 2.98 2.3  2.98 2.7 /

\plot 3.02 3.3  3.02 3.7 /
\plot 2.98 3.3  2.98 3.7 /

\multiput{$\left[\smallmatrix m \cr -e \endsmallmatrix \right]$} at   1.45 3.2 /
\put{$\left[\smallmatrix 1 \cr 0 \endsmallmatrix \right]$} at 2.15 2.5
\multiput{$\left[\smallmatrix e & m \endsmallmatrix \right]$} at 2.2 3.5  2.6 3.2  /
\multiput{$\ssize m$} at 1.45 1.15  1.45 2.15  3.1  1.45 /
\multiput{$\ssize e$} at 2.6 1.15  2.6 2.15 /
\put{$\ssize \phi$} at 2.1 1.45 
\endpicture} at 0 0
\put{\beginpicture
\multiput{$\bullet$} at 0 1  0 2  0 3  0 4  /
\plot 0 1  0 4 /
\put{$\Hom(C,Y)$} [l] at 0.1 4
\put{$0$} [l] at 0.1 1
\put{$\Hom(C,\mathcal  P,Y)$} [l] at .1 2
\put{$\bigcirc$} at 0 2
\endpicture} at  2.5 0 

\put{} at 3.2 0 
\endpicture}
$$

Finally, we consider the case where $\Lambda$ is of length 3.
$$
{\beginpicture
\setcoordinatesystem units <2.5cm,1.5cm>
\put{\beginpicture
\put{$0$} at 1 4
\multiput{$3$} at   2 4  3 4  2 2  2 1  /
\multiput{$1$} at 1 2  3 1 /
\multiput{$2$} at 1 1  1 3  3 2  3 3  / 
\multiput{$3$} at 2 1  2 2 /

\put{$3\oplus 1$} at 2 3
\arr{1.2 4}{1.7 4}
\arr{1.2 3}{1.7 3}
\arr{1.2 2}{1.7 2}
\arr{1.2 1}{1.7 1}

\plot 2.3 4.03  2.8 4.03 /
\plot 2.3 3.97  2.8 3.97 /
\arr{2.3 3}{2.8 3}
\arr{2.3 2}{2.8 2}
\arr{2.3 1}{2.8 1}

\arr{1 1.3}{1 1.7}
\arr{1 2.3}{1 2.7}
\arr{1 3.3}{1 3.7}

\arr{2 1.3}{2 1.7}
\arr{2 2.3}{2 2.7}
\arr{2 3.3}{2 3.7}

\arr{3 1.3}{3 1.7}
\plot 3.02 2.3  3.02 2.7 /
\plot 2.98 2.3  2.98 2.7 /

\arr{3 3.3}{3 3.7}

\multiput{$\left[\smallmatrix m \cr -e \endsmallmatrix \right]$} at   1.45 3.2 /
\put{$\left[\smallmatrix 1 \cr 0 \endsmallmatrix \right]$} at 2.15 2.5
\multiput{$\left[\smallmatrix \phi & m \endsmallmatrix \right]$} at 2.2 3.5  /
\multiput{$\left[\smallmatrix e & m \endsmallmatrix \right]$} at 2.6 3.2  /
\multiput{$\ssize m$} at 1.45 1.15  1.45 2.15   3.1  1.45  3.1 3.45 /
\multiput{$\ssize e$} at 2.6 1.15  2.6 2.15 /
\put{$\ssize \phi$} at 2.1 1.45 

\endpicture} at 0 0
\put{\beginpicture
\multiput{$\bullet$} at 0 1  0 2  0 3  0 4  /
\plot 0 1  0 4 /
\put{$\Hom(C,Y) = \Hom(C,\mathcal  P,Y)$} [l] at 0.1 4
\put{$0$} [l] at 0.1 1
\put{} at -.8 1
\put{$\bigcirc$} at 0 4
\endpicture} at  2.5 0 

\put{} at -1.7 0 
\endpicture}
$$

\section{Appendix: Modular lattices.}
\label{sec:19}

Let $L$ be a modular lattice of height $h$. The set of elements in $L$ of height $i$
will be denoted by $L_i$ or also $L^{h-i}.$
Here are some typical modular lattices:
	\medskip

The {\it chain} $\mathbb I(h)$ of height $h$, this is the poset of the integers
  $0,1\,\dots,h$ with the usual order relation. Note that 
$\mathbb I(h)_i$ consists of a single
  element, for $0 \le i \le h$.
Here are the lattices $\mathbb I(i)$ with $0\le i \le 4$.
$$
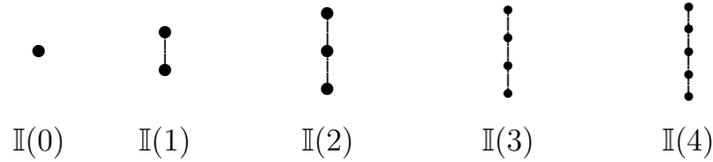
{\beginpicture
\setcoordinatesystem units <1.2cm,1cm>
\put{$\mathbb I(0)$} at 3.8 7.8
\put{$\mathbb I(1)$} at 5.2 7.8
\put{$\mathbb I(2)$} at 7 7.8
\put{$\mathbb I(3)$} at 9 7.8
\put{$\mathbb I(4)$} at 11 7.8

\put{$\bullet$} at 3.8 9 

\multiput{\beginpicture
\setcoordinatesystem units <.5cm,.5cm>
\multiput{$\bullet$} at  0 0  0 1 /
\setsolid
\plot 0 0  0 1  /
\endpicture
} at  5.2 9   /

\multiput{\beginpicture
\setcoordinatesystem units <.5cm,.5cm>
\multiput{$\bullet$} at  0 0  0 1  0 2  /
\setsolid
\plot 0 0  0 1  0 2  /
\endpicture
} at  7 9   /

\multiput
{\beginpicture
\setcoordinatesystem units <.37cm,.37cm>
\multiput{$\ssize \bullet$} at  0 0  0 1  0 2  0 3 /
\plot 0 0  0 3 /
\endpicture} at   9 9  /

\multiput
{\beginpicture
\setcoordinatesystem units <.3cm,.3cm>
\multiput{$\ssize \bullet$} at  0 0  0 1  0 2  0 3  0 4 /
\plot 0 0  0 4 /
\endpicture} at   11 9 /

\endpicture}
$$

Let $k$ be a field. A {\it  projective geometry} over $k$ is the 
lattice of all subspaces of a vector space over $k$. 
In accordance to the notation used throughout the paper, we 
denote the lattice of subspaces of the
$d$-dimensional vector space $k^d$ by $\mathcal S(k^{d})$ or also just by $\mathbb G(d)$
(this is called
the projective geometry over $k$ of {\it dimension} $d\!-\!1$, it is 
a lattice of height $d$). 
Here are the lattices $\mathbb G(d)$ for $0\le d \le 4$:
$$
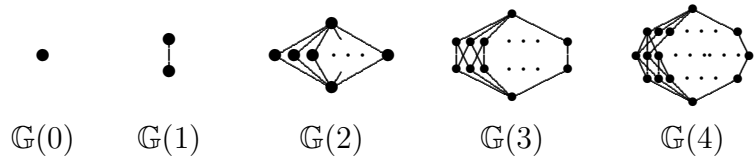
{\beginpicture
\setcoordinatesystem units <1.2cm,.7cm>
\put{$\mathbb G(0)$} at 3.8 7.4
\put{$\mathbb G(1)$} at 5.2 7.4
\put{$\mathbb G(2)$} at 7 7.4
\put{$\mathbb G(3)$} at 9 7.4
\put{$\mathbb G(4)$} at 11 7.4

\put{$\bullet$} at 3.8 9 

\multiput{\beginpicture
\setcoordinatesystem units <.5cm,.42cm>
\multiput{$\bullet$} at  0 0  0 1 /
\setsolid
\plot 0 0  0 1  /
\endpicture
} at  5.2 9   /

\multiput{\beginpicture
\setcoordinatesystem units <.5cm,.43cm>
\multiput{$\bullet$} at  1 1  2 0  3.5 1  2 2   1.5 1   0.5 1 /
\put{$\cdots$} at 2.5 1 
\setsolid
\plot 1 1  2 0  3.5 1  2 2  /
\plot 1 1  2 2 /
\plot 2 0  1.5 1  2 2 /
\plot 2 0  0.5 1  2 2 /
\plot 2 0  2.25 .5   /
\plot 2.25 1.5  2 2 /
\endpicture
} at  7 9   /

\multiput
{\beginpicture
\setcoordinatesystem units <.37cm,.37cm>
\multiput{$\ssize \bullet$} at  1 1   1 2  1.5 1   1.5 2  3 0  2 1  2 2  3 3 
    5 1  5 2    /
\setsolid
\plot 3 0  1 1  1 2  3 3   5 2  5 1  3 0 /
\plot 2 1  3 0  1.5 1 /
\plot 2 2  3 3  1.5 2 /
\plot 1 1  1.5 2  2 1  2 2  1.5 1  1 2 /
\multiput{$\cdots$} at 3.5 1  3.5 2   /
\endpicture} at   9 9  /

\multiput
{\beginpicture
\setcoordinatesystem units <.3cm,.31cm>
\multiput{$\ssize \bullet$} at 0.5 2  1 1  1 2  1 3  1.5 1  1.5 2  1.5 3  3 0  2 1  2 3  3 4     5 1  5 3  5.5 2  /
\setsolid
\plot 3 0  1 1  0.5 2  1 3  3 4   5 3  5.5 2  5 1  3 0 /
\plot 1 1  1 3 /
\plot 3 0  1.5 1  0.5 2  /
\plot 1.5 3   3 4 /
\plot 1.5 3  0.5 2 /
\plot  3 4  2 3  1 2  2 1  3 0 /
\plot  1 3  1.5 2  1.5 1  /
\plot 2 1  1.5 2 /
\multiput{$\cdots$} at 3.5 1  3.5 3  3 2  4.5 2 /
\endpicture} at   11 9 /

\multiput{} at 3.2 9 /

\endpicture}
$$
Of course, for $d\ge 2$, the number of elements of $\mathbb G(d)$ depends on the cardinality
of $k$. If $k$ is a field with $q$ elements, 
the cardinality of $\mathbb G(2)_1 = \mathbb G(2)^1$ is $q+1$ (here, $q$
is an arbitrary power of a prime number, thus if $\mathbb G(2)$ is finite, then
$|\mathbb G(2)_1| = 3,4,5,6,8,9,10,12,14,\dots).$

The subset $\mathbb G(d))_i$ is the
set of subspaces of dimension $i$ of a $k$-space of dimension $d$, it is often denoted by 
 $\mathbb G(i,d\!-\!i)$ and called a {\it Grassmannian;} both sets 
$\mathbb G(d)_1$ and $\mathbb G(d)^1$ are also denoted by $\mathbb P^{d-1}$ and called
the {\it projective space  of dimension $d\!-\!1$:}
$$
{\beginpicture
\setcoordinatesystem units <1cm,.7cm>
\put{$\mathbb G(d)_{d-1}$} [l] at 0 2
\put{$=\mathbb G(d)^1 \quad=\  \mathbb P^{d-1}$} [l] at 2 2
\put{$\mathbb G(d)_i$} [l] at 0 1 
\put{$=\mathbb G(d)^{d-i} =\ \mathbb G(i,d-i)$} [l] at 2 1
\put{$\mathbb G(d)_{1}$} [l] at 0 0
\put{$=\mathbb G(d)^{d-1} =\  \mathbb P^{d-1}.$} [l] at 2 0 
\endpicture}
$$

\section{Final remarks.}
\label{sec:20}
	
\noindent 
{\bf 20.1. Duality.} For all the results presented here, there is a dual
version which has the same importance. Note that we deal with
an artin algebra $\Lambda$ and consider finitely generated modules,
thus there is a duality functor. Namely, by assumption, $\Lambda$
is a module-finite $k$-algebra, where $k$ is a commutative artinian ring.
A minimal injective cogenerator $Q$ in $\mo k$ yields the duality
functor $\Hom_k(-,Q)$. 

We did not attempt to formulate the dual definitions and statements, 
but the reader is
advised to do so. To start with, we need the notion of left equivalence in order to
introduce $\langle Y \to ]$ as 
the set of left equivalence classes of maps starting in $Y$ (or of left minimal
maps). Then we need the notion of left $C$-determination in order to define
$\langle Y\rightarrow]^C$ as the set of the left equivalence classes
maps starting in $Y$ which are left $C$-determined.
	\medskip 

\noindent 
{\bf 20.2. Proofs of Auslander's two Main Theorems.}
Both results are presented in detail in the Philadelphia notes
\cite{[A1]}, see also \cite{[A2]}. There is a different treatment in the book \cite{[ARS]} of
Auslander-Reiten-Smal\o: Auslander's Second Theorem is presented in Theorem
XI.3.9., for the First Theorem, see Theorem XI.2.10
and Proposition XI.2.4 (this actually improves the assertion given in the
Philadelphia Notes). 
A concise proof of the First Theorem can also be found in \cite{[R9]}. 
As one of the main ingredients for the proof of the Second Theorem, one
may use Auslander's defect formula. For a direct approach to the defect
formula, we recommend the paper \cite{[K1]} by Krause. 
	\medskip

\noindent 
{\bf 20.3. The universal character of the Auslander bijections.}
The Auslander bijections are Auslander's approach to describe say the module
category for an artin algebra completely: not just to provide some invariants.
The importance of using invariants usually relies on the fact that they
will allow to distinguish certain objects, but they may not say much about
other ones --- the most effective invariants are often those which attach 
to objects just one of the values $1$ or $0$ (thus ``yes'' or ``no'').
Of course, the use of such an invariant 
is restricted to some specific problem. Now Auslander's approach is an attempt to
describe a module category completely, thus we may ask whether it does not have
to be tautological: if we don't forget some of the structure, we will not see
the remaining structure in more detail. Indeed, the Auslander bijections 
focus on parts of the category, namely the sets ${}^C[\to Y\rangle$,
but the decisive feature is the possibility to 
change the focus by enlarging $C$ (adding direct summands). The
universality of this approach is due to the fact that the category is covered
completely by such subsets.  
	\medskip 

\noindent 
{\bf 20.4. The irritation of the wording ``morphisms determined by modules''.}
Let us insert a short discussion of our hesitation to say that a morphism $f$ 
is right determined by a module $C$. One possible interpretation is to assert that 
even if the morphism $f$ itself is 
not determined by $C$, there is a certain factorization property of $f$ which
is determined by $C$. But there is another way out:

When Auslander asserts that 
{\it every morphism in $\mo\Lambda$ is right determined by a $\Lambda$-module}, 
one expects a classification of the (right minimal) morphisms in $\mo\Lambda$
using as invariants just $\Lambda$-modules. One even may strengthen Auslander's assertion
by saying {\it every morphism in $\mo\Lambda$ is right determined by the
isomorphism class of a multiplicity-free $\Lambda$-module.} Clearly, such a formulation is
irritating, since the set of isomorphism classes of multiplicity-free modules
may outnumber the set of right equivalence classes of right minimal morphisms
by far: Consider just the special case of a representation finite artin algebra
of uncountable cardinality, then
there are only finitely many isomorphism classes of multiplicity-free modules,
but usually uncountably many right equivalence classes of right minimal morphisms.
So how should it be possible that finitely many modules $C$ determine uncountably
many morphisms $\alpha$? The solution is rather simple: It is not just the
module $C$ which is needed to recover a morphism $\alpha: X \to Y$ 
but one actually needs a submodule of $\Hom(C,Y),$ with $\Hom(C,Y)$
being considered as an $\End(C)^{\text{op}}$-module. In the setting where $\Lambda$
is representation-finite and uncountable, one should be aware that usually
the modules $\Hom(C,Y)$ will have uncountably many submodules, thus we are no longer
in trouble. So if we assert that the morphism 
$\alpha: X \to Y$ is determined by the module $C$, then 
one should keep in mind that $C$ is only part of the data which are required
to recover $\alpha$; in addition to $C$ one will need a submodule of $\Hom(C,Y)$.
	\medskip 

\noindent 
{\bf 20.5. Modules versus morphisms, again.} 
The following feature seems to be of interest: The concept of the determination of
morphisms by modules concerns the category of maps with fixed target $Y$, namely one
wants to decide whether two elements in $[\to Y\rangle$ are comparable. The theory
asserts that there is a test set of modules, namely the 
indecomposable direct summands of $C(f)$. For the testing procedure,
they are just modules, but any such object $L$ comes equipped with a non-zero 
(thus right minimal) map $L \to Y$. 
	\medskip 

\noindent 
{\bf 20.6. Logic and category theory.} Let us stress that the setting of the Auslander bijections is well-accepted both in mathematical logic and in category theory. The lattice of $\Lambda$-pointed modules, as studied in
model theory (see for example \cite{[P]})
is precisely the lattice $\langle\Lambda \to]$. In this way, all the
results concerning pp-formulae and pp-definable subgroups concern the
Auslander setting.

In the terminology of abstract category theory, 
we deal with a comma category, namely 
the category of {\it objects over $Y$}:
its objects are the maps $X \to Y$, a map from $f: X\to Y$ to $f': X'\to Y$ being given by a map $h: X\to X'$ such that $f = f'h.$ 

Of course, the use of the representable functors 
$\Hom(X,-)$ and $\Hom(-,Y)$ is standard  in representation theory.
To provide here all relevant references would overload our presentation
due to the abundance of such papers. So we restrict to mention only few
names: of course Auslander himself, 
but also Gabriel \cite{[G]} as well as Gelfand-Ponomarev.
	\medskip

\noindent 
{\bf 20.7. Generalisations.} In this survey we have restricted
the attention on the module category of an artin algebra $\Lambda$, or, what is the same,
a length $k$-category $\mathcal  C$, where $k$ is  a commutative artinian ring,
such that $\mathcal  C$ is $\Hom$- and $\Ext$-finite, has only finitely many simple
objects and all objects have bounded Loewy length. Auslander's investigations
were devoted to larger classes of rings, and many of our considerations should be
of interest in a broader context. 

Krause \cite{[K2]} has considered the general setting of dualising varieties
in the sense of Auslander and Reiten \cite{[AR]} and has obtained 
the precise analogues of the Auslander Theorems which form the basis of our survey.
In particular, he was looking also at triangulated categories and showed the relationship
to the existence of Serre duality. 
The interested reader may try to work out in which way our presentation 
has to be modified in this context. 
	\bigskip

\noindent 
{\bf Acknowledgment.} Some of the material presented here has been exhibited in 2012
in lectures at SJTU, Shanghai, at USTC, Hefei, and at the ICRA conference
at Bielefeld, in 2013 in lectures at Woods Hole and at NEU, Boston. 
The author is grateful to the audience for questions and
remarks. He also wants to thank Xiao-Wu Chen, Lutz Hille, Henning Krause and Idun Reiten
for helpful comments.

\hrule
\medskip

\noindent 2010 {\it Mathematics Subject Classification}:
Primary 
16G70, 
18E10, 
18A25, 
18A32, 
16G60. 
Secondary 
16G20, 
18A20, 
06C05, 
14M15, 
19A49, 
03C60. 
	\medskip

\noindent {\emph Keywords:} Auslander bijections, Auslander-Reiten theory.
Right factorization lattice. Morphisms determined by modules. 
Finite length categories: global directedness, local symmetries. 
Representation type, Brauer-Thrall conjectures. Riedtmann-Zwara degenerations.
Hammocks. Kronecker quiver. Quiver Grassmannians, Auslander varieties. 
Modular lattices, meet semi-lattices.


\begin{thebibliography}{99}

\bibitem{ARS} M. Auslander, I. Reiten, and S. Smal\o, {\em Representation Theory of
Artin Algebras,} Cambridge Studies in Advanced Mathematics, 1994.

\bibitem{[A1]}Auslander, M.: Functors and morphisms determined by objects. In: 
  Representation Theory of Algebras. Lecture Notes in Pure Appl. Math. 37.
  Marcel Dekker, New York (1978), 1-244. Also in: Selected Works of Maurice 
  Auslander, Amer. Math. Soc. (1999).
\bibitem{[A2]}Auslander, M.: Applications of morphisms determined by objects. In: 
  Representation Theory of Algebras. Lecture Notes in Pure Appl. Math. 37.
  Marcel Dekker, New York (1978), 245-327. Also in: Selected Works of
  Maurice Auslander, Amer. Math. Soc. (1999).
\bibitem{[AR]}Auslander, M., Reiten, I.: Stable equivalences of dualizing $R$-varieties.
  Advances in Math. 12 (1974), 306-366.
\bibitem{[ARS]}Auslander, M., Reiten, I., Smal\o, S.: 
  Representation Theory of     Artin 
  Algebras. Cambridge Studies in Advanced Mathematics 36. Cambridge University 
  Press. 1997.
\bibitem{[B]} Baez, J.: The n-category cafe: Quivering with Excitement. Blog  May 4, 2012. 
   http://golem.ph.utexas.edu/category/2012/05/ \newline 
quivering$\_$with$\_$excitement.html
\bibitem{[BGP]} Bernstein, I.N.,  Gelfand, I.M,  Ponomarev. V.A.:
  Coxeter functors and Gabriel's theorem. Uspekhi Mat. Nauk 28(1973),
  Russian Math. Surveys 28 (1973), 17-32.
\bibitem{[BH]} Bongartz, K., Huisgen-Zimmermann, B: 
   Varieties of uniserial representations IV.
  Kinship to geometric quotients
  Trans Amer. Math. Soc. 353 (2001), 2091-2113. 
\bibitem{[CC]} Caldero, P., Chapoton, F.: 
  Cluster algebras as Hall algebras of quiver representations,
  Comment. Math. Helv. 81 (2006), 595–616. 
\bibitem{[CB]} Crawley-Boevey, W.: 
  On homomorphisms from a fixed representation to a general representation of a quiver. 
  Trans. Amer. Math. Soc. 348 (1996), no. 5, 1909–1919. 
\bibitem{[D]} Dedekind, R.: \"Uber die von drei Moduln erzeugte Dualgruppe, 
  Math. Ann. 53 (1900), 371-403.)
\bibitem{[DHW]} Derksen, H., Huisgen-Zimmermann, B., Weymann, J.: 
  Top stable degenerations of finite
  dimensional representations II. (to appear). 
\bibitem{[DR]} Dlab, V, Ringel, C. M.: Perfect elements in the free modular
   lattices. Math. Ann. 247 (1980), 95-100. 
\bibitem{[G]} Gabriel, P.: Representations indecomposables. Sem. Bourbaki
  (1973-74). LNM 431, Springer 1975, 143-160.
\bibitem{[GP1]} Gelfand, I.M., Ponomarev, V.A.: Indecomposable representations 
  of the   Lorentz group, Russian Math. Surveys 23 (1968), 1–58.
\bibitem{[GP2]} Gelfand, I.M., Ponomarev, V.A.: Free modular lattices and their 
  representations. Russ. Math. Surveys 29 (1974), 1-56.
\bibitem{[GP3]} Gelfand, I.M., Ponomarev, V.A.: Lattices, representations and algebras
  connected with them I.
  Russ. Math. Surveys 31 (1976), 67-85.
\bibitem{[GP4]} Gelfand, I.M., Ponomarev, V.A.: Lattices, representations and algebras
  connected with them II.   Russ. Math. Surveys 32 (1977), 91-114.
\bibitem{[GP5]}  Gelfand, I.M., Ponomarev, V.A.: Representations of graphs. 
  Perfect subrepresentations. Funkc. Anal. i Pril 13 (1980) 14-31. Funct. Anal. Appl 14 (1980), 177-190.
\bibitem{[Gr]} Gr\"atzer, G.: Lattice Theory: Foundation.
  Birkh\"auser (2011).
\bibitem{[Hi]} Hille, L.: Tilting line bundles and moduli of thin sincere representations
of quivers. An. St. Univ. Ovidius Constantza 4 (1996), 76-82.
\bibitem{[Hu]} Huisgen-Zimmermann, B.: The geometry of uniserial representations of
  finite dimensional algebras I. J. Pure Appl. Algebra 127 (1998), 39-72. 
\bibitem{[Kt]} Kraft, H.: Geometric methods in representation theory, 
 In:  Representations of Algebras, 
  Lecture Notes in Math. 944, Springer-Verlag, New York (1980), 180-257.
\bibitem{[K1]}Krause, H.: A short proof for Auslander's defect formula.
 Linear Algebra and its Applications. 365 (2003), 267 -- 270.
\bibitem{[K2]}Krause, H.: Morphisms determined by objects in triangulated categories. 
   arXiv:1110.5625.
\bibitem{[L]} Le Bruyn, L.: Quiver Grassmannins can be anything. Blog, May 2, 2012.
    http://www.neverendingbooks.org/quiver-grassmannians-can-be-anything.
\bibitem{[M]} MacLane, S.: Categories for the Working Mathematician. Springer, 1998
\bibitem{[P]}Prest, M.:
 Purity, Spectra and Localisation, Encyclopedia of Mathematics and 
 its Applications, Vol. 121, Cambridge University Press, 2009. 
\bibitem{[Re]} Reineke, M.: Every projective variety is a quiver Grassmannian.
   Algebras and Representation Theory (to appear). arXiv:1204.5730.
\bibitem{[R1]}Ringel, C. M.: Report on the Brauer-Thrall conjectures. 
  Proceedings ICRA 2.  Springer  LNM 831 (1980), 104-136. 
\bibitem{[R2]}Ringel, C. M.: The rational invariants of tame quivers. Invent.
    Math. 58 (1980), 217-239. 
\bibitem{[R3]} Ringel, C. M.: Four papers on problems in lniear algebra.
 In: Representation Theory. Selected Papers by Gelfand et al. London Math Soc. 
 Lecture Note Series 69 (1982), Cambridge University Press. 
\bibitem{[R4]}Ringel, C.M.:  Tame algebras and integral quadratic forms, 
  Lecture   Notes in Math.
  1099, Springer-Verlag, New York (1984).
\bibitem{[R5]} Ringel, C. M.: Hall polynomials for the representation-finite 
   hereditary algebras. 
   Adv. Math. 84 (1990), 137-178 
 \bibitem{[R6]}Ringel, C. M.: The Gabriel-Roiter measure. 
  Bull. Sci. math. 129 (2005), 726-748.
\bibitem{[R7]} Ringel, C. M.:
  Foundation of the representation theory of artin algebras, using the Gabriel-Roiter measure. 
  In: Trends in Representation Theory of Algebras and Related Topics. 
   (Edt. de la Pena and Bautista). 
    Contemporary Math. 406. Amer.Math.Soc. (2006), 105-135. 
\bibitem{[R8]} Ringel, C. M.:
  The first Brauer-Thrall conjecture. 
  In: Models, Modules and Abelian Groups. In Memory of A. L. S. Corner. 
  Walter de Gruyter,   Berlin (edt: B. Goldsmith, R.G\"obel) (2008), 369-374.
\bibitem{[R9]}Ringel, C. M.: Morphisms determined by objects:
  The case of modules over artin algebras.  Illinois Journal (to appear).
\bibitem{[R10]}Ringel, C. M.: Quiver Grassmannians and Auslander varieties 
 for wild algebras (in preparation). 
\bibitem{[RV]} Ringel, C. M., Vossieck, D.: Hammocks. 
   Proc. London Math. Soc. (3)    54 (1987), 216-246. 
\bibitem{[Ro]} Roiter, A. V.: Unboundedness of the dimensions of the 
   indecomposable representations of an algebra
   which has infinitely many indecomposable representations, 
   Izv. Akad. Nauk SSSR Ser.    Mat. 32 (1968), 1275–1282.
\bibitem{[Sc]} Schofield A.: General representations of quivers, 
   Proc. London Math. Soc. (3)    65 (1992).

\end{thebibliography}
\end{document}